\documentclass[11pt,twoside]{article}
\usepackage{latexsym}
\usepackage{amssymb,amsbsy,amsmath,amsfonts,amssymb,amscd}
\usepackage{mathrsfs}
\usepackage{epsfig, graphicx}
\usepackage{graphics,epsfig}
\usepackage{color}
\usepackage{enumerate}
\usepackage{cleveref}
\usepackage{comment}
\usepackage{epstopdf}
\usepackage{titlesec}
\usepackage[titletoc,toc,title]{appendix}
\usepackage{subfigure}
\usepackage[utf8]{inputenc}
\usepackage[T1]{fontenc}
\usepackage{lmodern}
\usepackage{graphicx}
\usepackage{caption}
\usepackage{floatrow}%
\usepackage{tikz}
\usetikzlibrary{arrows,positioning, arrows}

\usepackage{amsthm}
\usepackage{amsfonts}
\usepackage{graphicx}
\usepackage{caption}
\usepackage{floatrow}
\usepackage{amsmath}
\usepackage{amssymb}
\usepackage{amsmath}
\usepackage{bm}
\usepackage{subeqnarray}
\usepackage[linesnumbered,ruled]{algorithm2e}
\usepackage[noend]{algpseudocode}
\usepackage{cases}
\theoremstyle{plain}
\newtheorem{theorem}{Theorem}
\newtheorem{lemma}{Lemma}
\newtheorem{assumption}{Assumption}

\theoremstyle{definition} 

\newtheorem{remark}{Remark}
\newtheorem{example}{Example}

\textheight 9in \textwidth 6.5in \numberwithin{equation}{section}

\def\mE{{\mathcal E}}
\def\mX{{\mathcal X}}

\def\R{{\mathbb R}}

\def\mL{{\mathcal L}}
\def\bf{{\mathbf f}}

\def\gt{{\rightarrow}}

\def\1{{\mathbf 1}}
 
\newcommand{\eps}{\varepsilon}
\newcommand{\average}[1]{ \left\langle#1 \right\rangle}

\newcommand{\half}{\frac{1}{2}}
\newcommand{\RR}{\mathbb{R}}
\newcommand{\rd}{\mathrm{d}}

\newcommand{\Lop}{\mathcal{L}}

\newcommand{\tr}{\tilde{\rho}}

\newcommand{\bv}{{\boldsymbol{v}}}

\newcommand{\bx}{{\boldsymbol{x}}}
\newcommand{\bb}{{\boldsymbol{b}}}
\newcommand{\bn}{{\boldsymbol{n}}}
\newcommand{\bl}{{BL}}
\newcommand{\bs}{{\boldsymbol{S}}}

\newcommand{\norm}[1]{ \| #1 \|}
\newcommand{\order}{\mathcal{O}}

\title{Solving multiscale steady radiative transfer equation using neural networks with uniform stability}
\author{Yulong Lu \thanks{Department of Mathematics and Statistics, Lederle Graduate Research Tower, University of Massachusetts, 710 N. Pleasant Street, Amherst, MA 01003. (lu@math.umass.edu)} 
	\and  Li Wang\thanks{School of Mathematics, University of Minnesota, Twin cities, MN 55455. (wang8818@umn.edu)}
	\and  Wuzhe Xu \thanks{School of Mathematics, University of Minnesota, Twin cities, MN 55455. (xu000355@umn.edu)} }

\date{}

\begin{document}
	
	\maketitle
	
	\begin{abstract}
		This paper concerns solving the steady radiative transfer equation with diffusive scaling, using the physics informed neural networks (PINNs). The idea of PINNs is to minimize a least-square loss function, that consists of the residual from the governing equation, the mismatch from the boundary conditions, and other physical constraints such as conservation. It is advantageous of being flexible and easy to execute, and brings the potential for high dimensional problems. Nevertheless, due the presence of small scales, the vanilla PINNs can be extremely unstable for solving multiscale steady transfer equations. In this paper, we propose a new formulation of the loss based on the macro-micro decomposition. We prove that, the new loss function is uniformly stable with respect to the small Knudsen number in the sense that the $L^2$-error of the neural network solution is uniformly controlled by the loss. When the boundary condition is an-isotropic, a boundary layer emerges in the diffusion limit and therefore brings an additional difficulty in training the neural network. To resolve this issue, we include a boundary layer corrector that carries over the sharp transition part of the solution and leaves the rest easy to be approximated. The effectiveness of the new methodology is demonstrated in extensive numerical examples.  
	\end{abstract}
	
	
	
	\section{Introduction}
	Developing efficient and robust numerical scheme for multiscale kinetic equation has always been a challenging yet important subject of research, and has attracted a lot of attention in the past decade. The main difficulty comes from the stiffness raised by multiple scales of the equation, which generically requires fine spatial mesh grid and short time step to guarantee both accuracy and stability. A large number of numerical schemes has been devoted to relaxing such a requirement, in the traditional grid-based framework, including the finite difference method, finite volume method, discrete Galerkin method, and etc \cite{LM83, Jin12, DP14}. Recently, deep learning method has emerged as a competitive mesh-free method for solving partial differential equations (PDEs). The idea is to represent solutions of PDEs by (deep) neural networks to take advantage of the rich expressiveness of neural networks representation. The parameters of neural networks are chosen by training or optimizing some loss functions associated with the PDE. It is advantageous of being intuitive and easy to execute, and also offers an innovational approach for solving high dimensional problems. 
	
	Many deep learning methods, based on optimizing different loss functions, have been developed for solving PDEs. To the best of our knowledge, the first neural network method PDE solver dates back to \cite{lagaris1998artificial} and builds on minimizing the $L^2$-residual of the PDE and that of the boundary/initial conditions. The now-days popular physical informed neural network (PINN)  \cite{raissi2019physics} and deep Galerkin method (DGM) \cite{sirignano2018dgm} fail into the same residual minimization framework. Another method called Deep Ritz Method \cite{weinan2018deep} is designed to solve some PDE problems with variational structures by exploiting the Ritz formulation of PDEs. The deep BSDE method \cite{han2018solving} was developed for solving some parabolic PDEs based on the stochastic representation of the solutions. For discussions of other machine learning methods for PDEs, we refer the interested reader to the excellent review article \cite{han2020algorithms}. 
	
	Recently, several works \cite{chen2021solving, hwang2020trend,liu2021deep} proposed neural network methods for solving kinetic equations by employing the framework of PINNs. However, the error bounds proved in those works for vanilla PINNs deteriorate in the diffusive regime where the Knudsen number is small. More specifically, the stability estimates proved for the vanilla PINN loss functions blow up as the  Knudsen number tends to zero. The purpose of the present paper is to build a new loss function which satisfies a stability estimate that is uniform with respect to the Knudsen number in the diffusive regime. Consider the steady radiative transfer equation (RTE), which takes the following general form:  
	\begin{equation} \label{eqn:rte}
		\begin{cases}{}
			\displaystyle \eps \bv \cdot \nabla_{\bx} f(\bx, \bv) =  \sigma_s(\bx) \Lop f(\bx, \bv) -\eps^2 \sigma_a(\bx) f + \eps^2 G(\bx), (\bx,\bv) \in \Omega := \Omega_x \times  \bs^{d-1}\,,\\
			f(\bx, \bv) = \phi(\bx, \bv), (\bx,\bv) \in \Gamma_-.
		\end{cases}
	\end{equation}
	Here $f(t,\bx,\bv)$ is the distribution of particles at time $t$ and location $\bx$ with velocity $\bv$, $G(\bx)$ is source function and $\Gamma_- := \{(\bx,\bv) \in \partial \Omega_x \times \bs^{d-1} |\  \bv \cdot \bn_x < 0 \}$ is the inflow boundary. Assume that $\Omega_x$ is bounded and Lipschitz on $\R^d$. We also assume that the inflow boundary value $\phi(\bx, \bv) \in L^2(\Gamma_-)$. The parameter $\eps>0$, often termed as Knudsen number, is a dimensionless parameter that governs the regime of the equation. In particular, $\eps \sim \mathcal O (1)$ refers to kinetic regime, and $\eps \ll 1$ corresponds to the diffusive regime.    
	The scattering operator $\mL$ is defined by 
	\begin{equation*} 
		\Lop f = \frac{1}{|\bs^{d-1}|}\int_{\bs^{d-1}} K(\bv, \bv') ( f(\bv')-f(\bv)) \rd \bv',
	\end{equation*}
	where $K:\bs^{d-1}\times \bs^{d-1} \gt \R$ is a nonnegative kernel. 
	The functions $\sigma_s$ and $\sigma_a$ are the scattering coefficient and absorption coefficient respectively. In addition, we assume the following assumption  is valid. 
	
	\begin{assumption}\label{ass:sigma} There exist positive constants $\sigma_{\min}$ and $\sigma_{\max}$ such that 
		\begin{equation*} 
			0<\sigma_{\min} \leq \sigma_s(\bx) \leq \sigma_{\max}  \text{ and } 0\leq \sigma_a (\bx)\leq \sigma_{\max}.
		\end{equation*}
	\end{assumption}
	
	Throughout the paper, we also make the following assumption on the scattering operator $\Lop$, which will play an essential role in obtaining a stability estimate for our new PINN loss function.  
	
	\begin{assumption}\label{ass:L}
		The scattering operator $\mL$ satisfies 
		\begin{itemize}
			\item[1)] $\average{\Lop f} := \frac{1}{|\bs^{d-1}|} \int_{\bs^{d-1}} \Lop f \rd \bv = 0 $ for any $f(\bv) \in L^2(\bs^{d-1})$;
			\item[2)] the null space of $\Lop$ is $\mathcal N(\Lop) = \{f = \average{f}\}$;
			\item[3)] $\Lop$ is non-positive self-adjoint in $L^2(\bs^{d-1})$, and moreover, $\average{f \Lop f} \leq -c \average{f^2}$ for every $f \in \mathcal N^\perp (\Lop)$ and some constant $c>0$;
			\item[4)] $\Lop$ admits a pseudo-inverse from $\mathcal N^\perp(\Lop)$ to $\mathcal N^\perp(\Lop)$.
			\item[5)] There exists $C_K>0$ such that $\|\mL f\|_{L^2(\Omega)} \leq C_K \|f\|_{L^2(\Omega)}$.
		\end{itemize}
	\end{assumption}
	
	Under Assumption \ref{ass:sigma} and Assumption \ref{ass:L}, it is well-known that RTE \eqref{eqn:rte} is well-posedness as shown in the theorem below. To state the theorem, let us first define the function space $\mathcal{X}$ by setting
	$$
	\mathcal{X} := \{ f\in L^2(\Omega)\ |   \bv\cdot \nabla_{\bx}f \in L^2(\Omega)\}.
	$$
	
	\begin{theorem}\label{thm:wellpos}
		Suppose that Assumption \ref{ass:sigma}  and Assumption \ref{ass:L} hold. There exists a unique solution $f$ to \eqref{eqn:rte} such that $f\in \mX$ and 
		$$
		\|f\|_{L^2(\Omega)} +  \|\bv \cdot \nabla_{\bx} f\|_{L^2(\Omega)}  \leq C (\|\phi\|_{L^2(\Gamma_I)} + \|G\|_{L^2(\Omega_x)}),
		$$
		where the constant $C$ depends on $\sigma_a,\sigma_s$, $\Omega$ and $\eps$. 
	\end{theorem}
	\begin{proof}
		Thanks to \cite[Theorem 1.1]{egger2014lp}, problem \eqref{eqn:rte} has a unique solution in $L^2(\Omega)$. Moreover, 
		$$
		\|f\|_{L^2(\Omega)} \leq C(\|\phi\|_{L^2(\Gamma_I)} + \|G\|_{L^2(\Omega_x)}).
		$$ Notice that the stability bound in \cite[Theorem 1.1]{egger2014lp} is slightly stronger than the one stated above since the $L^2$-bound there is weighted against $\ell (\bx,\bv)$, which is the length of line
		segment through $\bx$ in direction $\bv$ completely contained in $\Omega_{\bx}$. The gradient bound on $\|\bv\cdot \nabla_{\bx} f\|$ follows directly by taking $L^2$-norm on both sides of \eqref{eqn:rte} and the estimate above. 
	\end{proof}

	In the diffusive regime ($\eps \ll 1$), problem \eqref{eqn:rte} is well approximated by the elliptic equation:
	\begin{equation*}
		\average{\bv \cdot \nabla_{\bx} \Lop^{-1} \left( \frac{1}{\sigma_s} \bv \cdot \nabla_{\bx} \rho_0 \right) } = -\sigma_a \rho_0 + G\,,\quad  \rho_0 \big|_{\partial \Omega} = \xi(x)\,,
	\end{equation*}
	where $\xi$ is obtained through the boundary layer analysis \cite{bardos1984diffusion}. In many applied problems, the magnitude of $\eps$ can vary significantly across different regions; in this case it is desirable to have a solver that can deal with both kinetic ($\eps \sim \mathcal O(1)$) and diffusion ($\eps \ll 1$) regimes. Methods that fulfill this task fall into two categories, the domain-decomposition method \cite{golse2003domain,LLS15} and the asymptotic preserving method \cite{jin2000uniformly, LM08, JTH09, BPR13, LW17, SJX17,  PCQL20, TWZ21}. The former one solves different equations in different regimes and constructs an interface condition to connect them, whereas the latter seeks a unified solver that works in both regimes and therefore avoids the complication in identifying the interface location and designing interface condition. For time dependent RTE, there has been a vast literature on developing asymptotic preserving methods \cite{jin2000uniformly}, with the focus on resolving the stability issue by way of an implicit-explicit time discretization. 
	
	For stationary problems, on the other hand, specific challenge arises due to the presence of boundary layer. In general, generic numerical method may induce a numerical boundary condition in the zero $\eps$ limit that does not match the theoretical boundary condition, and then introduces errors not only on the boundary but also inside the computational domain. To this end, several efforts have been made to incorporate part of the boundary layer information into the scheme. For instance, Klar \cite{klar1998asymptotic2} constructed a boundary condition for the diffusion equation by approximating a Milne problem. Han, Tang and Ying \cite{han2014two} developed a tailored finite volume scheme that is uniform accurate up to boundary by freezing the coefficient in each cell and use special solutions to the constant coefficient equation as local basis functions. Lemou and Mehats \cite{lemou2012micro} proposed a new macro-micro decomposition by choosing the macro part such that its incoming velocity moments coincide with that of the distribution function, and therefore directly injected the exact boundary condition into the macro-micro system. When the collision kernel $K$ is isotropic, boundary layer can also be resolved by the Chandrasekhar H-function \cite{manteuffel1999boundary, yang2006numerical, jin2000uniformly}. For general collision kernel, Li, Lu and Sun \cite{LLS15} have proposed a half space solver, which then leads to an interface condition to connect different regimes.

	The primary goal in this paper is to develop a neural network method  that is uniformly stable and accurate for solving  $\eqref{eqn:rte}$ in both the kinetic ($\eps \sim 1$) and the diffusive regimes ($\eps \ll 1$). It is important to emphasize that the vanilla PINN is not able to resolve the solution when $\eps \ll 1$. In fact, one can construct examples where the neural network solution differs much from the exact solution whereas the vanilla PINN loss is small; see Section \ref{sec:failure} for such an example. 
	To overcome the instability issue, we propose a new loss function to train the neural network using the idea of macro-micro decomposition that underlies many asymptotic preserving methods. We shall show later that under some assumptions the new loss satisfies a uniform stability estimate. As a illustration, let us compare in Figure~\ref{fig: pinn_rg_comp} the errors of solutions computed using two loss functions for the toy example in Section \ref{sec:failure}  (see equation \eqref{eqn:f_toy}). One sees that if the vanila PINN loss is used the relative $L^2$-error remains $O(1)$ even when the (empirical) PINN loss already decreases to below $10^{-7}$. Whereas our new PINN loss yields that the relative $L^2$-error decreases along with the decreasing empirical loss. 
	\begin{figure}[htbp]
		\centering
		{\includegraphics[width=0.4\textwidth]{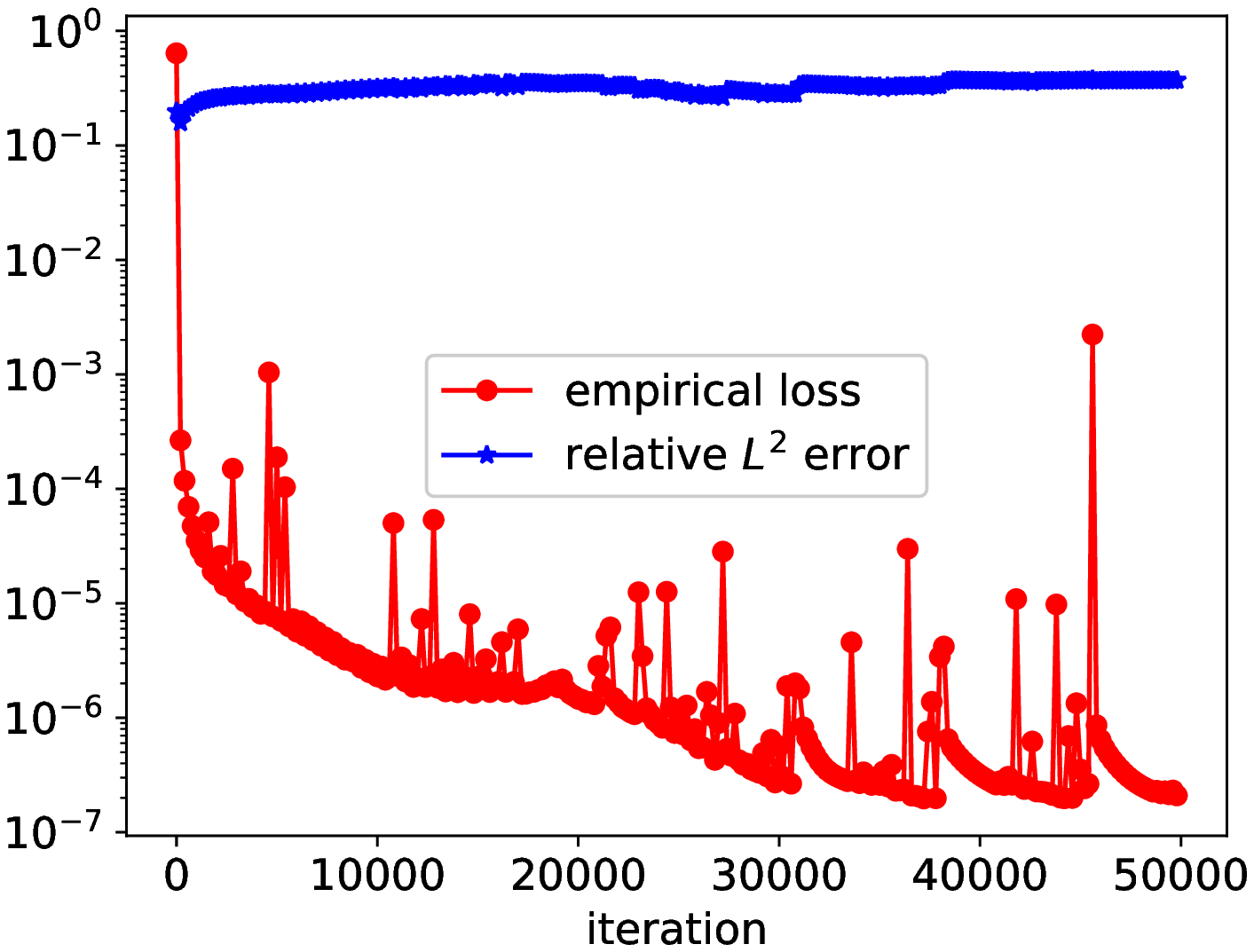}}
		{\includegraphics[width=0.4\textwidth]{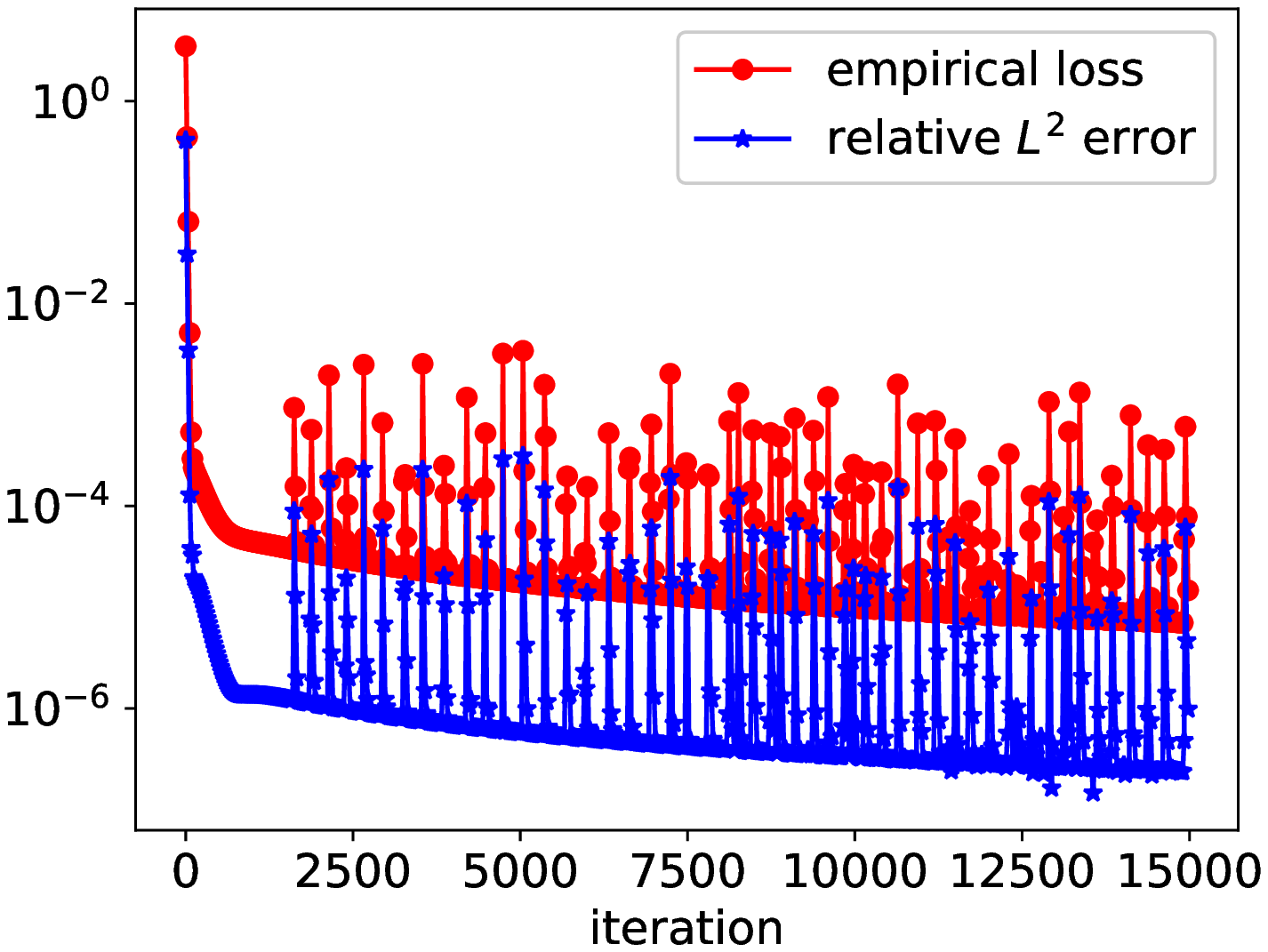}}
		\caption{Comparison of results solving $\eqref{eqn:f_toy}$ with $\eps = 10^{-3}$. Left is obtained with vanilla loss and right is with loss from macro-micro decomposition.}
		\label{fig: pinn_rg_comp}
	\end{figure}
	
	Another issue --- the boundary layer arises when the boundary data $\phi$ is variant in the $\bv$-direction. Its presence can significantly slow down the training of neural networks. To deal with this issue, we construct a boundary layer corrector that mitigates the sharp transition in the solution and therefore eases the training process significantly. In comparison with the recent work in the same vein \cite{lee2020model, liu2021deep}, our contributions are highlighted as follows:
	\begin{itemize}
		\item We design a new least-square-type loss function based on the idea of macro-micro decomposition of \eqref{eqn:rte} and prove that the new loss satisfies a stability estimate  that is uniform  with respect to small Knudsen number.

		\item When boundary layer is present in \eqref{eqn:rte}, we modify the macro-micro decomposition by incorporating a boundary layer corrector that can capture the sharp transition of the solution near the boundary.

		\item We demonstrate the accuracy and robustness of the proposed methodologies in a wide range of numerical experiments. 
	\end{itemize}
	
	Parallel to our work here, we would like to mention a recent manuscript \cite{jin2021asymptotic} that considers the time dependent case; it shares similar ideas of macro-micro decomposition, but with many details differently. 
	
	The rest of the paper is organized as follows. In the \cref{sec:diff_limit}, we recall a formal derivation of the diffusion limit of $\eqref{eqn:rte}$ and summarize the half space problem for boundary layer in multiple dimensions. In \cref{sec:pinn}, we first discuss the pitfalls of the vanilla PINN loss and then introduce new loss functions based on the  macro-micro decomposition (with and without boundary layer corrector). Theoretical stability estimates of the new loss function are proved in \cref{sec:mainthm}. Finally, we illustrate the accuracy and efficiency of our method by presenting several numerical examples in \cref{sec:num}.

	\section{The diffusion approximation for the radiative transfer equation}\label{sec:diff_limit}
	In this section, we collect some preliminary information regarding the diffusion approximation of the radiative transfer equation, both inside the domain and near the boundary. In particular, we have the following theorem.   
	
	\begin{theorem}\label{thm:diff_limit}
		Suppose $f$ solves \eqref{eqn:rte}. Then as $\eps \rightarrow 0$, $f(\bx, \bv)$ converges to $\rho_0(\bv)$, which solves 
		\begin{equation} \label{diff}
			\average{\bv \cdot \nabla_{\bx} \Lop^{-1} \left( \frac{1}{\sigma_s} \bv \cdot \nabla_{\bx} \rho_0 \right) } = -\sigma_a \rho_0 + G\,,
			\qquad \rho_0\big|_{\partial \Omega}  = \zeta(\bx)\,.
		\end{equation}
		Here $\zeta(\bx)$ at any point $\bx_\bb \in \partial_x \Omega$ is determined by
		\[
		\zeta(\bx_\bb) = \lim_{z \rightarrow \infty} f_\bl(z,\bv; \bx_\bb)\,,
		\]
		where $f_\bl(z,\bv; \bx_\bb)$ solves the half space problem: 
		\begin{equation} \label{eqn:rte_hsp}
			\displaystyle (-\bv \cdot \bn_{\bb}) \partial_z f_\bl  =  \Lop (f_\bl) \,,   \quad 
			f_\bl(0, \bv) = \phi(\bx_\bb, \bv), \quad  \bv \cdot \boldsymbol{n}_b < 0\,.
		\end{equation}
	\end{theorem}
	

	\begin{proof}
		Here we provide a formal derivation. Rigorous proof can be found in \cite{bardos1984diffusion}.
		Away from the boundary, consider the Hilbert expansion of $f(x,v)$:
		\[
		f(\bx,\bv) = f_0(\bx,\bv) + \eps f_1 (\bx,\bv) + \eps f_2(\bx,\bv)  + \cdots
		\]
		which inserting into \eqref{eqn:rte} leads to the following equations with like powers:
		\begin{align}
			& \mathcal{O}(1): \qquad \Lop (f_0) = 0\,, \label{eqn:o1}
			\\ & \mathcal{O}(\eps): \qquad  \bv \cdot \nabla_{\bx} f_0 = \sigma_s \Lop f_1\,, \label{eqn:o2}
			\\ & \mathcal{O} (\eps^2): \qquad \!\! \bv \cdot \nabla_{\bx} f_1 = \sigma_s \Lop f_2 - \sigma_a f_0 + G\,. \label{eqn:o3}
		\end{align}
		First \eqref{eqn:o1} implies $f_0(x,v) = \average{f_0} := \rho_0(x)$. From \eqref{eqn:o2}, according to the property of $\Lop$, since $\bv \cdot \nabla_{\bx} \rho_0 \in \mathcal N (\Lop)^\perp$, we can write $f_1 = \Lop^{-1} \left( \frac{1}{\sigma_s} \bv \cdot \nabla_{\bx} \rho_0 \right)$. Then plugging this relation to \eqref{eqn:o3} and taking average in $v$, one can show that $\rho_0$ satisfies the elliptic equation in \eqref{diff}. Consequently, we obtain that $f$ converges to $\rho_0$ as $\eps \rightarrow 0$, with $\rho_0$ solving \eqref{diff}.
		
		In general, the boundary condition for $\rho_0$ is different from $f$ due to the presence of boundary layer. Therefore, we need to conduct the matched asymptotic boundary layer analysis to obtain the correct boundary data for $\rho_0$. Our derivation follows \cite{BLP79}, see also \cite{klar1998asymptotic, klar1998asymptotic2}. For a given point $\bx_{\bb}$ on the boundary, i.e., $\bx_{\bb} \in \partial \Omega_x$, let $\bn_{\bb}$ be the outer normal direction at $\bx_\bb$, then we define locally a stretching variable $z = z(\bx; \bx_\bb) \in [0, \infty)$ in a way such that 
		\begin{equation} \label{xz}
			{  \sigma_s(\bx)} \bv \cdot \bn_\bb d_z = -\eps \bv \cdot \nabla_{\bx} \,, ~ d_z = \frac{\rd}{\rd z}\,.
		\end{equation}
		For instance, when $\sigma_s$ is independent of $x$, $z =\frac{(\bx_\bb - \bx) \cdot \bn_{\bb}}{\eps}$; when $d=1$, at the left boundary $x=x_L$, ${  z = \frac{1}{\eps} \int_{x_L}^x \sigma_s(t) \rd t}$. It is obvious that when $\bx = \bx_\bb$, $z=0$; when $\bx$ is away from $\bx_\bb$, $z \rightarrow \infty$ as $\eps \rightarrow 0$. Then along $z$ direction, \eqref{eqn:rte} reads, in the leading order of $\eps$,
		\begin{equation*} 
			\begin{cases}{}
				\displaystyle (-\bv \cdot \bn_{\bb}) \partial_z f  =  \Lop (f)   \,, \\
				f(0, \bv) = \phi(\bx_\bb, \bv), ~  \bv \cdot \boldsymbol{n}_b < 0\,.
			\end{cases}
		\end{equation*}
		The well-posedness of the above problem can be found in \cite{golse2003domain}. In particular, if $\phi(\bx_\bb, \bv) \in L^2(\boldsymbol{S}^{-}_{\bn_{\bb}}, |\bv \cdot \bn_{\bb}| \rd \bv)$, where $\boldsymbol{S}^{-}_{\bn_{\bb}} = \{  \bv \in \boldsymbol{S}^{d-1} | \bv \cdot \bn_{\bb} <0   \} $, $\eqref{eqn:rte_hsp}$ has a unique solution in $L^{\infty}(\RR_{+}; L^2(\boldsymbol{S}^{-}_{\bn_{\bb}}, |\bv \cdot \bn_{\bb}| \rd \bv))$. In addition, denote its solution as $f_\bl(z, \bv; \bx_\bb)$, then it can be shown that 
		\begin{equation}\label{eqn0914}
			f_\bl (z,\bv; \bx_\bb) \rightarrow f_\bl^\infty(\bx_\bb)\,, \quad {as}~ z\rightarrow \infty\,,
		\end{equation}
		where ${  f_\bl^\infty(\bx_\bb)}$ is a function independent of $\bv$, which gives the condition for $\rho_0$ at $\bx_\bb$, i.e., $\zeta(\bx_\bb) = f_\bl^\infty(\bx_\bb)$. The same procedure can be carried out at each point on the boundary $\partial \Omega_x$ and we therefore obtain the boundary condition $\zeta[\phi]$ for $\rho_0$.
	\end{proof}
	
	Note that, when the collision kernel is isotropic, i.e., $K(\bv, \bv') = 1$, there is an explicit relationship between $f_\bl^\infty(\bx_\bb)$ and $\phi(\bx_\bb, \bv)$, through the Chandrasekhar's H-function, see Section 2 (for dimension one) and Appendix B (for multi dimension) in \cite{golse2003domain}.

	Additionally, when the boundary condition is independent of $\bv$, that is,  $\phi(\bx,\bv) = \phi(\bx)$, $\bx \in \partial \Omega_x$, then there is no boundary layer and hence $\zeta(\bx) = \phi(\bx)$. This is seen from the fact that $f_\bl(z,\bv; \bx_\bb)\equiv \phi(\bx_\bb)$ is a solution to \eqref{eqn:rte_hsp}.

	\section{Approximation by Physics Informed Neural Networks (PINNs)}\label{sec:pinn}
	We aim to approximate the solution of \eqref{eqn:rte} with functions that are parameterized by  neural networks. Denote $f^{nn}(\theta;\bx,\bv)$ the neural network function where $\theta$ represents the set of neural network parameters including weights and biases in the neurons. In the framework of PINNs and other neural network-based approaches, one seeks the approximation $f^{nn}(\theta;\bx,\bv)$ by minimizing a loss function that is defined by the PDE problem. The vanilla PINN (population) loss  is defined as the sum of the $L^2$-misfit of the PDE and that of the boundary values: 
	\begin{equation}\label{eq:vanillapinn}
		\mE_0 (f) = \int_{\Omega} \Big|\eps \bv \cdot \nabla_{\bx} f(\bx, \bv) -  \sigma_s(\bx) \Lop f(\bx, \bv) +\eps^2 \sigma_a(\bx) f - \eps^2 G(\bx) \Big|^2 d\bx\, d\bv + \int_{\Gamma_-} |f - \phi|^2 d\bx\,d\bv\,.
	\end{equation}
	In practice, we need to approximate the integrations above and this leads to the definition of the empirical PINN loss:
	\begin{equation}\label{eq:vanillapinn2}\begin{aligned}
			\mE^N_0(f) & = \sum_{i=1}^{N_x^r} \sum_{j=1}^{N_v^r}  \Big|\eps \bv \cdot \nabla_{\bx} f(\bx_{i}^r, \bv_{j}^r) -  \sigma_s(\bx_{i}^r) \Lop f(\bx_{i}^r, \bv_{j}^r) +\eps^2 \sigma_a(\bx_{i}^r) f(\bx_{i}^r,\bv_{j}^r) - \eps^2 G(\bx_{i}^r) \Big|^2 w_{ij}^q \\
			& \qquad + \sum_{m=1}^{N_x^b} \sum_{n=1}^{N_v^b} |f(\bx_{i}^b, \bv_{j}^b)-\phi(\bx_{i}^b, \bv_{j}^b)|^2 w_{ij}^b. 
	\end{aligned}\end{equation}
	Here $\{ \bx^r_i  \}_{i=1}^{N^r_x}$, $\{ \bv^r_j  \}_{j=1}^{N^r_v}$ and  $\{ w^r_j  \}_{j=1}^{N^r_v}$, and  $\{ \bx^b_m  \}_{m=1}^{N^b_x}, \{ \bv^b_n  \}_{n=1}^{N^b_v}$ and $ \{ w^b_j  \}_{j=1}^{N^b_v}$ are the interior and boundary quadrature points and weights, respectively. See concrete choices of quadrature points and weights in Section \ref{sec:num}.  Then a neural network approximation $f^{nn}(\theta;\bx,\bv)$  is obtained by solving the minimization problem:
	$$
	\min_\theta \mE^N_0(f^{nn}(\theta;\bx,\bv)). 
	$$



	\subsection{Pitfall of vanilla PINN  with small $\eps$}\label{sec:failure} 
	In this section, we would like to point out one pitfall of vanilla PINN loss \eqref{eq:vanillapinn} (or \eqref{eq:vanillapinn2}) in the case where the Knudsen number $\eps$ is small; this is illustrated in the following simple example. Consider the one dimensional boundary value problem:
	\begin{equation}\label{eqn:f_toy}
		\begin{cases}{}
			\eps v \partial_x f = \langle f \rangle - f - \eps v \, \quad x \in [0,1], v \in [-1,1]\,,\\
			f(0, v>0)= 1, \quad 
			f(1, v<0)= 0\,,
		\end{cases}
	\end{equation}
	whose analytic solution is given by $f^\ast(x,v) = 1-x$. Its vanilla PINN loss takes the form 
	\begin{equation} \label{PINN00}
		\mE_v(f) = \| \eps v \partial_x f - \average{f} + f + \eps v \|_{L^2(\Omega)}^2 + \| f(0,\cdot)-1\|^2_{L^2([0,1])} + \norm{ f(1,\cdot)}^2_{L^2([-1,0])}  \,.
	\end{equation}
	Then it is obvious that when $\eps \ll 1$, any $v$-independent function $f$ that satisfies the boundary condition, such as $f(x,v)=(1-x)^2$, and $(1-x)^3$ leads to $\mE(f) = \order(\eps^2)$. However, for those $f$ we have $\|f -f^\ast\|_{L^2(\Omega)} = \order(1)$. This shows that the vanilla PINN loss does not provide a good error indicator for the kinetic equation \eqref{eqn:rte} when $\eps$ is small. Consequently, training neural networks with the vanilla PINN loss can potentially lead to inaccurate estimation of the solution; see Figures \ref{fig: f_toy_epsi1}. 
	Here we use a fully connected neural network with 4 hidden layers and 50 neurons within each hidden layer. In computing \eqref{eq:vanillapinn2}, the parameters are chosen as $N^b_v = 60$, $N^r_x = 80$ and $N^r_v=60$.
	In fact, 
	when $\eps = 1$, as shown in Figure~\ref{fig: f_toy_epsi1}, PINN is able to give an accurate approximation to the analytic solution. However, when $\eps = 10^{-3}$, we observed in Figure~\ref{fig: f_toy_epsi1} that even the empirical loss decreases to as small as $10^{-8}$, the prediction is far away from the analytic solution. This observation motivates us to consider the macro-micro decomposition technique, which has become a standard numerical technique in solving multiscale problems.   
	
	\begin{figure}[htbp]
		\centering
		{\includegraphics[width=0.3\textwidth]{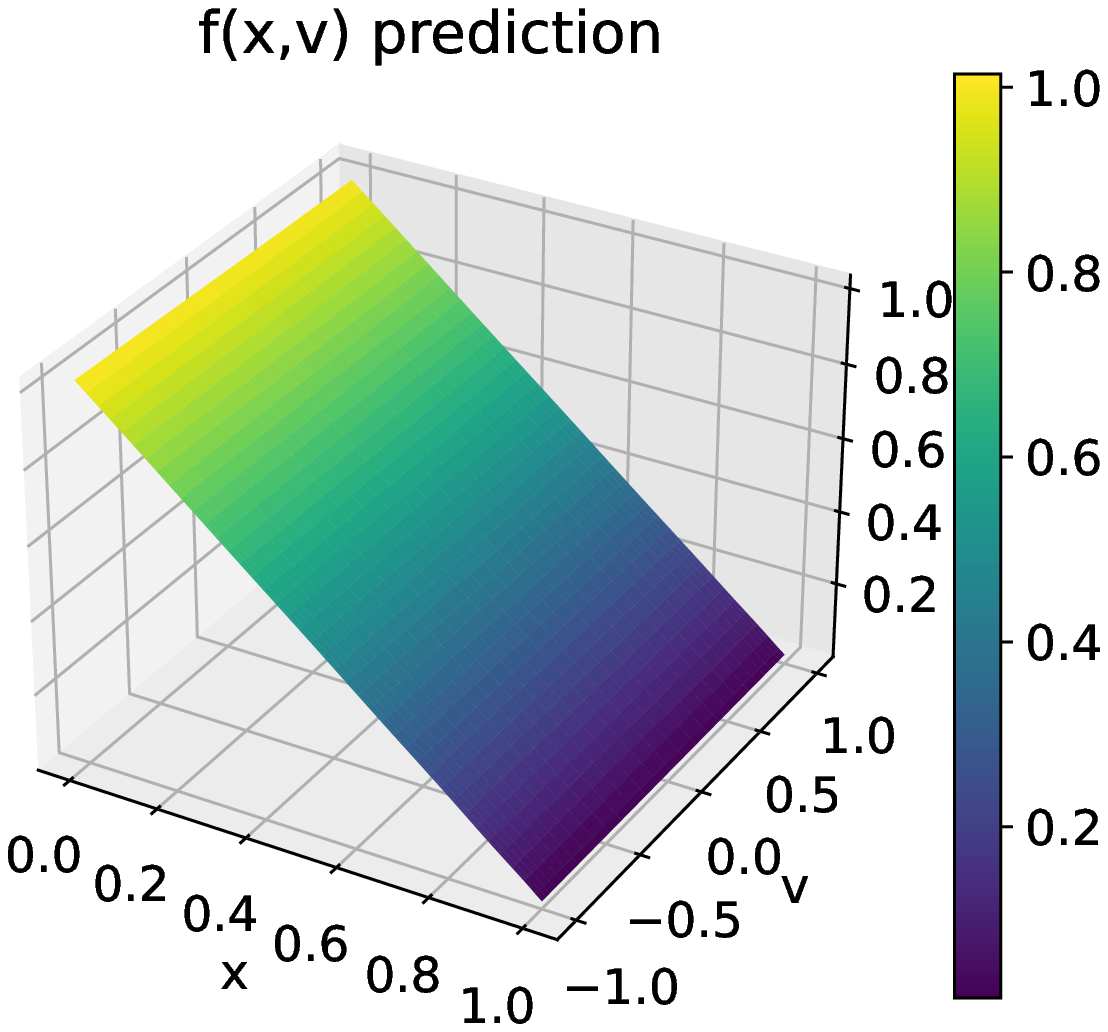}}
		{\includegraphics[width=0.3\textwidth]{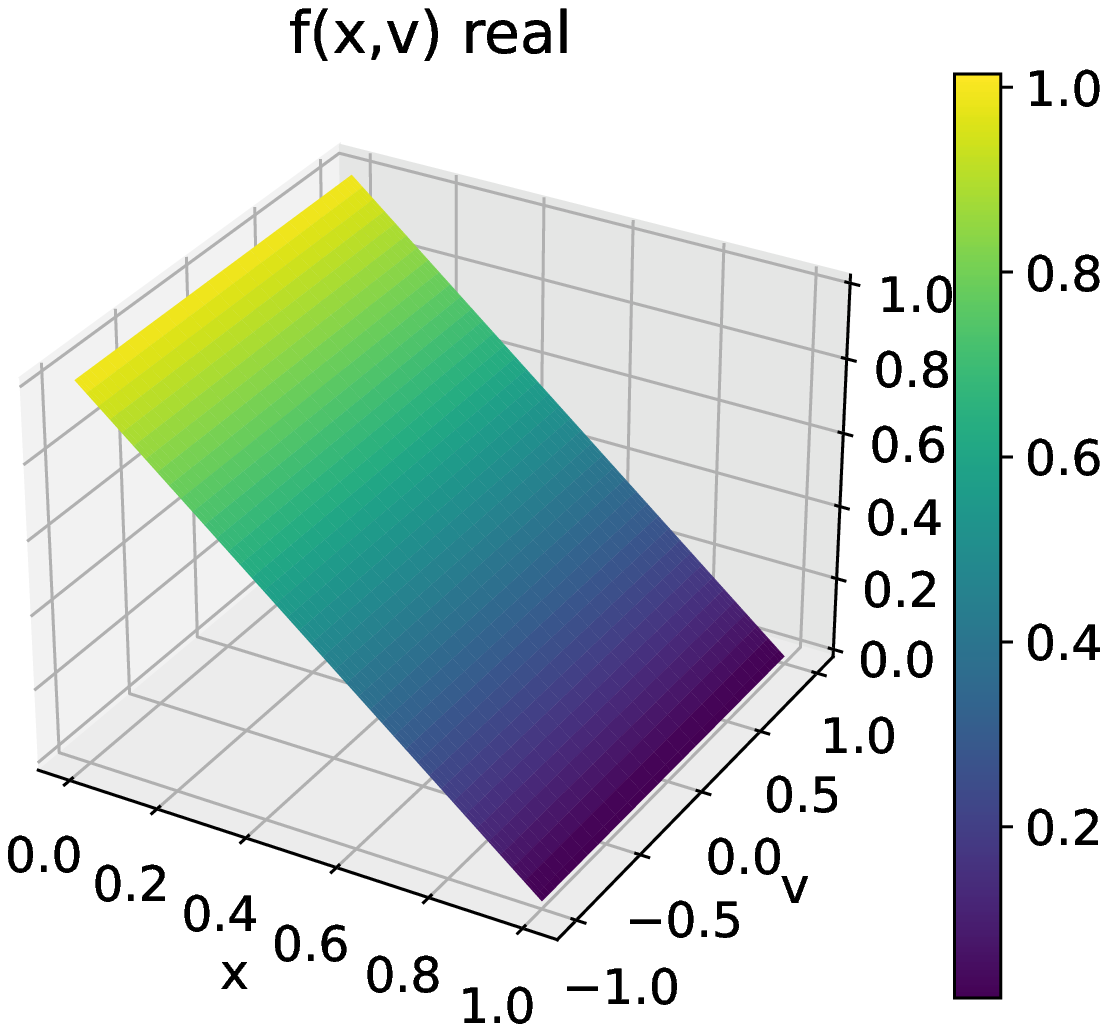}}
		{\includegraphics[width=0.3\textwidth]{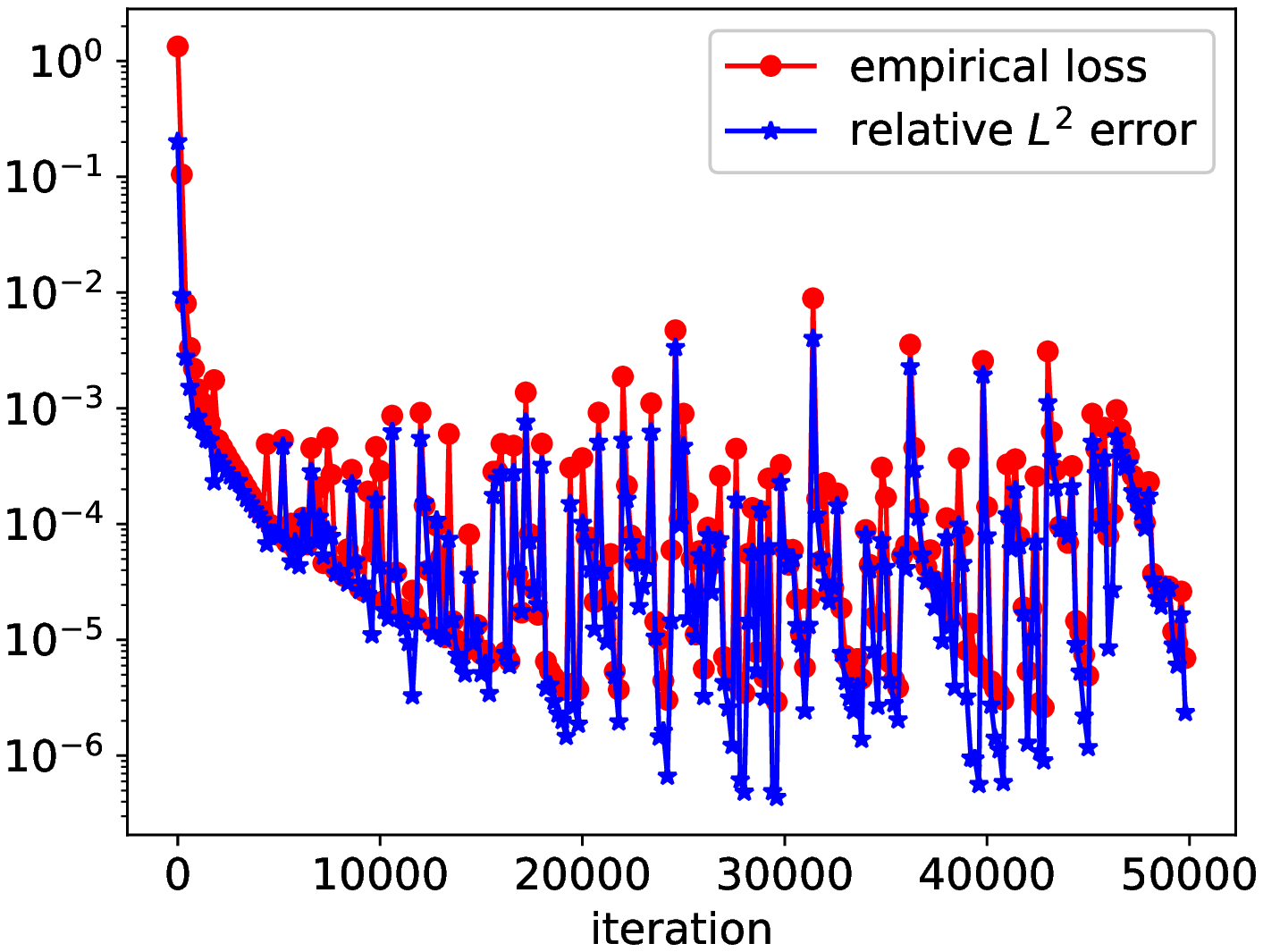}}
		{\includegraphics[width=0.3\textwidth]{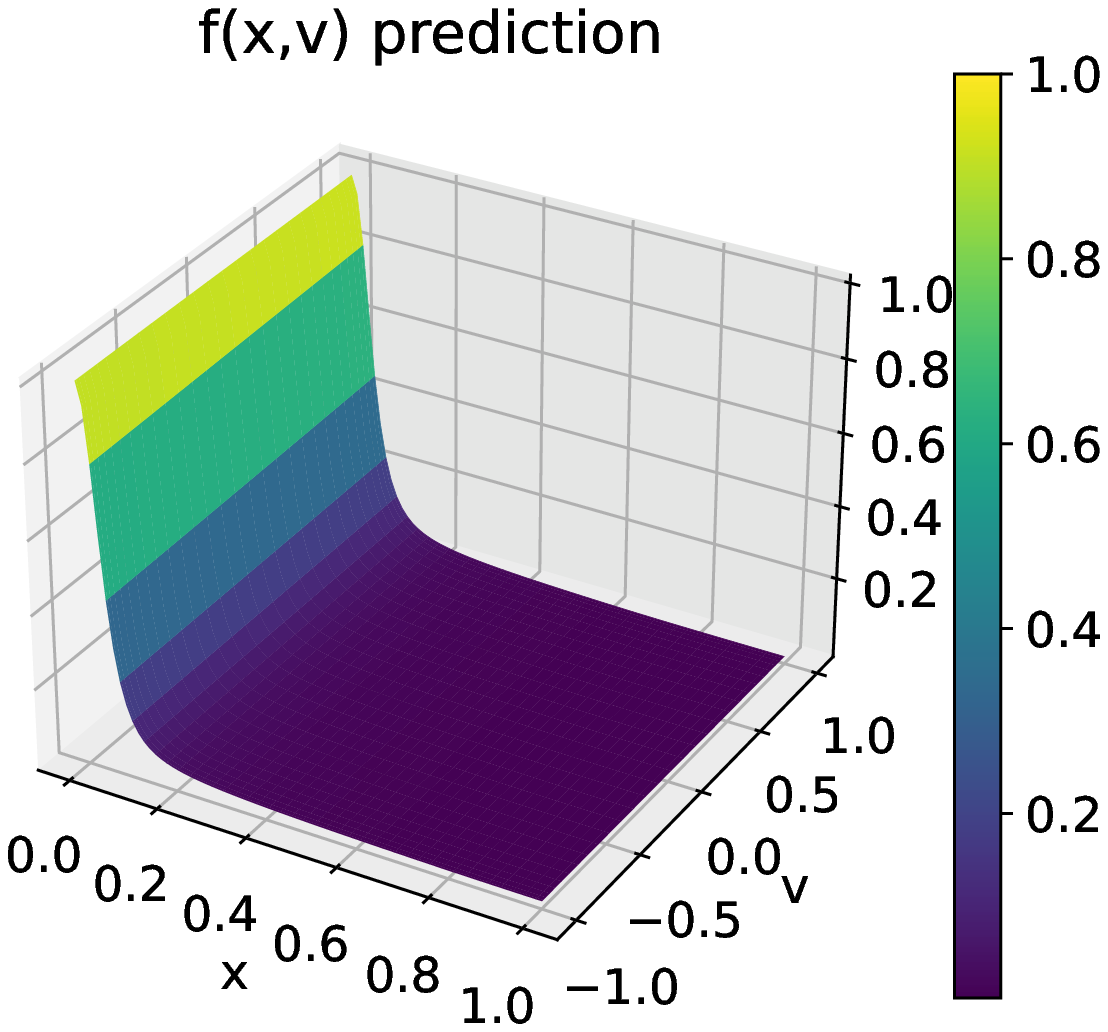}}
		{\includegraphics[width=0.3\textwidth]{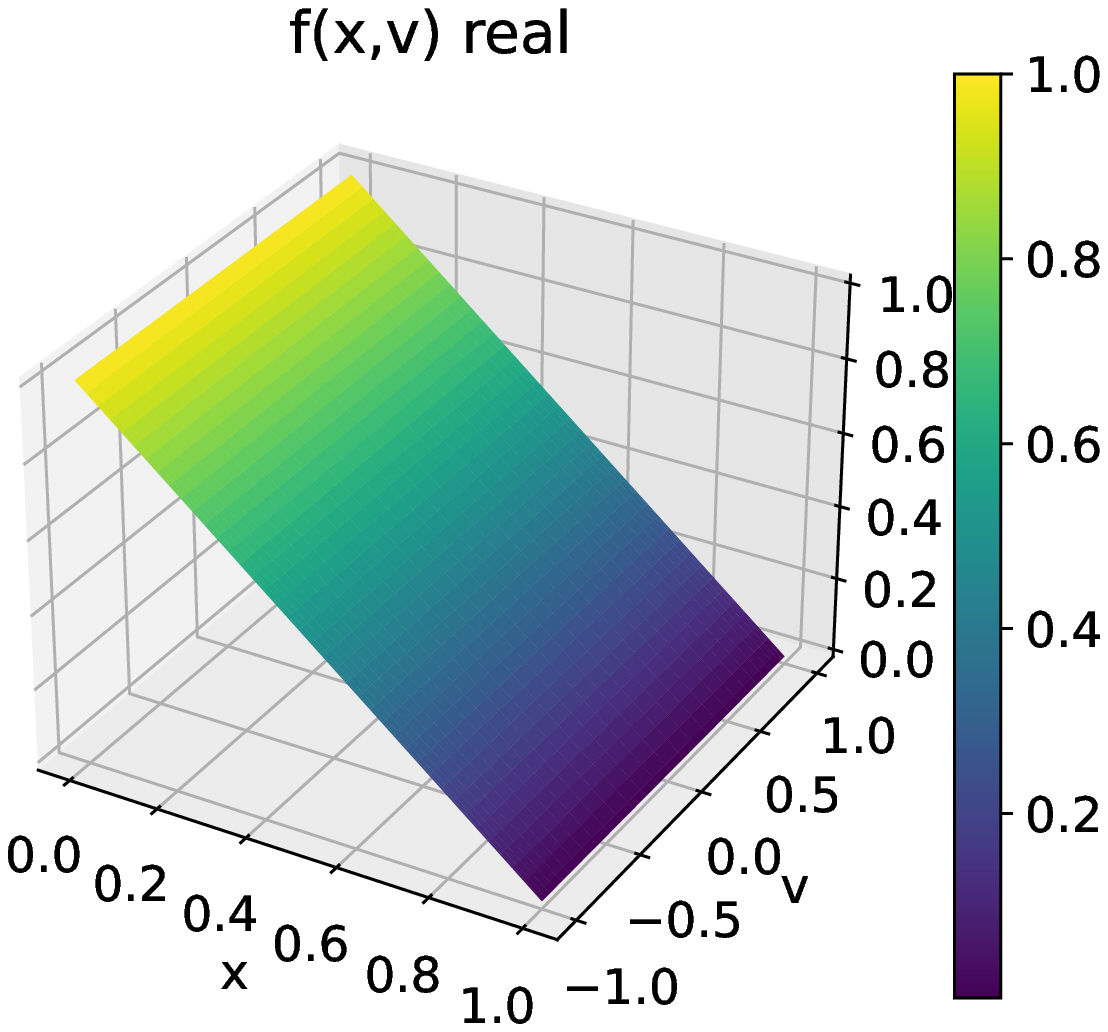}}
		{\includegraphics[width=0.3\textwidth]{test_f_emp_epsi_zpzz1_w1_error_re.eps}}
		\caption{Computation of \eqref{eqn:f_toy} with vanilla loss \eqref{PINN00}. The top row is for $\eps = 1$ and bottom row is for $\eps =0.001$. Left column is the predicted $f(x,v)$ from neural network, middle column is the analytic solution, and the right column is the error versus iterations with Adam optimizer.}
		\label{fig: f_toy_epsi1}
	\end{figure}


	\subsection{PINN based on macro-micro decomposition}\label{sec: PINN_decomp}
	In order to resolve the issue mentioned in the last section, we propose a new loss function based on a macro-micro decomposition.
	Write 
	\begin{equation} \label{decompf0}
		f = \rho(\bx) + \eps g(\bx, \bv), \quad \text{with}~ \rho = \average{f}, ~ \average{g} = 0\,,
	\end{equation}
	then \eqref{eqn:rte} can be decomposed into
	\begin{equation}\label{eq:mmd}
		\begin{cases}{}
			\average{\bv \cdot \nabla_{\bx} g} =  -\sigma_a \rho + G  , 
			\\ \bv \cdot \nabla_{\bx} (\rho + \eps g) - \eps \average{\bv\cdot \nabla_{\bx} g} = \sigma_s \Lop g - \eps^2 \sigma_a g  , 
			\\ \rho + \eps g \big|_{\Gamma_-} = \phi \,.
		\end{cases}
	\end{equation}
	Instead of \eqref{eq:vanillapinn}, we propose the following loss function:
	\begin{equation}\label{eq:loss0}
		\begin{aligned}
			\mE(f) & := \mE(\rho, g)\\
			& = \|\average{\bv \cdot \nabla_{\bx} g} + \sigma_a \rho \!-\! G\|_{L^2(\Omega)}^2 + \|\average{g}\|_{L^2(\Omega)}^2+ \|\rho+\eps g - \phi\|_{L^2(\Gamma_-)}^2
			\\ & \quad  + \|\bv \cdot \nabla_{\bx} (\rho + \eps g)-\eps\average{ \bv\cdot\nabla_{\bx} g}  \!-\! \sigma_s \Lop g  +\eps^2 \sigma_a g \!-\! \eps G\|_{L^2(\Omega)}^2 \,.
		\end{aligned}
	\end{equation}
	
	Now let us revisit the example $\eqref{eqn:f_toy}$. Applying the decomposition \eqref{eq:mmd} leads to 
	\begin{equation} \label{eqn:rg_toy}
		\begin{cases}{}
			\langle  v \partial_x g \rangle = 0 \,, \\
			v \partial_x (\rho + \eps g) = -g - v \,, \\
			\rho(0) + \eps g(0, v>0)= 1, ~ \rho(1) + \eps g(1, v<0) = 0 \,,
		\end{cases}
	\end{equation}
	which gives  the following loss function 
	\begin{equation}\label{eqn0902}
		\begin{aligned}
			\mE(f) &  = \|\average{v\partial_x g} \|_{L^2(\Omega)}^2 + \|v\partial_x (\rho + \eps g)+ g + v\|_{L^2(\Omega)}^2 \\
			& \quad + \int_0^1(\rho (0)+\eps g(0,v) - 1)^2 \rd v 
			+ \int_{-1}^0(\rho (1)+\eps g(1,v))^2 \rd v.
		\end{aligned}
	\end{equation}
	Note that we have eliminated the term $\|\average{g}\|_{L^2(\Omega)}^2 $ as $\average{g}=0$ is guaranteed by the second equation in \eqref{eqn:rg_toy}. With the new loss function \eqref{eqn0902}, we get a good approximation to the analytic solution for $\eps=10^{-3}$, see Figure~\ref{fig: rg_toy_epsizpzz1}.
	\begin{figure}[htbp]
		\centering
		{\includegraphics[width=0.3\textwidth]{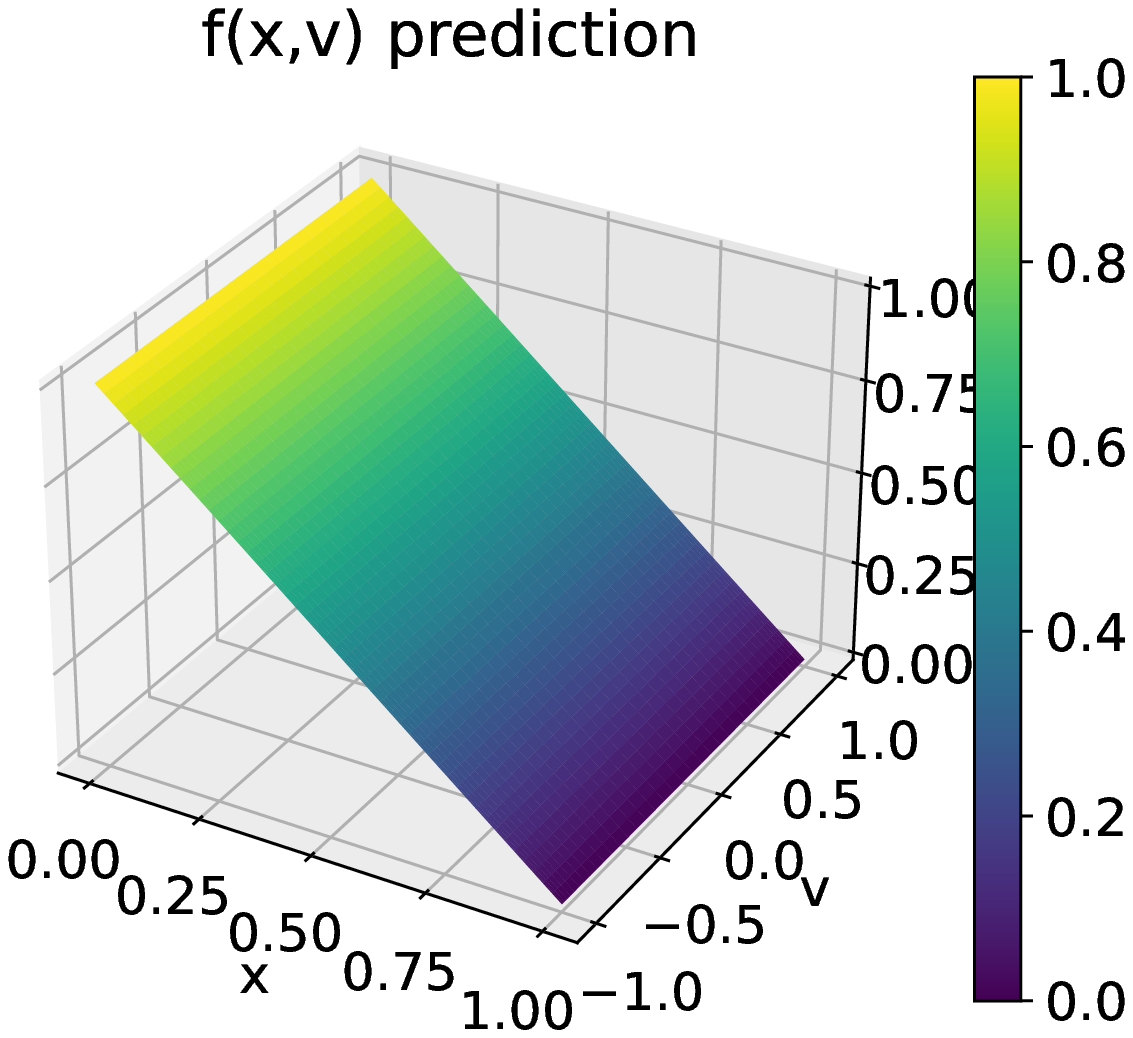}}
		{\includegraphics[width=0.3\textwidth]{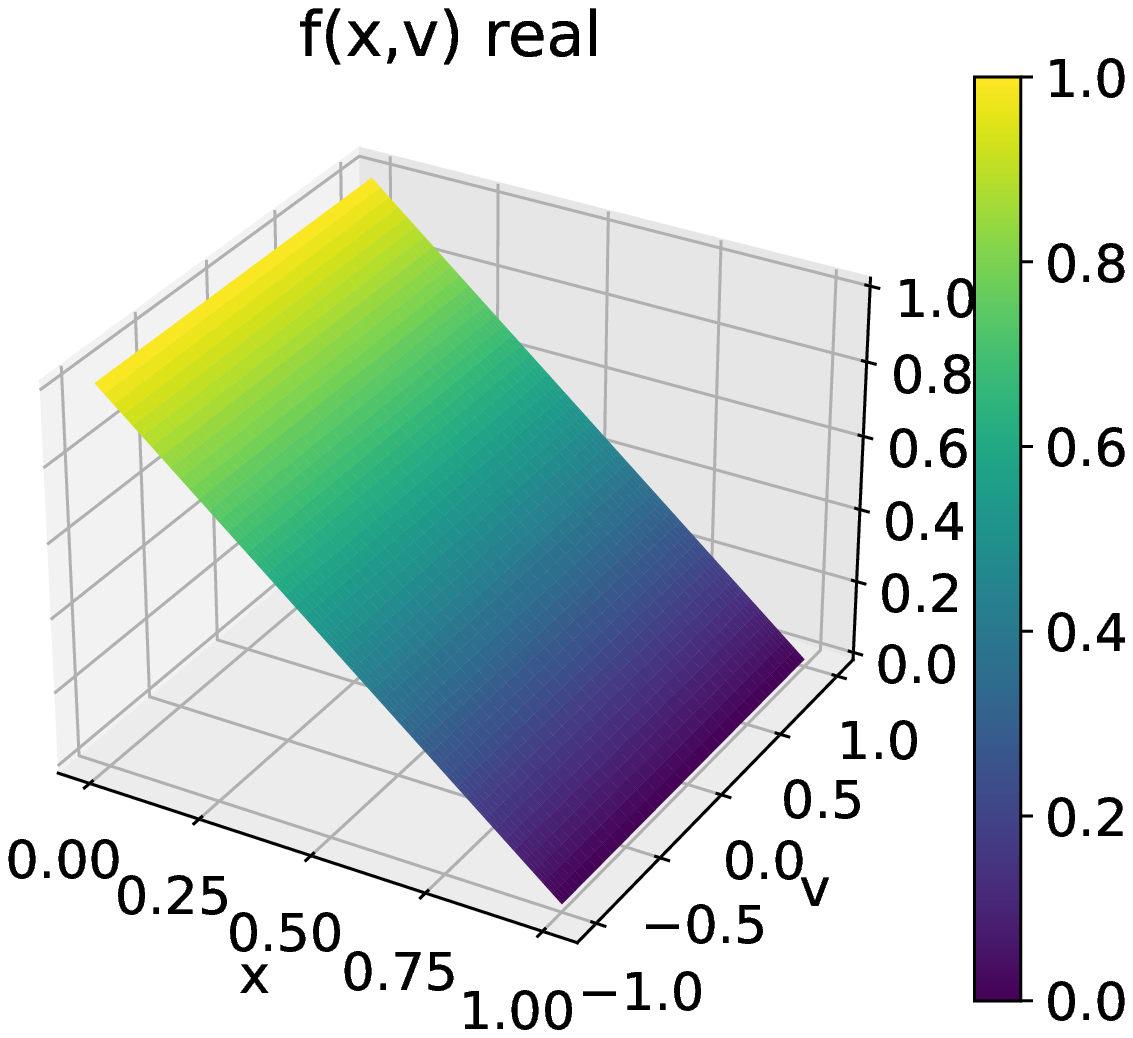}}
		{\includegraphics[width=0.3\textwidth]{test_rg_emp_re.eps}}
		\caption{Computation of \eqref{eqn:f_toy} with new loss \eqref{eqn0902} for $\eps = 10^{-3}$. 
			The left figure is the predicted $f(x,v)$ from neural network, the middle is the analytic solution, and the right is the error versus iterations with Adam optimizer.}
		\label{fig: rg_toy_epsizpzz1}
	\end{figure}

	\subsection{Boundary layer corrector}
	The example \eqref{eqn:f_toy} we have mentioned thus far is with homogeneous in $\bv$ boundary condition, and the PINN loss \eqref{eq:loss0} works just fine. However, when boundary value $\phi$ depends on $\bv$, the boundary layer will arise, which brings in additional challenge as one needs to approximate a fast varying function. 
	
	We illustrate this difficulty through an example. Consider
	\begin{equation} \label{eqn:f_toy_2}
		\begin{cases}{}
			\eps v \partial_x f = \langle f \rangle - f \,, \\
			f(0, v>0)= 5 \sin (v), \quad 
			f(1, v<0)= 0
		\end{cases}
	\end{equation}
	with $\eps = 10^{-3}$.
	Its macro-micro decomposition has the form
	\begin{equation} \label{eqn:rg_toy_2}
		\begin{cases}{}
			\langle v \partial_x g \rangle = 0 \,, \\
			v \partial_x (\rho + \eps g) = -g \,, \\
			\rho(0) + \eps g(0, v>0)= 5 \sin (v), \quad \rho(1) + \eps g(1, v<0) = 0\,.
		\end{cases}
	\end{equation}
	On the left boundary at $x=0$, one sees that, $\rho(0)$ takes a value independent of $v$, and leaves $\eps g(0, v>0)$ of $\mathcal{O}(1)$ magnitude, and therefore $g(0,v>0)$ is of $\mathcal{O}(1/\eps)$ magnitude. However, inside the domain, $g$ is of order $\mathcal{O}(1)$ according to the second equation in $\eqref{eqn:rg_toy_2}$, thus a sharp transition on $g$ at the left boundary is expected.

	To see how such a sharp transition affects the neural network approximation, we now apply the loss function \eqref{eq:loss0} to \eqref{eqn:rg_toy_2} to get
	\begin{equation}\label{eqn0913}
		\begin{aligned}
			\mE(f) &  = \|\average{v\partial_x g} \|_{L^2(\Omega)}^2 + \|v \partial_x (\rho + \eps g)+g \|_{L^2(\Omega)}^2\\
			& \qquad + B_{w,0} \int_0^1 (\rho(0)+\eps g(0, v) - 5 \sin(v))^2 \rd v + B_{w,1} \int_{-1}^0 (\rho(1) + \eps g(1,v))^2 \rd v\,,
		\end{aligned}
	\end{equation}
	and collect the results in Figure~\ref{fig: rg_toy_2_bl_w1} with $B_{w,0} = B_{w,1} = 1$. 
	As displayed, the empirical loss remains large even after 50,000 iterations, and the $f$ prediction is still far off the reference solution, which is plotted in Figure~\ref{fig: hsp_1d_bl_ref} with finite difference method on a non-uniform mesh.
	\begin{figure}[htbp]
		\centering
		{\includegraphics[width=0.4\textwidth]{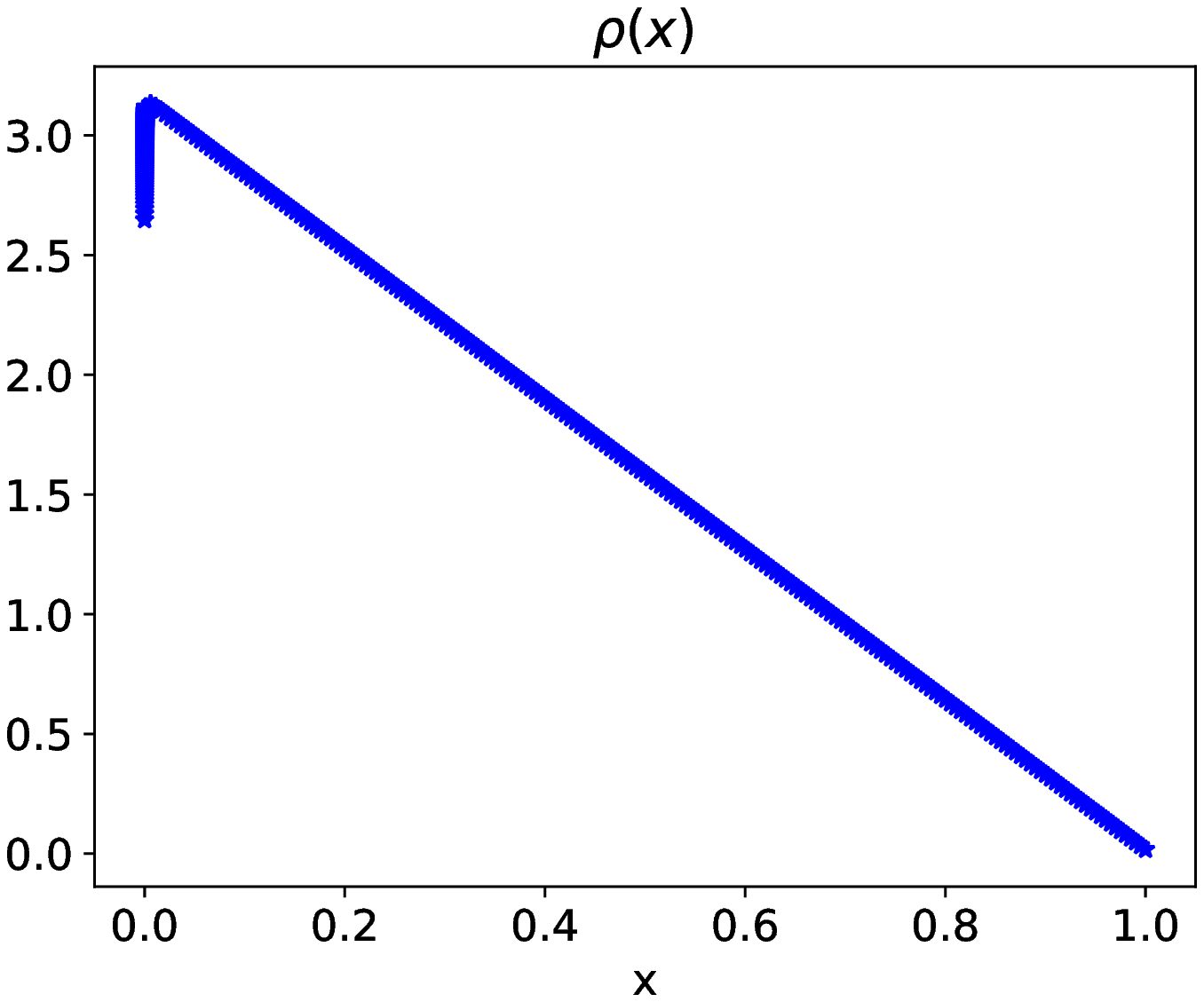}}
		{\includegraphics[width=0.4\textwidth]{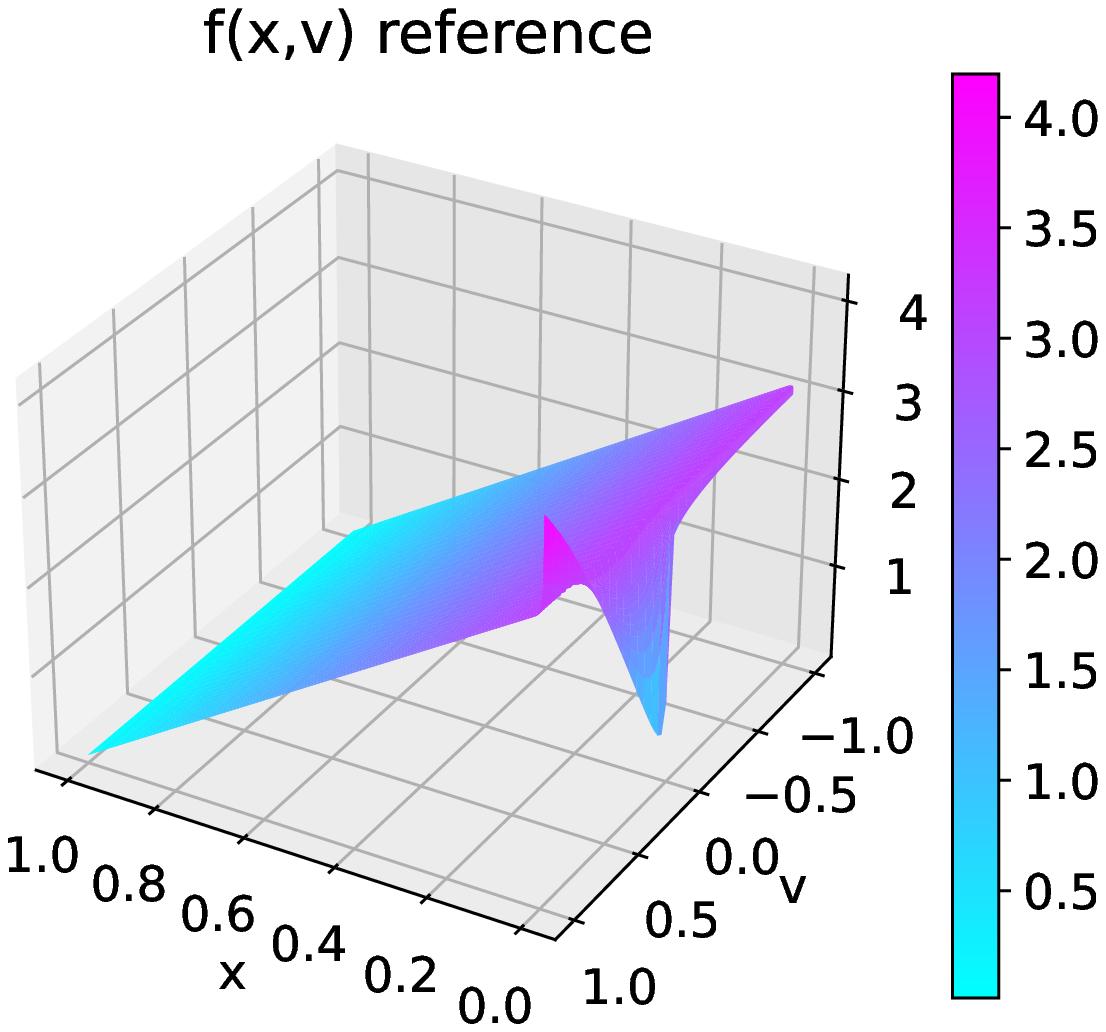}}
		\caption{Reference solution of \eqref{eqn:f_toy_2} computed by a finite difference method with nonuniform mesh.}
	\label{fig: hsp_1d_bl_ref}
\end{figure}

We noticed that, the dominated loss that hinders the convergence is the boundary loss in \eqref{eqn0913}, due to the presence of boundary layer, which is intrinsically harder to approximate. Therefore, we tried to put more emphasize on the boundary term by increasing the weight $B_{w,0}$ from $1$ to $1/\eps = 10^3$, but the result is unfortunately barely improved, see the left plot of Figure~\ref{fig: rg_toy_2_bl_r2}. A more sophisticated dynamics re-weighting \cite{wang2020and, wang2020understanding} might improve the performance, but adjusting the weight appropriately seems to be very artificial and nontrivial. 
Another typical way of dealing with functions with sharp transition is to use non-uniform mesh and put more points near the fast transition region. This technique works well for grid based method, but not for our case. In fact, we have tried to assign $150$ uniform points inside the boundary layer $[0, \eps]$ (from the reference solution, we observe that the thickness of the boundary layer is $\eps$ in this example) and $50$ points in the rest of the domain $(\eps, 1]$, but the result is still unsatisfactory, see the right plot of Figure~\ref{fig: rg_toy_2_bl_r2}.

\begin{figure}[htbp]
	\centering
	{\includegraphics[width=0.4\textwidth]{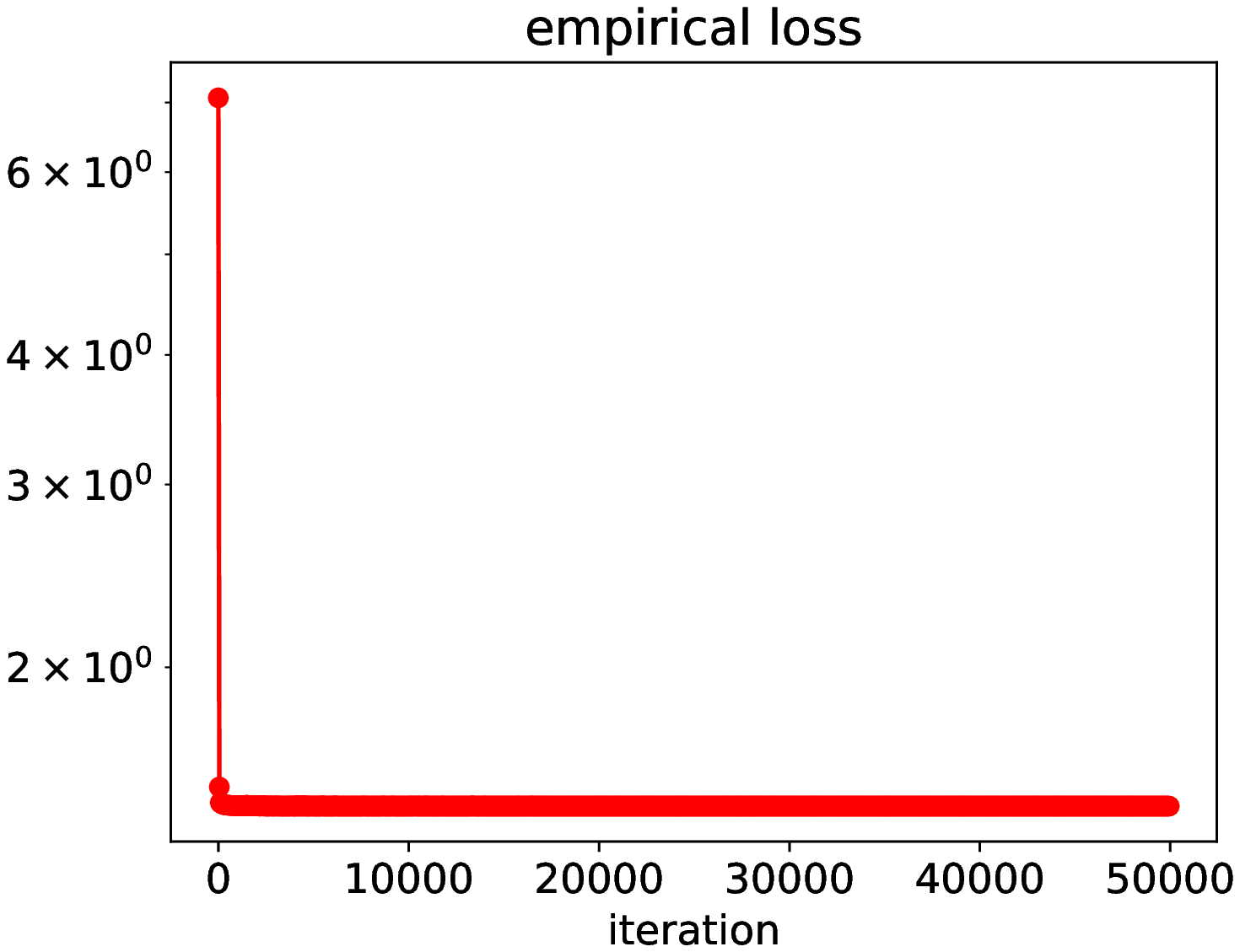}}
	{\includegraphics[width=0.4\textwidth]{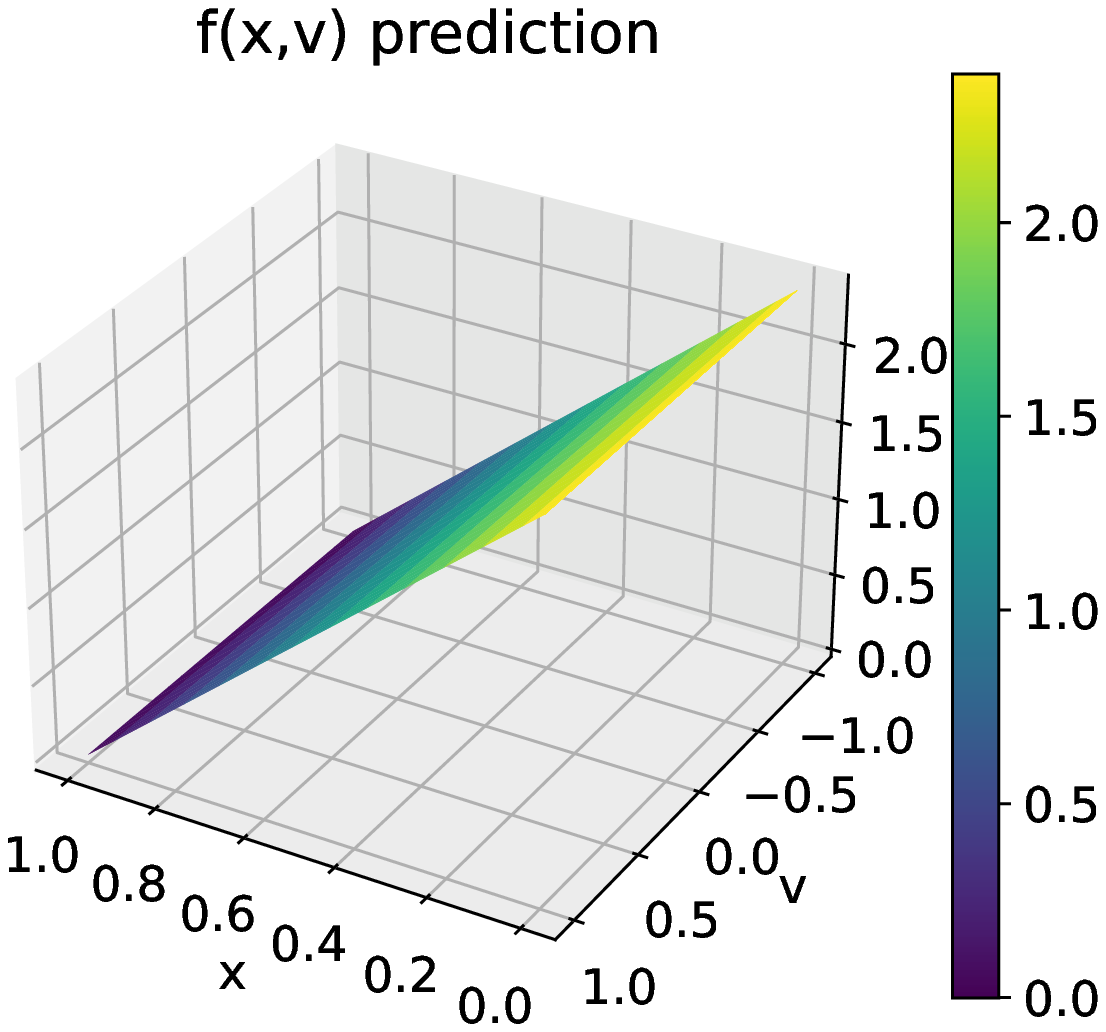}}
	\caption{Computation of \eqref{eqn:f_toy_2} with \eqref{eqn0913} and $B_{w,0} = B_{w,1} = 1$ for $\eps=10^{-3}$. Here we use a fully connected neural network with 4 hidden layers and 50 neurons within each hidden layer, and $N^b_v = 60$, $N^r_x = 80$ and $N^r_v=60$ in computing the empirical loss. The left is  the empirical loss versus iteration, and right is the prediction of $f(x,v)$. }
	\label{fig: rg_toy_2_bl_w1}
\end{figure}

\begin{figure}[htbp]
	\centering
	{\includegraphics[width=0.35\textwidth]{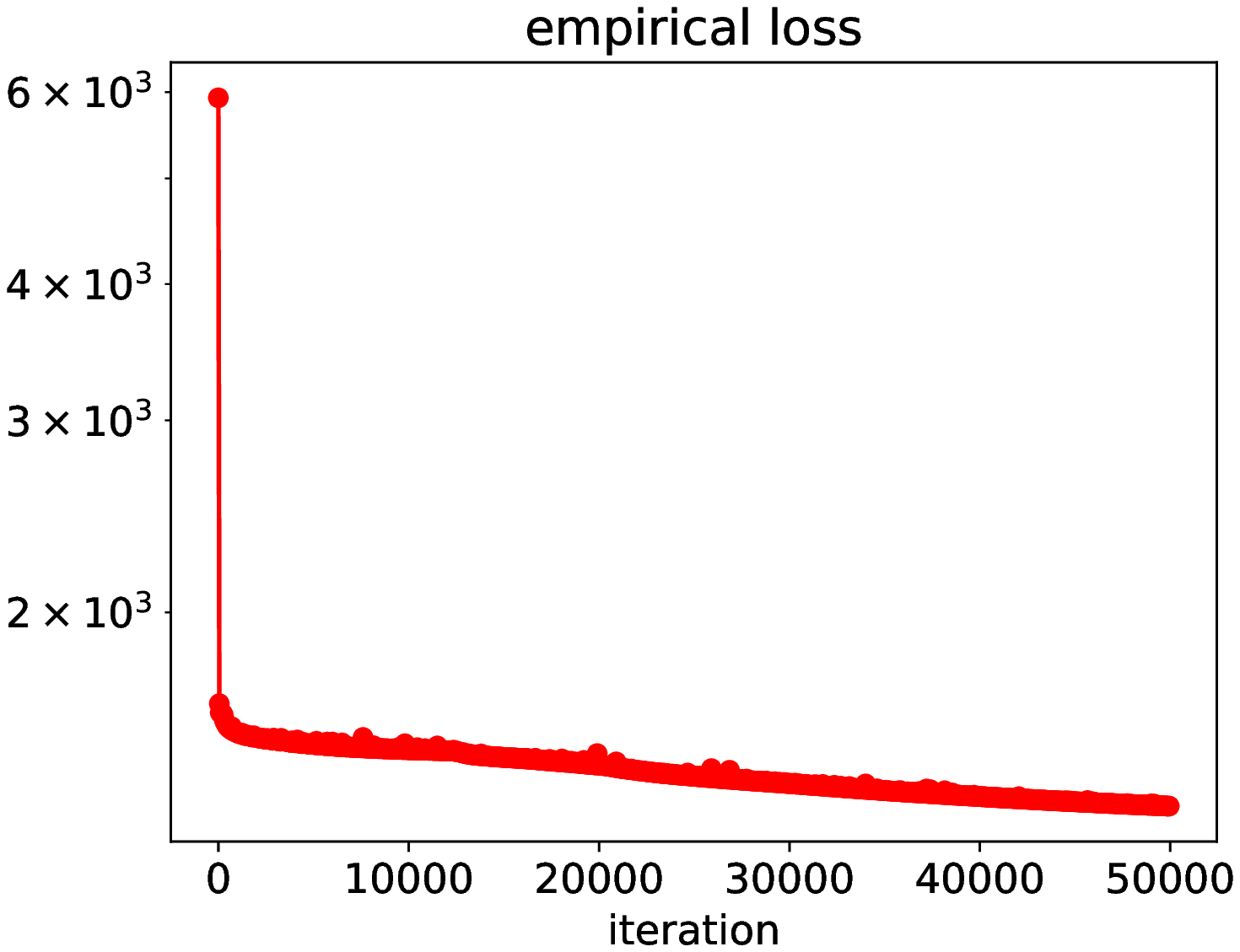}}
	{\includegraphics[width=0.35\textwidth]{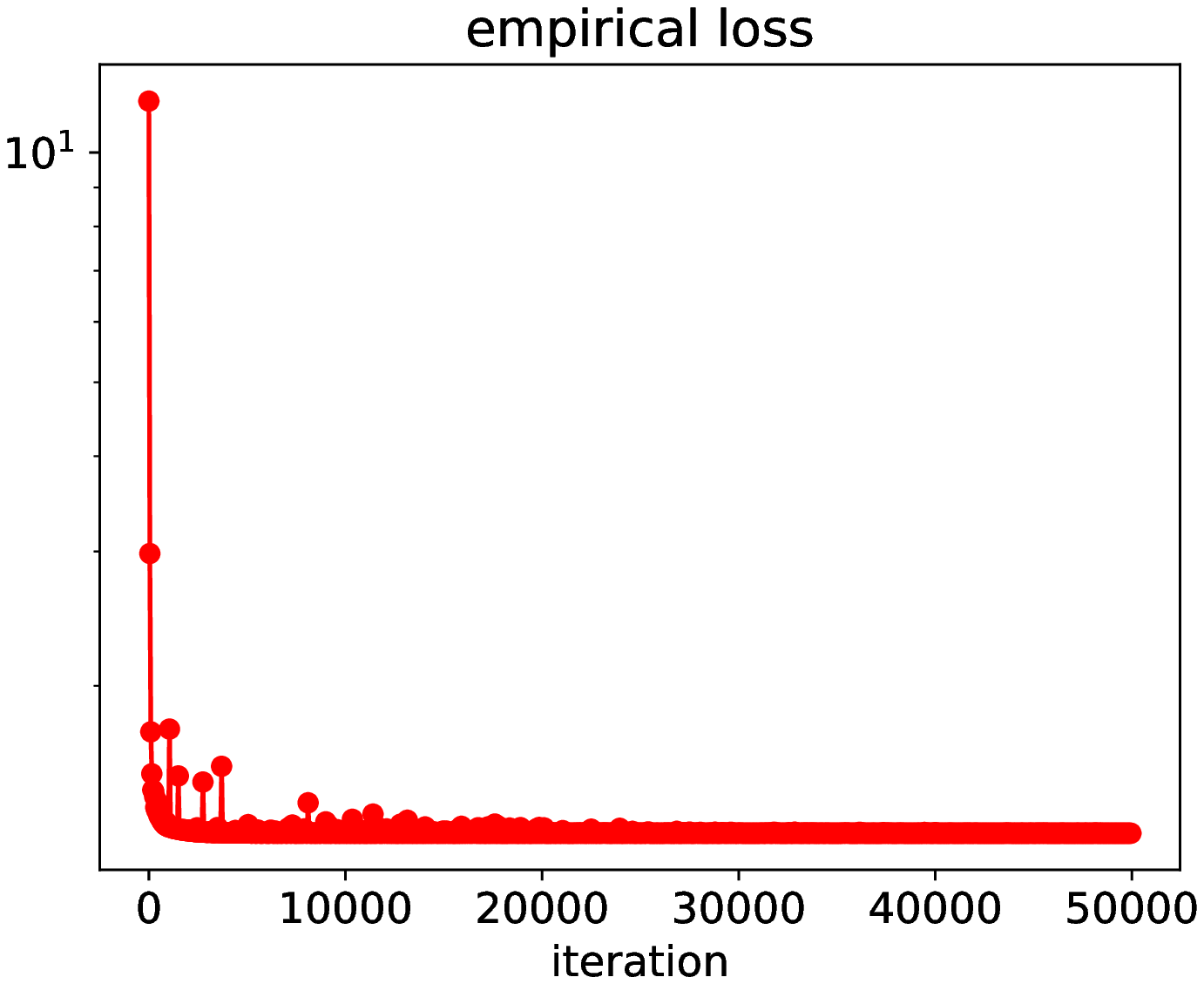}}
	\caption{Computation of \eqref{eqn:f_toy_2} with $\eps=10^{-3}$, using neural network approximation with 4 hidden layers and 50 neurons within each hidden layer. Left: using the loss function \eqref{eqn0913} with  $B_{w,0} =10^3$ and $B_{w,1} = 1$. And $N^b_v = 60$, $N^r_x = 80$ and $N^r_v=60$ in computing the empirical loss.  Right:  { using the loss function \eqref{eqn0913}} with  $B_{w,0} =B_{w,1} = 1$. And $N^b_v = 60$, $N^r_{x1} = 150$ in $(0, \eps)$, $N^r_{x2} = 50$ in $( \eps, 1)$ and $N^r_v=60$ in computing the empirical loss. }
	\label{fig: rg_toy_2_bl_r2}
\end{figure}

Therefore, we propose a new decomposition that includes a boundary layer corrector. In particular, we decompose $f$ as 
\begin{equation*} 
	f(\bx, \bv) = \tr(\bx) + \eps g(\bx, \bv) + \Gamma(\bx, \bv)\,, \quad \text{with}~ \average{g} = 0,~ \tr(\bx) = \average{f(\bx, \bv) - \Gamma(\bx, \bv)}\,.
\end{equation*}
Compared to \eqref{decompf0}, the main difference lies in the boundary layer corrector $\Gamma(\bx, \bv)$, which is obtained by solving the half space problem. More precisely, consider a change of variable $\Psi_\eps$:
\[
\Psi_\eps (\bx): \bx \in \Omega_x \subset \RR^d \mapsto (z, \bx_\bb), ~ z\in [0, \infty), ~\bx_\bb \in \partial \Omega_x\,,
\]
where $z$ is chosen according to \eqref{xz}, and thus $\Psi_\eps$ depends on $\eps$. For instance, when $\eps \rightarrow 0$ and $\bx \notin \partial \Omega_x$, $z=\infty$. Let $f_\bl(z,\bx_\bb, \bv)$ be the solution to \eqref{eqn:rte_hsp}, then $\Gamma(\bx, \bv)$ is obtained by
\begin{align*}
	\Gamma(\bx,\bv) ={ f_\bl(\Psi_\eps(\bx), \bv) - f_\bl(\Psi_{\eps =0}(\bx),\bv)} \,.
\end{align*}
Consequently, $\Gamma$ carries over the sharp transition part of $f$, and leaves the remaining $\tilde \rho$ and $g$ smooth and {  easily } approximated by neural networks. In particular, $\tilde \rho$ and $g$ solve
\begin{equation}\label{eq:mmdBL}
	\begin{cases}{}
		\average{\bv \cdot \nabla_{\bx} g} + \frac{1}{\eps} \average{\bv \cdot \nabla_\bx \Gamma} =  -\sigma_a (\tilde\rho + \average{\Gamma}) + G  , 
		\\ \bv \cdot \nabla_{\bx} (\tilde \rho+ \Gamma  + \eps g) - \eps \average{\bv\cdot \nabla_{\bx} g} - \average{\bv \cdot \nabla_\bx \Gamma} = \sigma_s \Lop g + \frac{\sigma_s}{\eps} \Lop \Gamma  - \eps^2 \sigma_a g - \eps \sigma_a (\Gamma-\average{\Gamma}), 
		\\ \tilde \rho + \Gamma + \eps g \big|_{\Gamma_-} = \phi \,.
	\end{cases}
\end{equation}

In practice, $\Psi_\eps$ has an explicit form for domains with special geometry. Below we list special cases both in 1D and 2D domain.

\subsubsection{$\Gamma(x,v)$ in 1D}
Consider $\eqref{eqn:rte}$ in one dimensional domain with $x \in [0,1]$ and $v \in [-1,1]$:
\begin{equation}\label{eqn:rte_1d}
	\begin{cases}{}
		\displaystyle \eps v \partial_x f =  \sigma_s(x) \Lop (f)- \eps^2 \sigma_a(x) f+ \eps^2 G(x)  \,, \\
		f(0, v>0) = \phi_L(v) \,, \quad 
		f(1, v<0) = 0 .
	\end{cases}
\end{equation}
Without loss of generality, we only let boundary layer appears on the left, as the one on the right shall be treated in exactly the same way. Define the stretch variable 
$
z = \frac{1}{\eps}\int_0^x \sigma_s(s) \rd s  \in [0, \infty)
$
and let $f_\bl(z,v)$ solves 
\begin{equation} \label{eqn:hsp_1d1}
	\begin{cases}{}
		v \partial_z f_\bl(z,v) = \Lop (f_\bl) , \quad 
		\\ f_\bl (0, v)= \phi_L(v), ~v \in (0, 1].
	\end{cases}
\end{equation}
Then $\Gamma$ is obtained via
\begin{equation} \label{eqnGamma1}
	\Gamma(x,v) = f_\bl(\frac{1}{\eps} \int_0^x \sigma_s(s) \rd s, v) - f_\bl^\infty\,,
\end{equation}
where $f_\bl^\infty= \lim_{x\rightarrow \infty} f_\bl(x,v)$. 

In practice, we cannot solve \eqref{eqn:hsp_1d1} on an infinite domain. Instead, we pick a large enough number $Z$, and impose an additional condition 
\begin{equation} \label{vBL}
	\average{v f_\bl(z,\cdot)} =0\,, \quad \text{ for any } z\,.
\end{equation}
Indeed, taking average of $\eqref{eqn:hsp_1d1}$, one gets $\partial_z \average{v f_\bl} =0$, which implies that $\average{v f_\bl}$ is a constant. Noting from \eqref{eqn0914}, $f_\bl^\infty$ is independent of $v$, hence $\average{v f_\bl(\infty, v)} = 0$. Therefore \eqref{vBL} generally holds. Additionally, according to the Lemma 2.1 in \cite{golse2003domain}, $f_\bl(x,v)$ converges to the $f_{BL}^\infty$ exponentially fast, and therefore a moderate value of $Z$ shall be sufficient. In sum, we utilize the following loss function to obtain the solution to \eqref{eqn:hsp_1d1}: 
\begin{equation*}
	\mE(f_\bl) = \| v \partial_z f_\bl(z, v) - \Lop(f_\bl) \|^2_{L^2(\Omega)} + \| \average{v f_\bl} \|^2_{L^2(\Omega_z)} + \| f(0,v) - \phi_L(v) \|^2_{L^2(\Gamma_{-})}\,,
\end{equation*}
where $\Omega:=[0,Z] \times [-1, 1]$ and $\Gamma_{-}=(0,1]$.

Note that since $f_\bl^\infty$ is a constant and $f_\bl$ solves \eqref{eqn:hsp_1d1}, $\Gamma$ obtained from \eqref{eqnGamma1} should satisfy 
\[
\eps v \partial_x \Gamma = \sigma_s(x) \Lop(\Gamma) \,,
\]
and therefore the $(\tilde \rho, g)$ system \eqref{eq:mmdBL} simplifies to 
\begin{equation} \label{rhog1D}
	\begin{cases}{}
		\average{\bv \cdot \nabla_{\bx} g}  =  -\sigma_a (\tilde\rho + \average{\Gamma}) + G  , 
		\\ \bv \cdot \nabla_{\bx} (\tilde \rho  + \eps g) - \eps \average{\bv\cdot \nabla_{\bx} g} = \sigma_s \Lop g   - \eps^2 \sigma_a g - \eps \sigma_a (\Gamma-\average{\Gamma}), 
		\\ \tilde \rho + \Gamma + \eps g \big|_{\Gamma_-} = \phi \,.
	\end{cases}
\end{equation}

\subsubsection{$\Gamma(\bx,\bv)$ in 2D square domain}
As a second example, we consider a two dimensional square domain with  $\bx = (x,y) \in [-1,1]^2$, $\bv = (\cos \alpha, \sin \alpha), ~ \alpha \in [0, 2 \pi]$. We assume that only the boundary $x=-1$ has a boundary layer and choose $\sigma_s(\bx) = 1$, then the boundary value problem reads:
\begin{equation*} 
	\begin{cases}{}
		\displaystyle \eps \bv \cdot \nabla_{\bx} f =  \Lop (f) - \eps^2 \sigma_a(\bx) f + \eps^2 G(\bx)  \,, \\
		f(-1, y, \alpha)  = \phi_L(y, \alpha), ~ \alpha \in [0, \pi/2] \cup [3 \pi/2, 2 \pi]  \,, \\
		f(1, y ,\alpha)  = 0,  ~ \alpha \in [\pi/2, 3 \pi /2]  \,, \\
		f(x, -1 ,\alpha) = 0,  ~ \alpha \in [0, \pi ] \,, \\
		f(x, 1 ,\alpha) = 0,  ~ \alpha \in [\pi, 2 \pi ]\,,
	\end{cases}
\end{equation*}
where $\Lop (f) = \average{f} -f$ with $\average{f} = \frac{1}{2\pi} \int_0^{2\pi} f \rd \alpha$. 
In this case, we define the stretch variable $z$ as
\[
z=\frac{x+1}{\eps}\,,
\]
and solve $f_\bl(z,y,\alpha)$ from
\begin{equation} \label{eqn:hsp_2Dy}
	\begin{cases}{}
		\cos \alpha \partial_z f_\bl = \Lop (f_\bl) \,, \\
		f_\bl (0, y, \alpha) = \phi_L(y, \alpha), ~ \cos \alpha >0.
	\end{cases}
\end{equation}
Then the boundary layer corrector can be obtained as
\[
\Gamma(x,y,\alpha) = f_\bl(\frac{x+1}{\eps}, y, \alpha) - f_\bl^\infty (y) \,,
\]
where $f_\bl^\infty (y) = \lim_{x \rightarrow \infty} f_\bl(x,y, \alpha)$.

As in the previous case, we do not solve \eqref{eqn:hsp_2Dy} on infinite domain. Instead, we impose the zero flux condition
\[
\average{\cos \alpha f_\bl(z,y, \alpha)} = \frac{1}{2\pi} \int_0^{2\pi} \cos \alpha f_\bl(z,y, \alpha) \rd \alpha =   0\,.
\]
To implement, we place $N_y$ grid points on $y$ and denote them by $y_j,~ j= 1, \cdots, N_y$. Then for each fixed $y_j$, we use the following loss function to obtain $f_\bl(z,y_j, \alpha)$:
\begin{align*}
	\mE(f_\bl(\cdot,y_j, \cdot)) = &\| \cos \alpha \partial_z f_\bl(\cdot,y_j,\cdot) - \Lop(f_\bl(\cdot,y_j,\cdot)) \|^2_{L^2(\Omega)} \\ & + \| \average{\cos \alpha f_\bl(\cdot,y_j,\alpha)} \|^2_{L^2(\Omega_z)} + \| f_\bl(0,y_j,\cdot) - \phi_L(y_j, \cdot) \|^2_{L^2(\Gamma_{-})}\,,
\end{align*}
where $\Omega:=\Omega_z \times [0, 2 \pi] = [0,Z] \times [0, 2 \pi]$ and $\Gamma_{-}=[0, \pi/2] \cup [3 \pi/2, 2 \pi] $.
Consequently, since 
$\eps \cos \alpha \partial_x \Gamma = \Lop(\Gamma)$, $\eqref{eq:mmdBL}$ in 2D is simplified to
\begin{equation*} 
	\begin{cases}{}
		\average{\bv \cdot \nabla_{\bx} g} + \frac{1}{\eps} \average{\sin \alpha \partial_y \Gamma} =  -\sigma_a (\tilde\rho + \average{\Gamma}) + G  , 
		\\ \bv \cdot \nabla_{\bx} (\tilde \rho  + \eps g) + \sin \alpha \partial_y \Gamma  = \Lop g   - \eps \sigma_a (\tr + \eps g + \Gamma) + \eps G, 
		\\ \tilde \rho + \Gamma + \eps g \big|_{\Gamma_-} = \phi \,.
	\end{cases}
\end{equation*}

\subsection{Algorithm}
In this section, we include implementation details of our algorithm. For expository simplicity, we will describe it for one dimensional case, i.e., \eqref{eqn:rte_1d}. The generalization to two dimensions is straightforward.

The first step is to obtain a neural network approximation $f_\bl^{nn}(\theta; z,v)$ to the half space problem \eqref{eqn:hsp_1d1}. To this end, we first generate the training set. For residual loss, given a sufficiently large number $Z$, assign $N_z^r$ uniform grid points on $z \in [0, Z)$, i.e. $z_i^r = (i-1) \Delta z$ with $\Delta z = Z/N_z^r$. { Since the physical dimension of $v$ is at most two, to numerically evaluate the integration we choose $N_v^r$ Gaussian quadrature points, which requires relatively smaller number of points and provide better accuracy than the uniform grid points and the randomly sample points. Denoted the Gaussian quadrature points as $\{v_j^r\}_{j=1}^{N_v^r}$}, with corresponding weights $\{ w^r_j  \}_{j=1}^{N^r_v}$. For boundary loss function, we randomly sample $N_v^b$ velocity points and denote them as $ \{ v^b_j  \}_{j=1}^{N^b_v}$. Then the empirical loss function becomes
\begin{equation*}
	\begin{aligned}
		\mE_\bl^N (f_\bl^{nn}(\theta)) &= \sum_{i=1}^{N_z^r} \sum_{j=1}^{N_v^r} (\partial_z f_\bl^{nn}(\theta; z_i^r, v_j^r) - \Lop f_\bl^{nn}(\theta; z_i^r, v_j^r) )^2 w_j^r \Delta z \\
		& + \sum_{i=1}^{N_z^r}( \sum_{j=1}^{N_v^r} w^r_j v_j^r f_\bl^{nn}(\theta; z_i^r, v_j^r)   )^2 \Delta z + \frac{1}{N_b} \sum_{j=1}^{N_v^b} (f_\bl^{nn}(\theta; 0, v^b_j) - \phi_L(v^b_j))^2\,.
	\end{aligned}
\end{equation*}
Here the second term on the right hand side corresponds to the condition \eqref{vBL}. Minimizing over $\theta$ of $\mE_\bl^N (f_\bl^{nn}(\theta))$, one gets 
\begin{equation} \label{eqn0916}
	\theta_* = \arg \min_{\theta} \mE_\bl^N (f_\bl^{nn}(\theta)) \,,
\end{equation}
and thus obtains $f_\bl^{nn}(\theta_*; z, v)$. Here we summarize the procedure of training the neural network in Algorithm~\ref{alg1}.

\begin{algorithm}[H]
	\SetAlgoLined
	Input: Training set $\{z_i^r\}_{i=1}^{N_z^r}$, $\{v_j^r\}_{j=1}^{N_v^r}$,  $\{ w^r_j  \}_{j=1}^{N^r_v}$, $ \{ v^b_j  \}_{j=1}^{N^b_v}$; neural network parameters: number of hidden layer $n_l$, number of neurons in each layer $n_r$ and activation function; two max iteration numbers $I_{max1}$, $I_{max2}$.
	
	Output: $\theta_*$
	
	Initialize neural network; 
	
	Set $k_1=0$, $k_2=0$ \\
	\While{$k_1<I_{max1}$ and $\mE_\bl^N(\theta^{k_1}) < \delta_1$}{
		Update $\theta^{k_1}$ by applying Adam to problem $\eqref{eqn0916}$, $k_1 = k_1 +1$\,; 
	}
	Let $\theta^{k_2=0} = \zeta^{k_1}$ \;
	\While{$k_2<I_{max2}$ and $\nabla_\theta \mE_\bl^N(\theta^{k_2}) < \delta_2$}{
		Update $\theta^{k_2}$ by applying LBFGS to problem $\eqref{eqn0916}$, $k_2 = k_2 +1$ \,;
	}
	$\theta_* = \theta^{k_2}$
	\caption{Algorithm for $\eqref{eqn0916}$}
	\label{alg1}
\end{algorithm}

After this, we denote $f_\bl^\infty \approx f_\bl^{nn}(\theta; Z, 0)$ as it is homogeneous in $v$ and 
extend the function value of $f_\bl^{nn}(\theta_*; z, v)$ as 
\begin{equation*}
	f_\bl^{nn}(\theta_*; z, v) = \left \{ \begin{array}{cc} f_\bl^{nn}(\theta_*; z, v)\,, & 0 \leq z \leq Z \,,
		\\ f_\bl^\infty \,, & z > Z\,,
	\end{array} \right.
\end{equation*}
and  compute the boundary layer corrector as follows: 
\begin{equation*}
	\Gamma(x,v) = f_\bl^{nn}(\theta_*; \frac{x}{\eps}, v) - f_\bl^{\infty}\,.
\end{equation*}

To proceed, we solve the following macro-micro system, which is tailored from \eqref{rhog1D} to adapt the specific boundary condition here  
\begin{equation*} \label{eqn: rg_hsp_general}
	\begin{cases}{}
		\langle v \partial_x g \rangle = -\sigma_a \tr - \sigma_a \average{\Gamma} + G, \, \\
		v \partial_x(\tr + \eps g) = \sigma_s \Lop(g)  -\eps \sigma_a(\tr + \eps g + \Gamma) + \eps G  , \, \\
		\tr(0)  + \eps g(0, v>0) + \Gamma(0,v>0) = \phi_L(v), \\
		\tr(1) + \eps g(1, v<0) + \Gamma(1, v<0) = 0 .
	\end{cases}
\end{equation*}
As before, we first generate the training set, which consists of uniform grids in $x$ and Gaussian quadrature point in $v$ for residual loss, and random sample points in $v$ for boundary loss. Once the empirical loss function is formed, applying the same optimization procedure, we obtain the predicted solution $\tilde \rho^{nn}$ and $g^{nn}$. 

\section{Theoretical analysis}\label{sec:mainthm}
We hereby provide a theoretical justification of our neural network formulation. In particular, we intend to show that the $L^2$ error of the predicted solution by neural network is {\it uniformly} bounded by the loss function. Let us first state a theorem that justifies the well-posedness of the $(\rho,g)$-system \eqref{eq:mmd}. 
\begin{theorem}
	Let Assumption \ref{ass:sigma} and Assumption \ref{ass:L} hold. Then the system \eqref{eq:mmd} has a unique solution $(\rho,g) \in \mX\times \mX$ with $\average{g} = 0$. 
\end{theorem}

\begin{proof}
	The existence of $(\rho,g) \in \mX\times \mX$ follows from Theorem \ref{thm:wellpos}. Indeed, let $f\in \mX$ be the unique solution of \eqref{eqn:rte}. Then by construction, the pair $(\rho, g) := (\average{f}, f - \average{f})\in \mX\times\mX$ solves \eqref{eq:mmd} with $\average{g} =0$. Moreover, the uniqueness follows by tracking the proof of Lemma \ref{lem:stab} (see the bound \eqref{eq:stab5}).
\end{proof}

Next we proceed to show that our new (population) loss function $\mE(f)$ defined in \eqref{eq:loss0} satisfies a stability estimate, namely the  $L^2$-error between the neural networks solution $f$ and the exact solution $f^\ast$ can be bounded above by $\mE(f)$.  Let $(\rho^\ast, g^\ast)\in \mX\times \mX$ be the solution to the macro-micro system \eqref{eq:mmd} and $f^\ast =\rho^\ast+\eps g^\ast $ be the exact solution to \eqref{eqn:rte}.  Let $(\rho, g)\in \mX\times \mX$ be a neural network approximation to  $(\rho^\ast, g^\ast)$ and let $f =\rho+\eps g$. 
The the main theoretical result is as follows. 

\begin{theorem}\label{thm:main}
	Let  $(\rho, g)\in \mX\times \mX$. Then there exists a constant $C_\eps>0$ such that $\lim_{\eps \downarrow 0}C_\eps < \infty$ and that 
	\begin{equation}\label{eq:sstab1}
		\|f - f^\ast\|_{L^2(\Omega)}^2 \leq \frac{C_\eps}{\eps}\mE(f)\,,
	\end{equation}
	where $\mE(f)$ is defined in \eqref{eq:loss0}.
	If in addition $\phi = \phi(x) \in H^{\frac{1}{2}}(\partial \Omega_{\bx})$ and $\rho \in H^1(\Omega_{\bx})$, then 
	\begin{equation}\label{eq:sstab2}
		\|f - f^\ast\|_{L^2(\Omega)}^2 \leq C_\eps\mE(f),
	\end{equation}
	where again $C_\eps$ satisfies that $\lim_{\eps \downarrow 0}C_\eps < \infty$.
\end{theorem}

\begin{remark}
	The estimate \eqref{eq:sstab2} of Theorem \ref{thm:main} shows that if the boundary data $\phi \in H^{\frac{1}{2}}(\partial \Omega_{\bx})$ and the approximate solution $\rho \in H^1(\Omega_\bx)$, then the $L^2$-error between $f$ and $f^\ast$ can be bounded by the loss $\mE(f)$ uniformly in the regime where $\eps $ is small. It is worth to comment on the role of the above stability estimate in the numerical analysis of neural network methods for solving PDEs. In fact, from a practical perspective, an approximate solution is parameterized by neural networks $f^N_\theta$ and is obtained by minimizing the empirical loss $\mE^N(f^{nn})$ (instead of $\mE(f^{nn})$) with respect to the neural network parameters $\theta$. Thanks to the well-established generalization theory of statistical learning \cite{shalev2014understanding}, the difference between the population loss $\mE(f_\theta^N)$ and the empirical loss $\mE^N(f_\theta^N)$ (also known as the generalization gap) can be made arbitrarily small as both the number of quadrature points and the complexity of the neural network increase to infinity. As a result, the population loss $\mE(f_\theta^N)$, and equivalently the $L^2$-error $\|f-f^\ast\|_{L^2(\Omega)}$ (thanks to Theorem \ref{thm:main}), can be made small through minimizing the  empirical loss $\mE^N(f_\theta^N)$ via training. In another word, the stability bounds enable us to transfer the bound on trainable loss function to the solution.  In the present paper, we only focus on the stability estimate and leave the complete generalization error analysis to the interested readers; such generalization analysis for neural networks has been carried out in the context of PDEs, see e.g.  \cite{mishra2020estimates,lu2021priori,lu2021priori2}.  
	
\end{remark}

\begin{proof}[Proof of Theorem \ref{thm:main}]
	Let us define $\tilde{\rho} = \rho - \rho^\ast, \tilde{g} = g - g^\ast$ and $\tilde{f} = f - f^\ast$. Then it is easy to verify that $(\tilde{\rho}, \tilde{g})$ solves the  boundary value problems
	\begin{equation*}
		\begin{aligned} 
			\average{\bv \cdot \nabla_{\bx} \tilde{g}} + \sigma_a \tilde{\rho}  & =: r_1 & \text{ in }  \Omega,\\
			\bv \cdot \nabla_{\bx} (\tilde{\rho} + \eps \tilde{g})  -\eps\average{ \bv\cdot\nabla_{\bx} \tilde{g}} - \sigma_s \Lop \tilde{g} + \eps^2 \sigma_a \tilde{g} & =: r_2 & \text{ in } \Omega, \\
			\tilde{\rho} + \eps \tilde{g} & =: r_3 & \text{ on }  \Gamma_-\,.
		\end{aligned}
	\end{equation*}
	Then $\tilde{f}$ satisfies that
	$$\begin{aligned}\label{eq:tf}
		\eps \bv \cdot \nabla_{\bx} \tilde{f} & = \sigma_s(\bx) \Lop \tilde{f}(\bx, \bv) - \eps^2\sigma_a \tilde{f} + \eps^2r_1(\bx) + \eps r_2(\bx,\bv) & \text{ on }  \Omega,\\
		\tilde{f} & = r_3 & \text{ on }  \Gamma_-.
	\end{aligned}
	$$
	Observe that by definition $\average{r_2} = \eps^2 \sigma_a \langle  g\rangle$. { Then an application of Lemma \ref{lem:stab} with $\eta = r_2 - \langle r_2\rangle$ and  $\xi = r_1 +  \eps^{-1}\langle r_2\rangle = r_1+ \eps \sigma_a \average{g}$},  the estimate \eqref{eq:sstab1} follows from \eqref{eq:stab1}. Furthermore, if $\phi = \phi(x) \in H^{\frac{1}{2}}(\partial \Omega_x)$ and $\rho \in H^1(\Omega)$, then on $\Gamma_-$ one has $f(\bx) = \rho(\bx) - \phi(\bx) + \eps g(\bx,\bx)$ with $ \rho(\bx) - \phi(\bx)  \in H^{\frac{1}{2}}(\partial \Omega_x)$. Therefore applying the estimate \eqref{eq:stab2} of Lemma \ref{lem:stab} leads to \eqref{eq:sstab2}.
\end{proof}

\begin{lemma}\label{lem:stab}
	Let $f \in H^1(\Omega)$ solve the problem 
	\begin{equation}\label{eq:pdeg}
		\begin{aligned}
			\eps \bv\cdot \nabla_{\bx} f(\bx,\bv)& = \sigma_s \mL f(\bx, \bv) - \eps^2\sigma_a f  + \eps^2 \xi(\bx) + \eps \eta(\bx, \bv) &  \text{ on } \Omega,\\
			f(\bx,\bv) & = \zeta(\bx,\bv) &  \text{ on } \Gamma_-,
		\end{aligned}
	\end{equation}
	where $\xi\in L^2(\Omega_x), \eta \in L^2(\Omega)$ with $\average{\eta} = 0$ and $\zeta\in L^2(\Gamma_-)$. 
	Then there exists a constant $C_\eps>0$ depending on $\eps,\sigma_a,\sigma_s, C_K$ such that $\lim_{\eps\downarrow 0} C_\eps<\infty$ and that 
	\begin{equation}\label{eq:stab1}
		\|f\|_{L^2(\Omega)}^2\leq C_\eps (\|\xi\|_{L^2(\Omega_x)}^2 + \|\eta\|_{L^2(\Omega)}^2 + \eps^{-1} \|\zeta\|_{L^2(\Gamma_-)}^2)\,.
	\end{equation}
	If in addition $\zeta  = \zeta_1(\bx) + \eps \zeta_2(\bx,\bv)$ where $\zeta_1 \in H^{\frac{1}{2}}(\partial \Omega_x)$ and $ \zeta_2\in L^2(\Gamma_-)$, then 
	\begin{equation}\label{eq:stab2}
		\|f\|_{L^2(\Omega)}^2\leq  C_\eps(\|\xi\|_{L^2(\Omega_x)}^2 + \|\eta\|_{L^2(\Omega)}^2 + \|\zeta_1\|_{H^{\frac{1}{2}}(\partial \Omega_x)}^2 + \eps\|\zeta_2\|_{L^2(\Gamma_-)}^2)\,.
	\end{equation}
	In particular, the stability constants in \eqref{eq:stab2} are uniformly bounded in $\eps$ as $\eps \downarrow 0$.
\end{lemma}

\begin{proof}
	Let us first prove the estimate \eqref{eq:stab1}. 
	Multiplying \eqref{eq:pdeg} with $f$ and the integrating on $\Omega$ leads to 
	\begin{equation}\label{eq:stab3}
		\eps \int_{\partial \Omega}(\bv,0) \cdot {\bn}  f^2 ds - \int_{\Omega}\sigma_s \mL f f  \rd x\,\rd v + \eps^2\int_{\Omega}   \sigma_a f^2 \rd x\,\rd v = \eps^2 \int_{\Omega} \xi f\rd x\,\rd v + \eps \int_{\Omega} \eta f\rd x\,\rd v.
	\end{equation}
	Notice from the definition of $\Gamma_-$ that 
	$$
	(\bv,0) \cdot \bn  =
	\begin{cases}
		- |(\bv,0) \cdot \bn|  & \text{ on } \Gamma_-, \\
		|(\bv,0) \cdot \bn|  & \text{ on } \Gamma_+. 
	\end{cases}
	$$ 
	Therefore we have from \eqref{eq:stab3} that 
	\begin{align*} 
		& \eps \int_{\Gamma_+} |(\bv,0) \cdot \bn|  f^2 ds \!-\! \int_{\Omega}\sigma_s \mL f f  \rd x\,\rd v + \eps^2\int_{\Omega}   \sigma_a f^2 \rd x\,\rd v 
		\\ & \quad =  \eps \int_{\Gamma_-} |(\bv,0) \cdot \bn|  f^2 ds \! + \!  \eps^2 \int_{\Omega} \xi f\rd x\,\rd v + \eps \int_{\Omega} \eta f \rd x\,\rd v.
	\end{align*}
	Thanks to 
	part (3) of Assumption \eqref{ass:L}, the positivity of $\sigma_s$ and the non-negativity of $\sigma_a$ we have from above that 
	\begin{equation}\label{eq:stab5}
		c \sigma_{\min} \|f - \average{f}\|_{L^2(\Omega)}^2+ \eps^2\int_{\Omega}   \sigma_a f^2 \rd x\,\rd v \leq \eps  \int_{\Gamma_-}|(\bv,0) \cdot \bn|  f^2 ds +   \eps^2 \int_{\Omega} \xi f \rd x\,\rd v + \eps \int_{\Omega} \eta f \rd x\,\rd v.
	\end{equation}
	Now let us write $f(\bx,\bv) = \rho(\bx) + \eps g(\bx, \bv) $ with $\rho = \average{f}$ and ${  g = \frac{1}{\eps} (f - \average{f})}$. Then  $(\rho, g)$ satisfy 
	\begin{equation}\label{eq:stab6}
		\begin{aligned}
			\average{\bv\cdot \nabla_{\bx} g} & = -\sigma_a \rho + \xi &  \text{ on } \Omega, \\
			\bv\cdot \nabla_{\bx} (\rho+\eps g) -\eps   \average{\bv\cdot \nabla_{\bx} g} - \sigma_s\mL g &= \eps^2 \sigma_a g+ \eta&  \text{ on } \Omega,\\
			\rho + \eps g &= \zeta&   \text{    on } \Gamma_-.
		\end{aligned}
	\end{equation}
	By the assumption that $\average{\eta} = 0$, one has 
	\begin{equation}\label{eq:stab7}
		\int_{\Omega} \eta f \rd x\,\rd v = \eps \int_{\Omega} \eta g \rd x\,\rd v. 
	\end{equation}
	It follows from \eqref{eq:stab5}, \eqref{eq:stab7} and Young's inequality that for $\alpha>0$,
	$$
	c \sigma_{\min} \|g \|_{L^2(\Omega)}^2 + \! \int_{\Omega} \!\!  \sigma_a f^2 \rd x\,\rd v \leq \frac{1}{\eps} \!\! \int_{\Gamma_-} \!\!\!\! |(\bv,0) \cdot \bn|  f^2 ds + \frac{\|\xi\|^2_{L^2(\Omega)}  +  \|\eta\|^2_{L^2(\Omega)}}{4\alpha}  + \alpha \|f\|^2_{L^2(\Omega)} + \alpha \|g \|^2_{L^2(\Omega)}.
	$$
	In particular, for any $\alpha \leq \frac{c\sigma_{\min}}{2}$, we have 
	$$
	\frac{c \sigma_{\min}}{2} \|g \|_{L^2(\Omega)}^2 +  \int_{\Omega}   \sigma_a f^2 \rd x\,\rd v \leq \frac{1}{\eps}  \int_{\Gamma_-}|(\bv,0) \cdot \bn|  f^2 ds + \frac{\|\xi\|^2_{L^2(\Omega)}  \!\!+\!  \|\eta\|^2_{L^2(\Omega)}}{4\alpha}  + \alpha \|f\|^2_{L^2(\Omega)}.
	$$
	Now taking $L^2$-norm on the second line of \eqref{eq:stab6} and applying Lemma \ref{lem:poin}, one obtains that 
	$$
	\begin{aligned}
		\|f\|_{L^2(\Omega)}^2 & \leq C_P\left(\|\bv \cdot \nabla_{\bx} f\|_{L^2(\Omega)}^2 + \int_{\Gamma_-} | (\bv, 0)\cdot \bn|  f^2 ds \right)\\
		& \leq C_P \left(\eps^2\| \xi - \sigma_a \rho\|_{L^2(\Omega)}^2 + \|\eps^2 \sigma_a g+ \eta\|_{L^2(\Omega)}^2 + \|\sigma_s\mL g\|_{L^2(\Omega)}^2 + \|\zeta\|_{L^2(\Gamma_-)}^2 \right)\\
		& \leq C_P \Big(2\eps^2\| \xi\|_{L^2(\Omega)}^2 + 2\eps^2 \|\sigma_a f\|_{L^2(\Omega)}^2 + 4 \eps^4\|\sigma_a g\|_{L^2(\Omega)}^2 + 2\| \eta\|_{L^2(\Omega)}^2 \\
		&  \qquad + \sigma_{\max}^2 C_K^2 \| g \|_{L^2(\Omega)}^2 + \|\zeta\|_{L^2(\Gamma_-)}^2 \Big)\\
	\end{aligned}
	$$
	where in the second inequality we used the fact that $(\bv,0)\cdot \bn \leq 1$ since $|\bv|=1$, and in the last inequality we have used part (5) of Assumption \ref{ass:L} and the fact that $\mL \average{g} = 0$. Combining the last two inequality and using the fact that $0\leq \sigma_a \leq \sigma_{\max}$, we obtain that 
	$$\begin{aligned}
		\|f\|_{L^2(\Omega)}^2 & \leq C_P \Big( 2\eps^2 \|\xi\|_{L^2(\Omega)}^2 
		+ \Big(2\eps^2\sigma_{\max} + \frac{2\sigma_{\max}^2(4\eps^4+C_K^2)}{c \sigma_{\min}}\Big)\\
		& \qquad \qquad \times \Big(\eps^{-1} \|\zeta\|_{L^2(\Gamma_-)}^2 +\frac{\|\xi\|^2_{L^2(\Omega)}  +  \|\eta\|^2_{L^2(\Omega)}}{4\alpha}  + \alpha \|f\|^2_{L^2(\Omega)}\Big) \\
		& \qquad\qquad + 2\| \eta\|_{L^2(\Omega)}^2  +  \|\zeta\|_{L^2(\Gamma_-)}^2 
		\Big).
	\end{aligned}
	$$
	Setting 
	$$
	\alpha = \alpha^\ast := \Big(4\eps^2\sigma_{\max} + \frac{4\sigma_{\max}^2(4\eps^4+C_K^2)}{c \sigma_{\min}}\Big)^{-1}\wedge \frac{c \sigma_{\min}}{2}
	$$ 
	in the above leads to  
	$$\begin{aligned}
		\|f\|_{L^2(\Omega)}^2 & \leq 2C_P \left( 2\eps^2 + \frac{1}{4\alpha^\ast} \Big(2\eps^2 \sigma_{\max}+ \frac{2\sigma_{\max}^2(4\eps^4+C_K^2)}{c \sigma_{\min}}\Big)\right)  \|\xi\|_{L^2(\Omega)}^2\\
		& + 2C_P \left( 2 + \frac{1}{4\alpha^\ast} \Big(2\eps^2\sigma_{\max} + \frac{2\sigma_{\max}^2(4\eps^4+C_K^2)}{c \sigma_{\min}}\Big)\right)  \|\eta\|_{L^2(\Omega)}^2 \\
		& + 2C_P \left( 1 +\frac{1}{4\alpha^\ast} \Big(2\eps^2\sigma_{\max} + \frac{2\sigma_{\max}^2(4\eps^4+C_K^2)}{c \sigma_{\min}}\Big)\eps^{-1}\right)  \|\zeta\|_{L^2(\Gamma_-)}^2.
	\end{aligned}
	$$
	This in particular implies \eqref{eq:stab1}. 
	
	Next we prove the improved estimate \eqref{eq:stab2} when  $\zeta  = \zeta_1(\bx) + \eps \zeta_2(\bx,\bv)$ where $\zeta_1 \in H^{\frac{1}{2}}(\partial \Omega_x)$ and $ \zeta_2\in L^2(\Gamma_-)$. In fact, let us first decompose the solution as $f(\bx) = f_1(\bx) + f_2(\bx,\bv)$, where $f_1$ solves the Laplace problem 
	$$\begin{aligned}
		\Delta f_1  & = 0\,,  & \text{ on } \Omega_{\bx},\\
		f_1 & = \zeta_1(\bx)\,, & \text{ on } \partial \Omega_{\bx}\,,
	\end{aligned}
	$$
	where by the standard regularity estimate  $\|f_1\|_{H^1(\Omega_{\bx})} \leq C_1\|\zeta_1\|_{H^{\frac{1}{2}}(\partial \Omega_x)}$ for some $C_1>0$, 
	and $f_2 = f - f_1$ solves 
	$$
	\begin{aligned}
		\eps \bv\cdot \nabla_{\bx} f_2(\bx,\bv)& = \sigma_s \mL f_2(\bx, \bv) - \eps^2 \sigma_a(\bx) f_2(\bx,\bv) - \eps^2 \sigma_a(\bx) f_1(\bx) & 
		\\ &  \qquad + \eps^2 \xi(\bx) + \eps \eta(\bx, \bv) - \eps \bv\cdot \nabla_{\bx}f_1(\bx)   & \text{ on } \Omega,\\
		f_2(\bx,\bv) & = \eps \zeta_2(\bx,\bv)   & \text{ on } \Gamma_-.
	\end{aligned}
	$$
	Applying the estimate \eqref{eq:stab1} to the problem above and noticing that $\average{\bv\cdot \nabla_{\bx} f_1(x)} = 0$, we have 
	$$\begin{aligned}
		\|f_2\|_{L^2(\Omega)}^2 & \leq C_{2,\eps}(\|\xi\|_{L^2(\Omega_x)}^2+\|\sigma_a f_1\|_{L^2(\Omega_x)}^2 + \|\eta\|_{L^2(\Omega)}^2 + \|\bv\cdot \nabla_{\bx}f_1\|_{L^2(\Omega)}^2 + \eps \|\zeta_2\|_{L^2(\Gamma_-)}^2) \\
		& \leq \tilde{C}_{2,\eps}(\|\xi\|_{L^2(\Omega_x)}^2 + \|\eta\|_{L^2(\Omega)}^2 + \|\zeta_1\|_{H^{\frac{1}{2}}(\partial \Omega_x)}^2 + \eps\|\zeta_2\|_{L^2(\Gamma_-)}^2).
	\end{aligned}$$
\end{proof}
Let us recall the following directional Poincar\'e inequality from \cite{manteuffel1999boundary}.
\begin{lemma}[Directional Poincar\'e inequality]\label{lem:poin}
	There exists a constant $C_P$ depending only on $\Omega$ such that 
	$$
	\|f\|_{L^2(\Omega)}^2 \leq C_P \left(\|\bv \cdot \nabla_{\bx} f\|_{L^2(\Omega)}^2 + \int_{\Gamma_-} | (\bv, 0)\cdot \bn|  f^2 ds \right).
	$$
\end{lemma}

\section{Numerical examples}\label{sec:num}
In this section, we conduct  extensive numerical experiments to verify the efficiency and accuracy of our neural network formulation based on macro-micro-(boundary layer) decomposition. For the  structure of the neural network, we always use a fully connected network with $n_l$ layers and $n_r$ number of neurons within each layer. In the following examples, we use $\sigma^l(z) = \tanh (z)$ as the activation function of the hidden layer. For the activation function of the output layer, we use $\sigma_{\rho}^{o} (z)= \ln (1+e^z)$, $\sigma_{g}^{o} (z) = z$, and $\sigma_{f_\bl}^{o} (z) = {C_a}/{(1+e^{-z})}$ for the macro part $\rho$, micro part $g$ and boundary layer $f_\bl$, respectively. Here $C_a$ is tuned according to the $L_{\infty}$ norm of the incoming boundary condition. For instance, $C_a = \| \phi_L(v) \|_{\infty}$ in solving $\eqref{eqn:hsp_1d1}$. When training the neural network, as introduced in Algorithm~\ref{alg1},  $I_{max1}, \delta_1$, $I_{max2}, \delta_2$ are the stopping parameters for Adam and LBFGS step, respectively. In 1D case, we choose $I_{max1} = 1.2\times10^4$, $\delta_1=0.005$, $I_{max2} = 10^4$, $\delta_2=10^{-6}$; in 2D case, we choose $I_{max1} = 2\times 10^4$, $\delta_1=0.01$, $I_{max2} = 10^4$, $\delta_2=10^{-6}$. Unless otherwise specified, the learning rate for Adam step is fixed to be $10^{-3}$. 

Upon obtaining the neural network prediction $f^{nn}:= \rho^{nn}(\bx) + \eps g^{nn}(\bx,\bv)$ or $f^{nn}:= \tilde \rho^{nn}(\bx) + \eps g^{nn}(\bx,\bv) + \Gamma^{nn}(\bx,\bv)$, we calculate its $L^2$ error to the reference solution as
\begin{equation} \label{error}
	error  = \frac{\sum_{i=1}^{N_x} \sum_{j=1}^{N_v} (f^{nn}(\bx_i, \bv_j) - f^{ref}(\bx_j,\bv_j))^2 w_i w_j}{\sum_{i=1}^{N_x} \sum_{j=1}^{N_v} (f^{ref}(\bx_j,\bv_j))^2 w_i w_j}\,.
\end{equation}
Here $\{\bx_i, w_i\}$ and $\{\bv_j, w_j\}$ are the test set and corresponding weight we use to calculate the reference solution, and therefore will be more refined than the training set. In particular, we again use the Gaussian quadrature for $\bv$ and uniform mesh in $\bx$ without boundary layer, or two sets of uniform mesh with boundary layer.

\subsection{RTEs without boundary layers}
In this subsection we consider RTEs in one and two dimensions where the solutions do not have  boundary layers. 

\begin{example}\label{ex:homo}
	1D problem with spatially homogeneous scattering:
	\begin{equation*} 
		\begin{cases}{}
			\eps v \partial_x f = \langle f \rangle - f \,, \\
			f(0, v>0)= 1, \quad f(1, v<0)= 0.
		\end{cases}
	\end{equation*}
\end{example}
Using the loss function \eqref{eq:loss0}, we collect the results of $\eps = 1$ and $\eps = 10^{-3}$ in Figure~\ref{fig: rg_1d_epsi1}. Here we use $n_l = 4$, $n_r = 50$, $N^r_x = 80$, $N^r_v=60$ and $N^b_v = 60$ for training. The reference solution is obtained by finite difference method on test set with $N_x=200$, $N_v=80$. It is evident that in both cases, the prediction obtained by the neural networks matches well with the reference solution, which is provided by a finite difference solver. Additionally, the relative $L^2$ error \eqref{error} is well controlled by the loss function.

\begin{figure}[htbp]
	\centering
	{\includegraphics[width=0.3\textwidth]{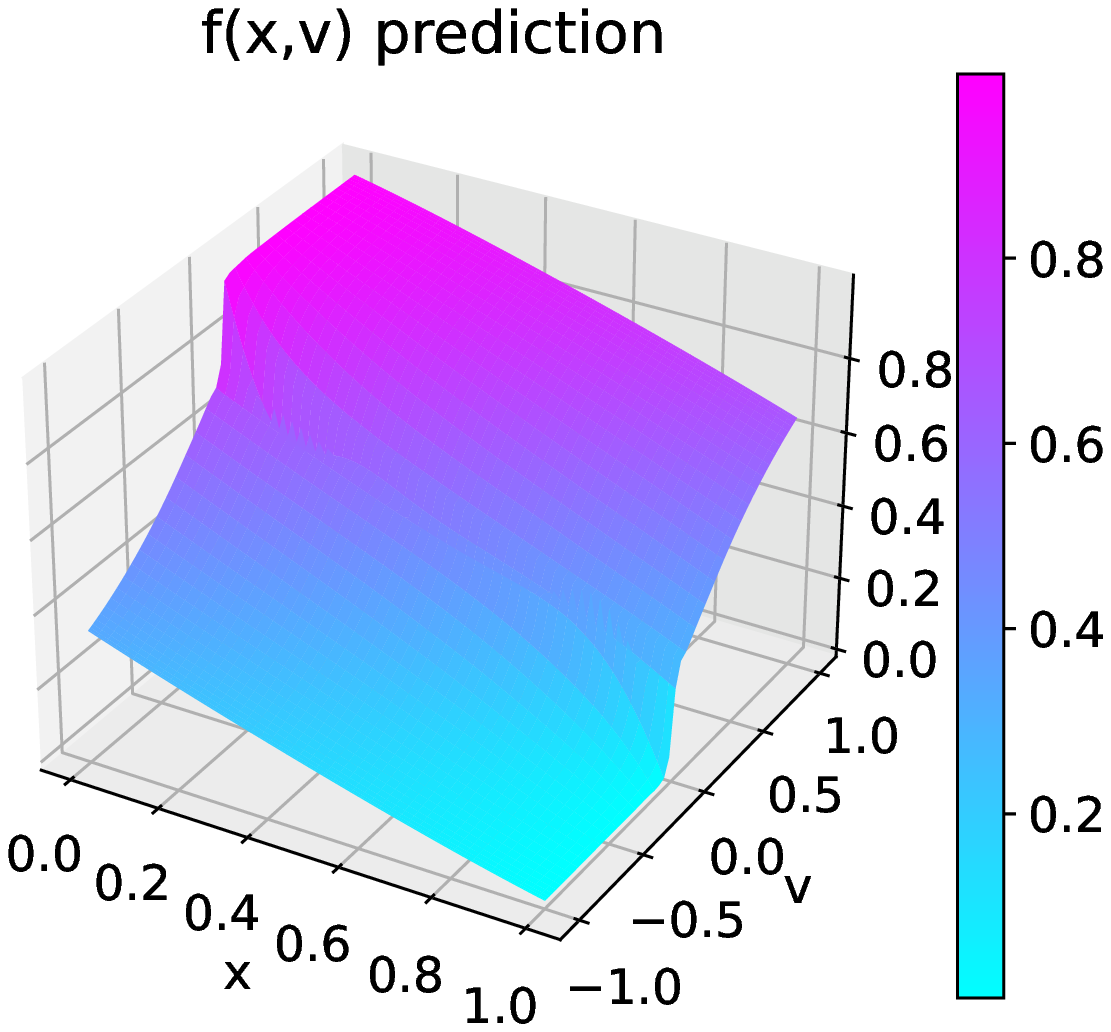}}
	{\includegraphics[width=0.3\textwidth]{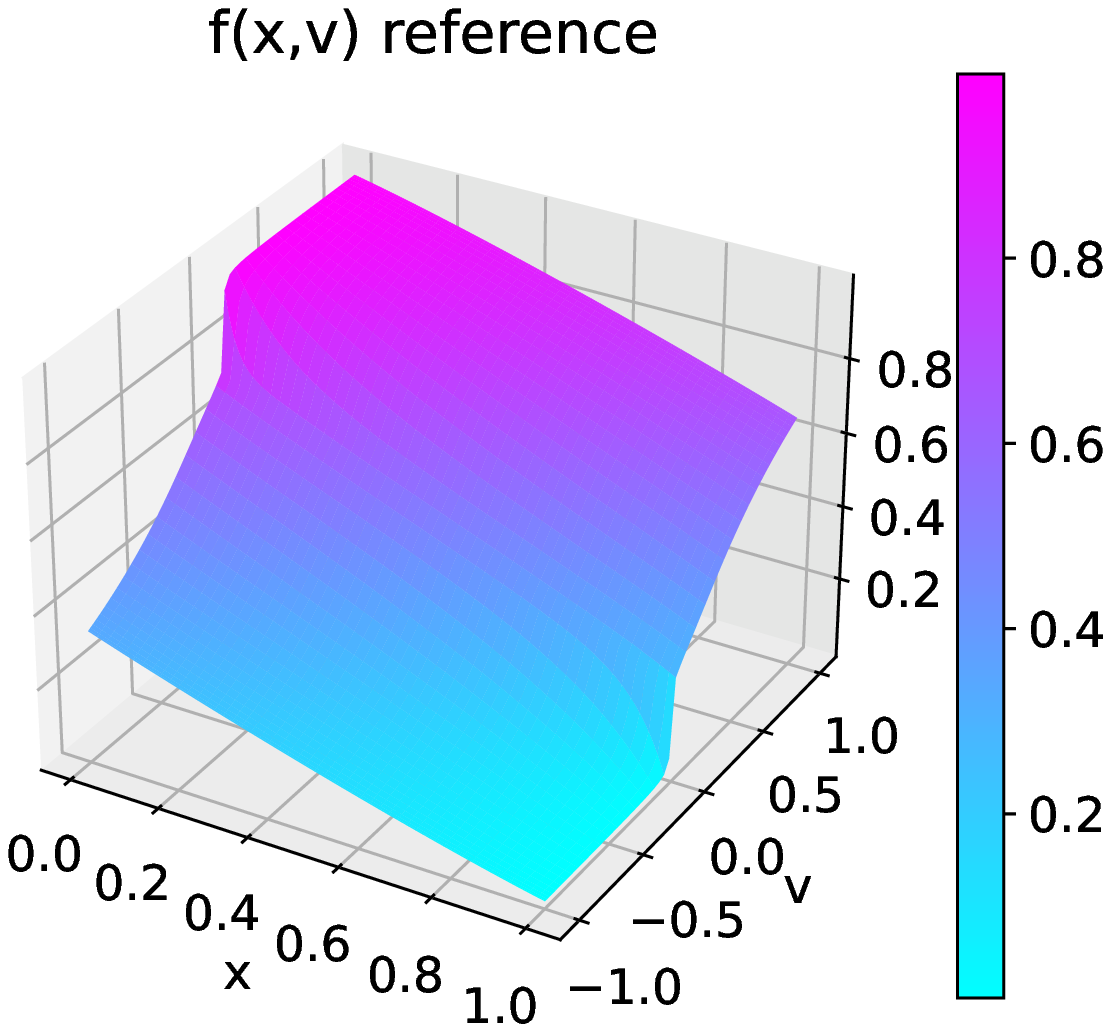}}
	{\includegraphics[width=0.3\textwidth]{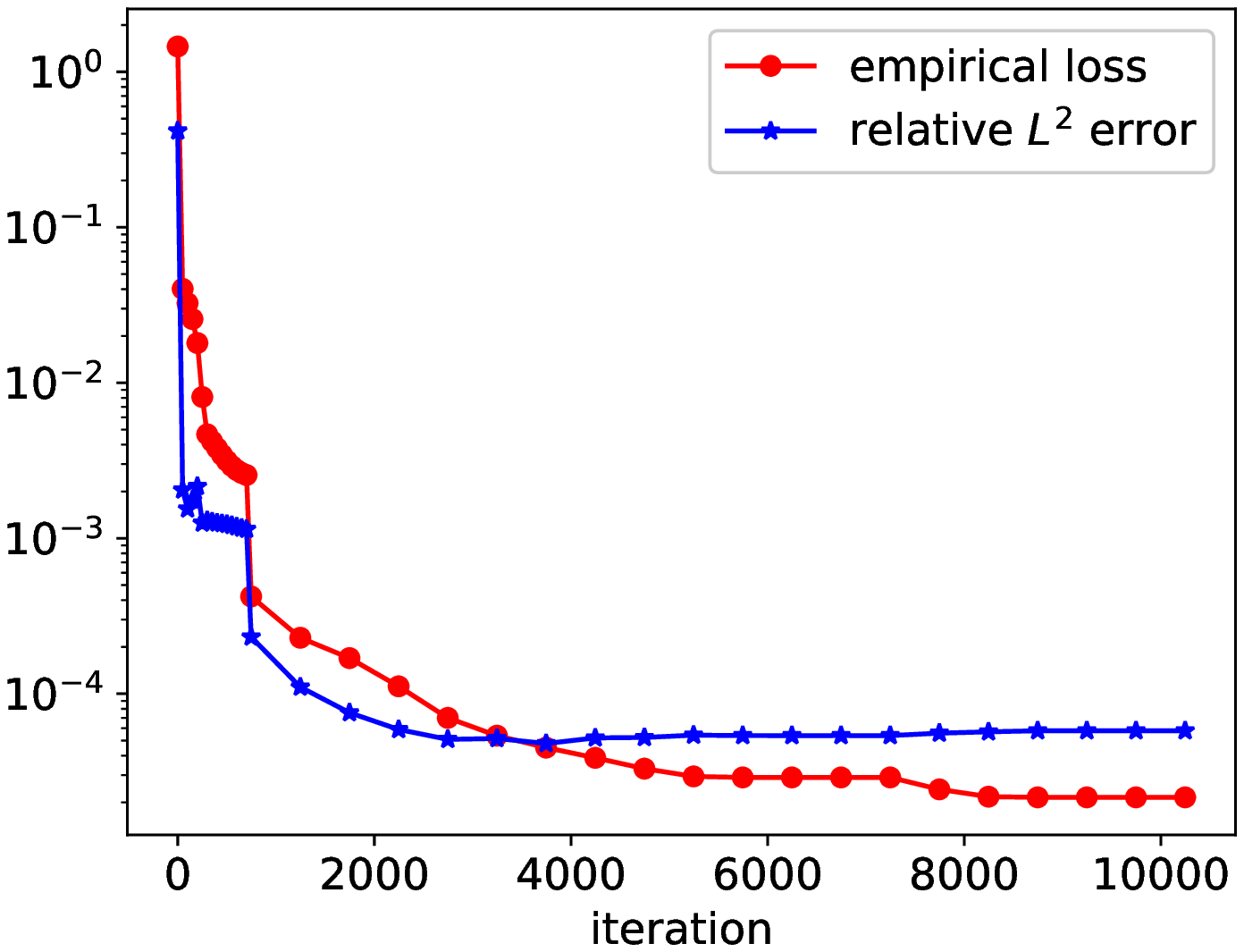}}
	{\includegraphics[width=0.3\textwidth]{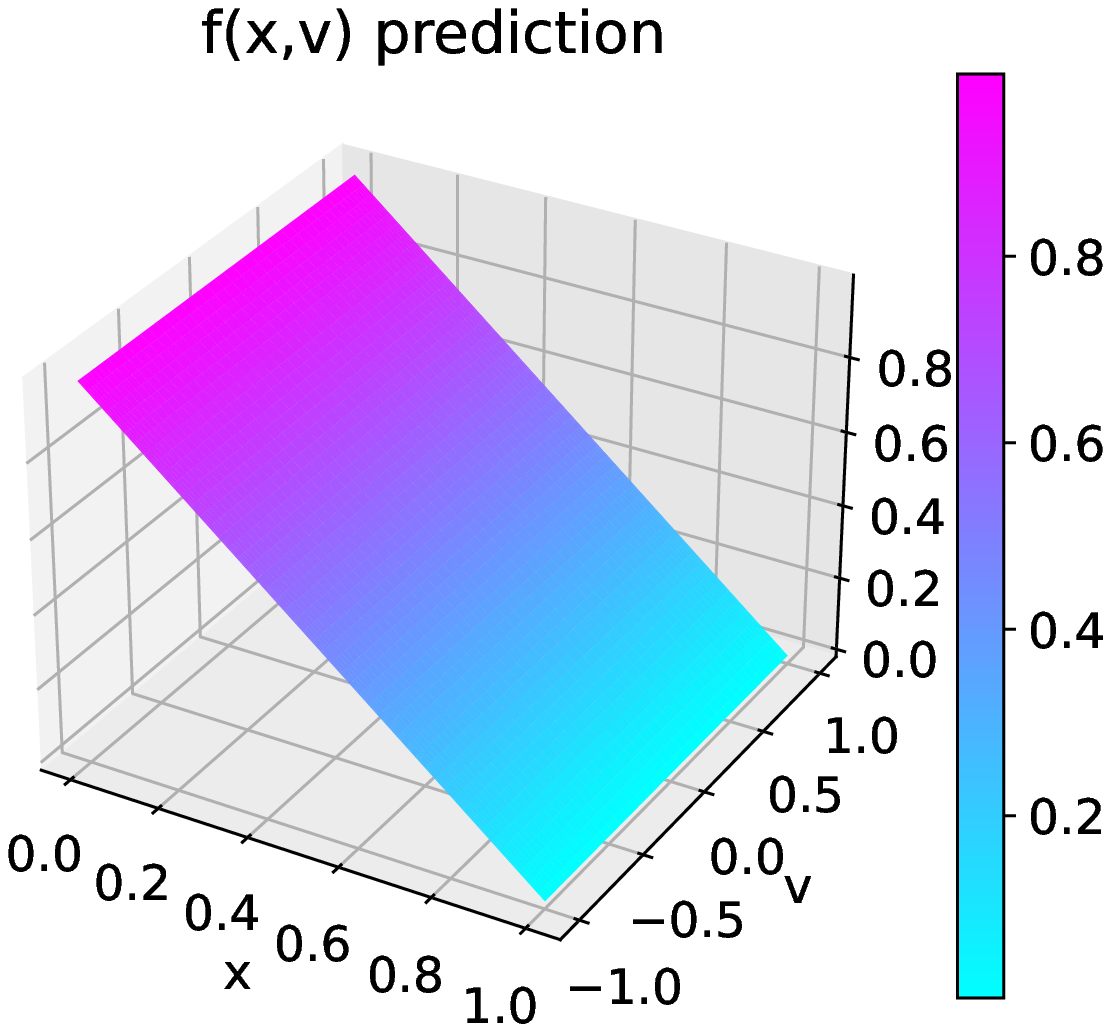}}
	{\includegraphics[width=0.3\textwidth]{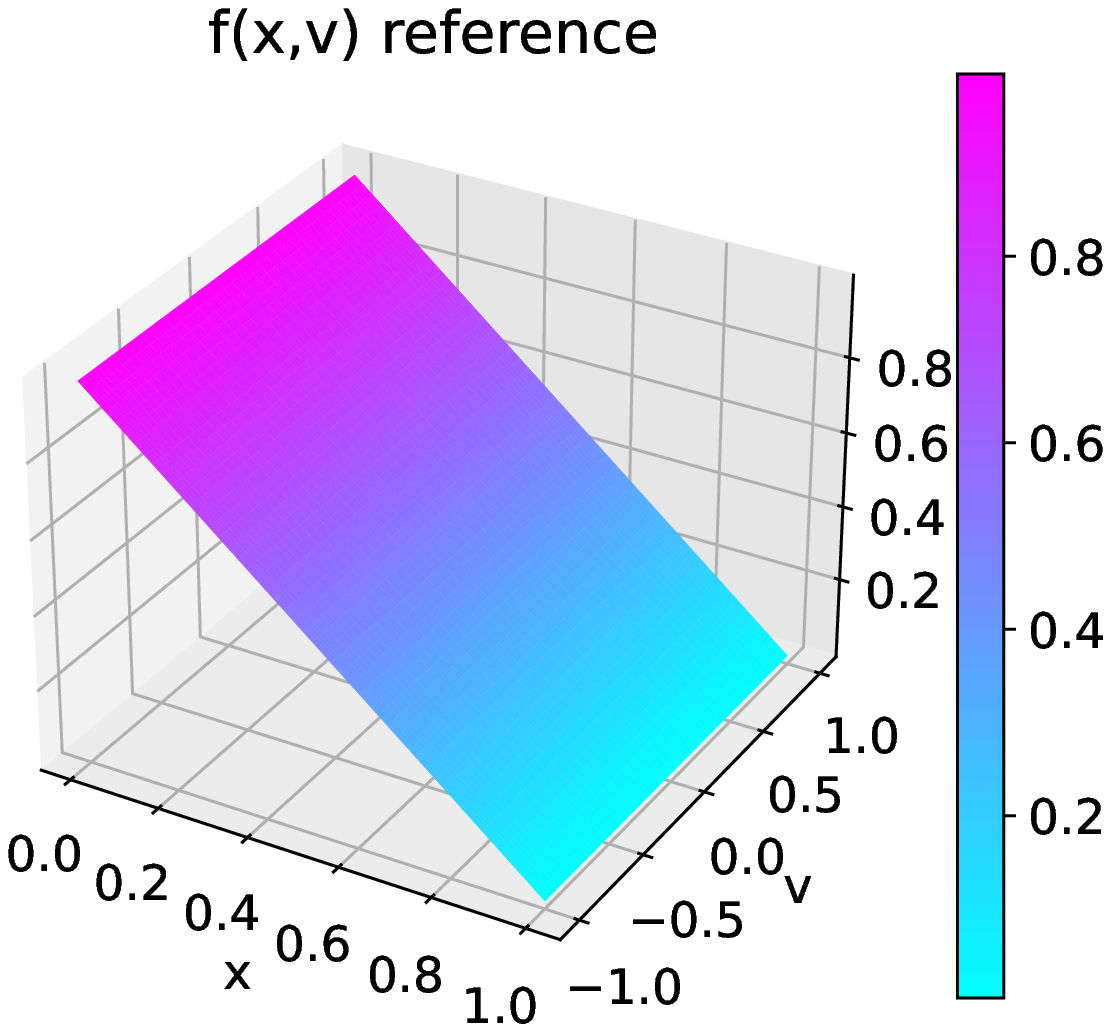}}
	{\includegraphics[width=0.3\textwidth]{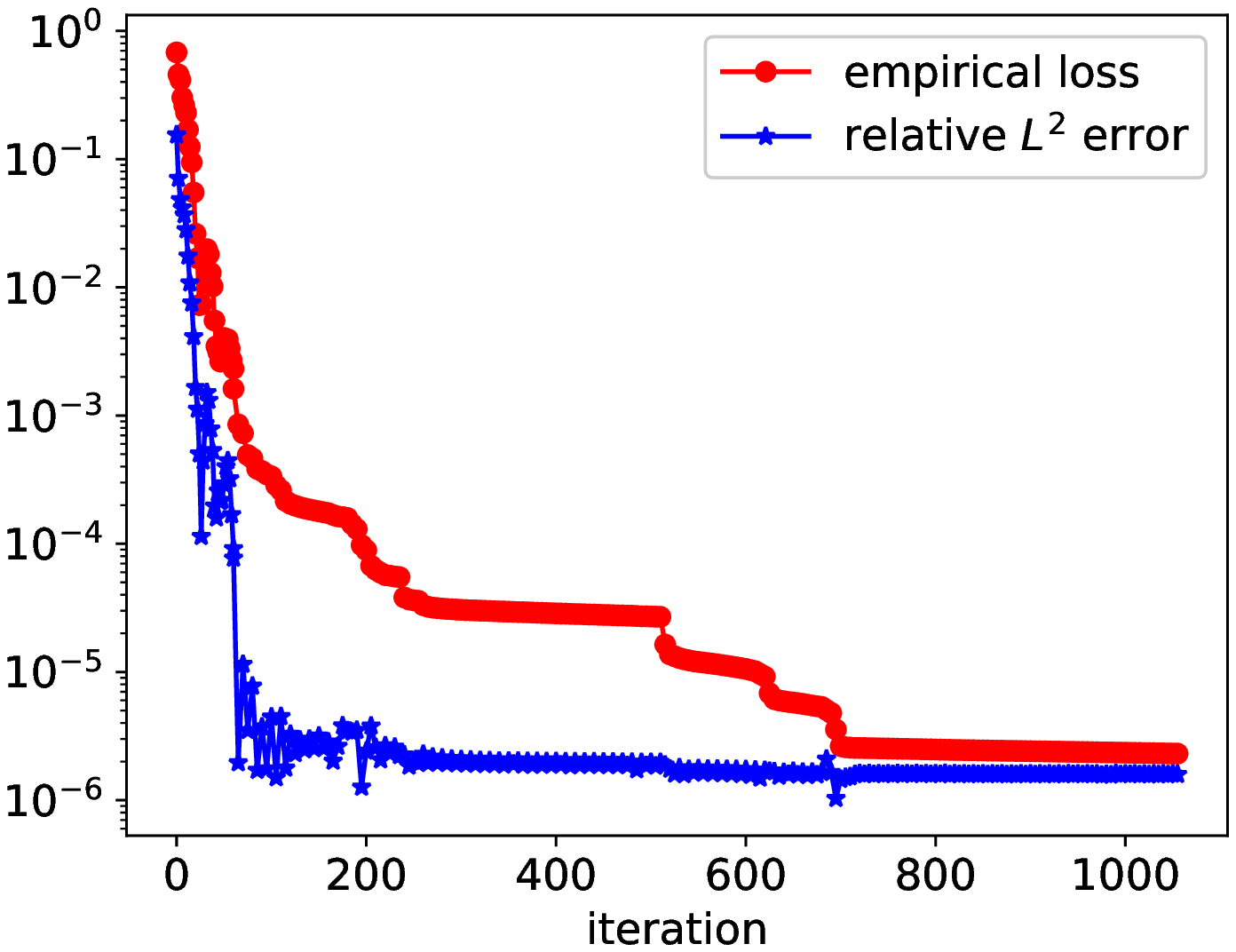}}
	\caption{Example~\ref{ex:homo} with $\eps = 1$ in top row and $\eps=0.001$ in bottom row. The left column is $f(x,v)$ prediction, the middle column is the reference $f(x,v)$, the right column is the empirical loss and relative $L^2$ error to reference solution. }
	\label{fig: rg_1d_epsi1}
\end{figure}

\begin{example}\label{ex:analytic}
	2D problem with $\bx \in [-1,1]^2,~ \bv = (\cos \alpha, \sin \alpha)$:
	\begin{equation*}  
		\begin{cases}
			\eps \boldsymbol{v} \cdot \nabla_{\boldsymbol{x}} f =  \frac{1}{2 \pi} \int_{|\boldsymbol{v}|=1} f(\bx, \bv) d \bv' -  f + \eps^2 G(\bx, \bv)  \,, \\
			f(-1, y, \alpha) =  e^{1-y}, ~ \alpha \in [0, \pi/2] \cup [3 \pi/2, 2 \pi]  \,, \\
			f(1, y ,\alpha) = e^{-1-y},  ~ \alpha \in [\pi/2, 3 \pi /2]  \,, \\
			f(x, -1 ,\alpha) = e^{1-x},  ~ \alpha \in [0, \pi ] \,, \\
			f(x, 1 ,\alpha) = e^{-1-x},  ~ \alpha \in [\pi, 2 \pi ] \,,
		\end{cases}
	\end{equation*}
	where $
	G(x,y,\alpha) = \frac{1}{\eps} (- \cos \alpha - \sin \alpha) e^{-x-y}.
	$
	This problem has an analytic solution
	$
	f(x,y,\alpha) = e^{-x-y}\,.
	$
	The numerical solutions for $\eps = 1$ and $\eps =10^{-3}$ are presented in Figure~\ref{fig:2d_analytic}, where the reference solution is the above analytic form. In both cases, the numerical parameters we use are: $n_l = 4$, $n_r = 30$, $N^r_x = 40$, $N^r_y=40$, $N^r_v=40$,  $N^b_v=40$, $N^b_x = 40$, $N^b_y = 40$ and $N^b_v=40$ for training. 
\end{example}
\begin{figure}[htbp]
	\centering
	{\includegraphics[width=0.3\textwidth]{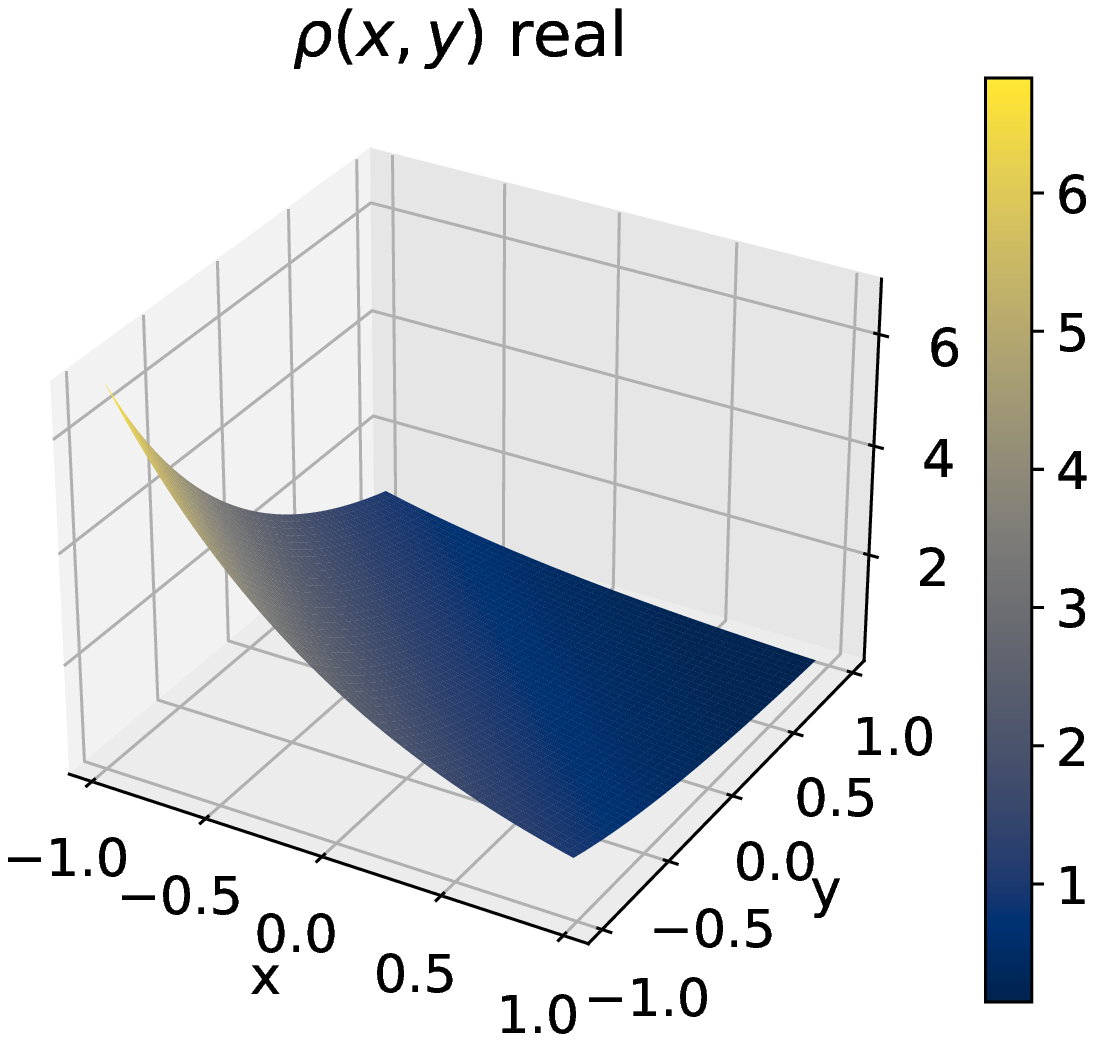}}
	{\includegraphics[width=0.3\textwidth]{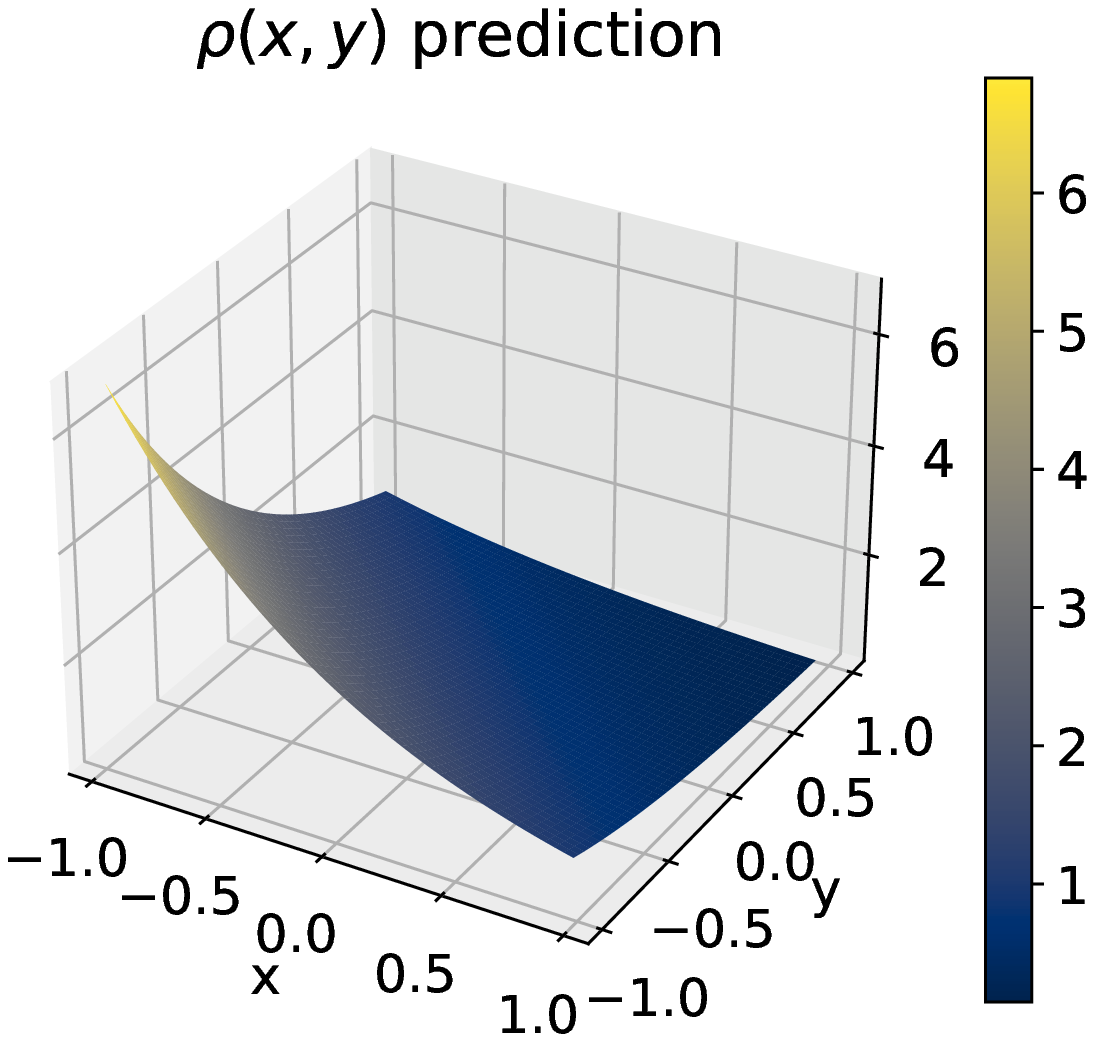}}
	{\includegraphics[width=0.3\textwidth]{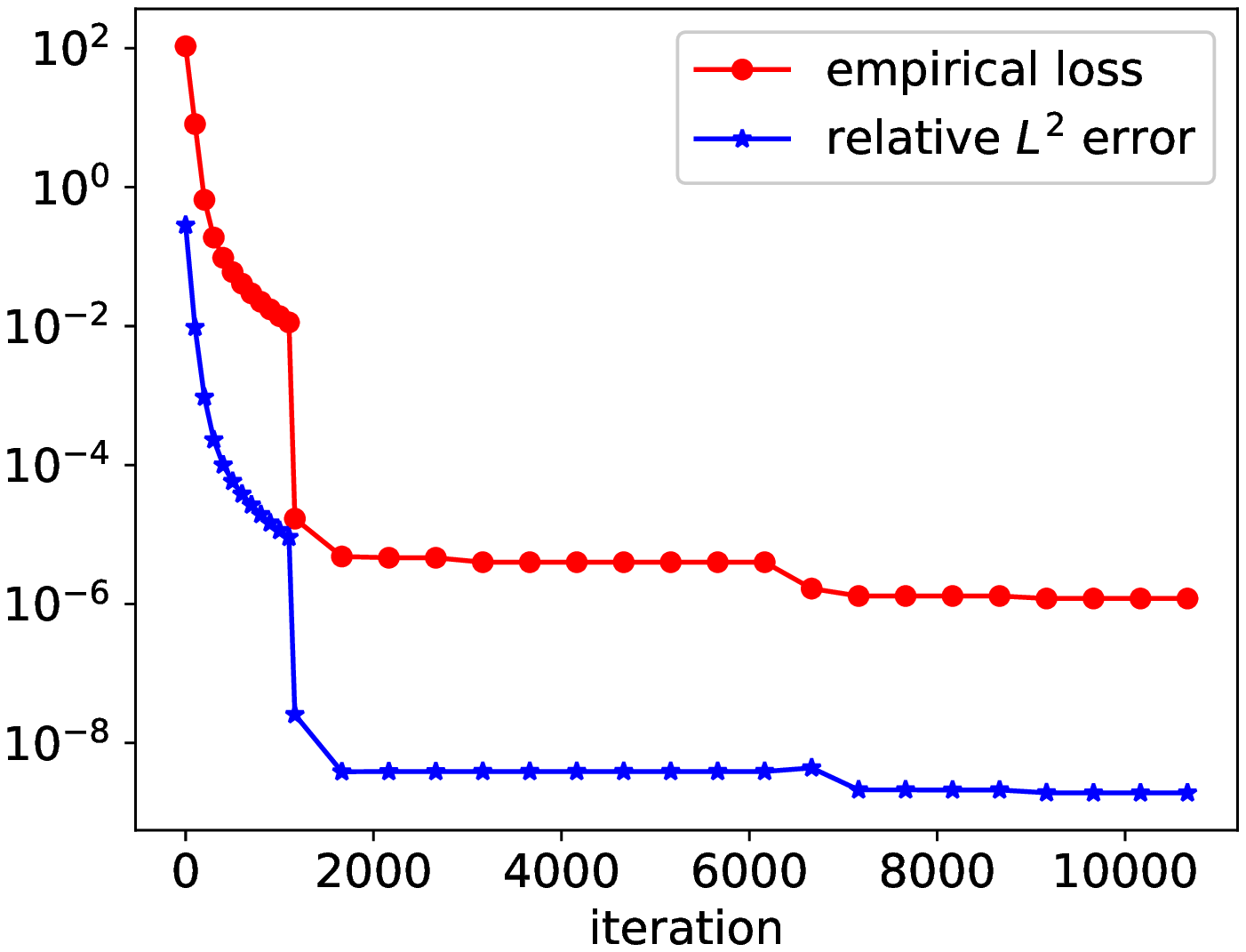}}
	{\includegraphics[width=0.3\textwidth]{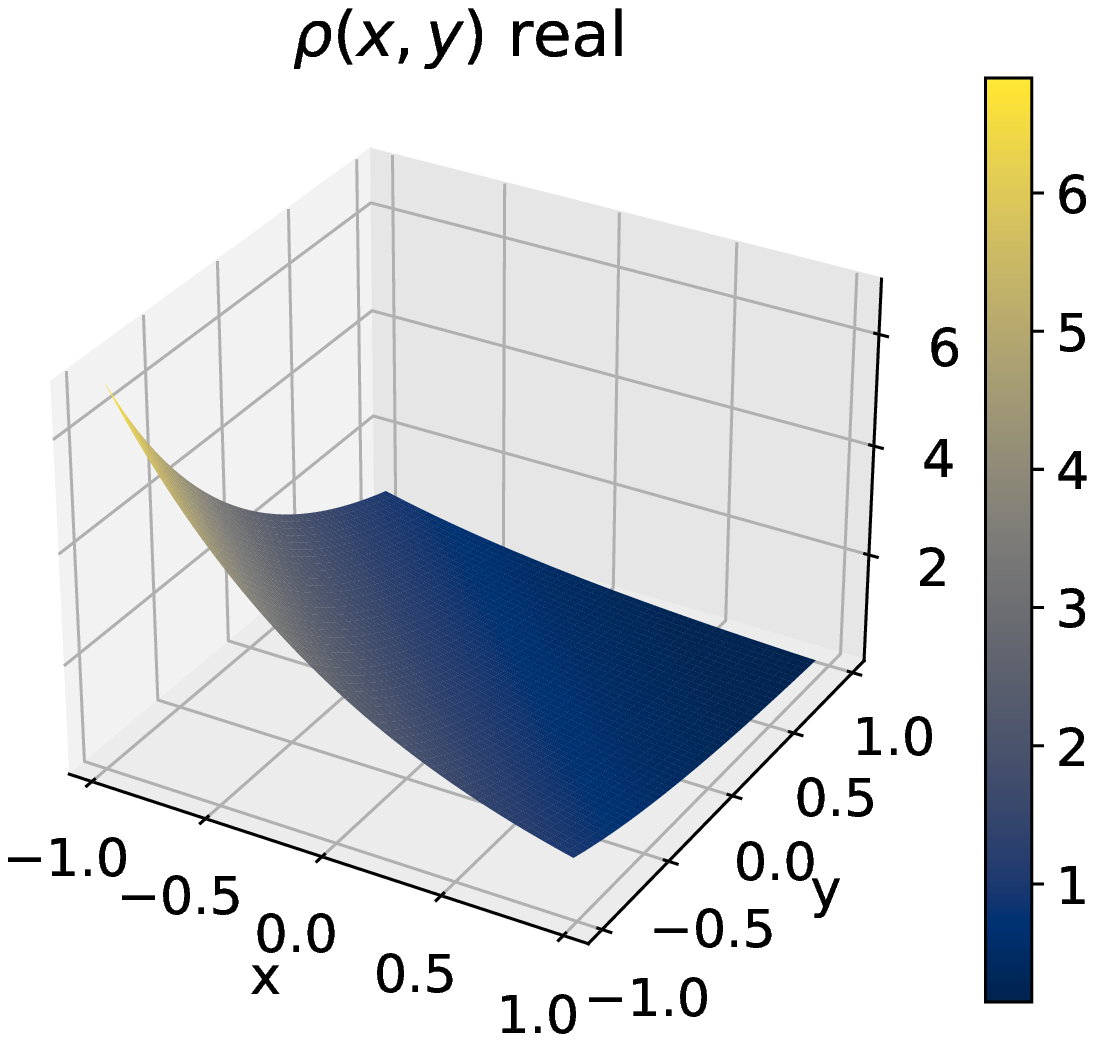}}
	{\includegraphics[width=0.3\textwidth]{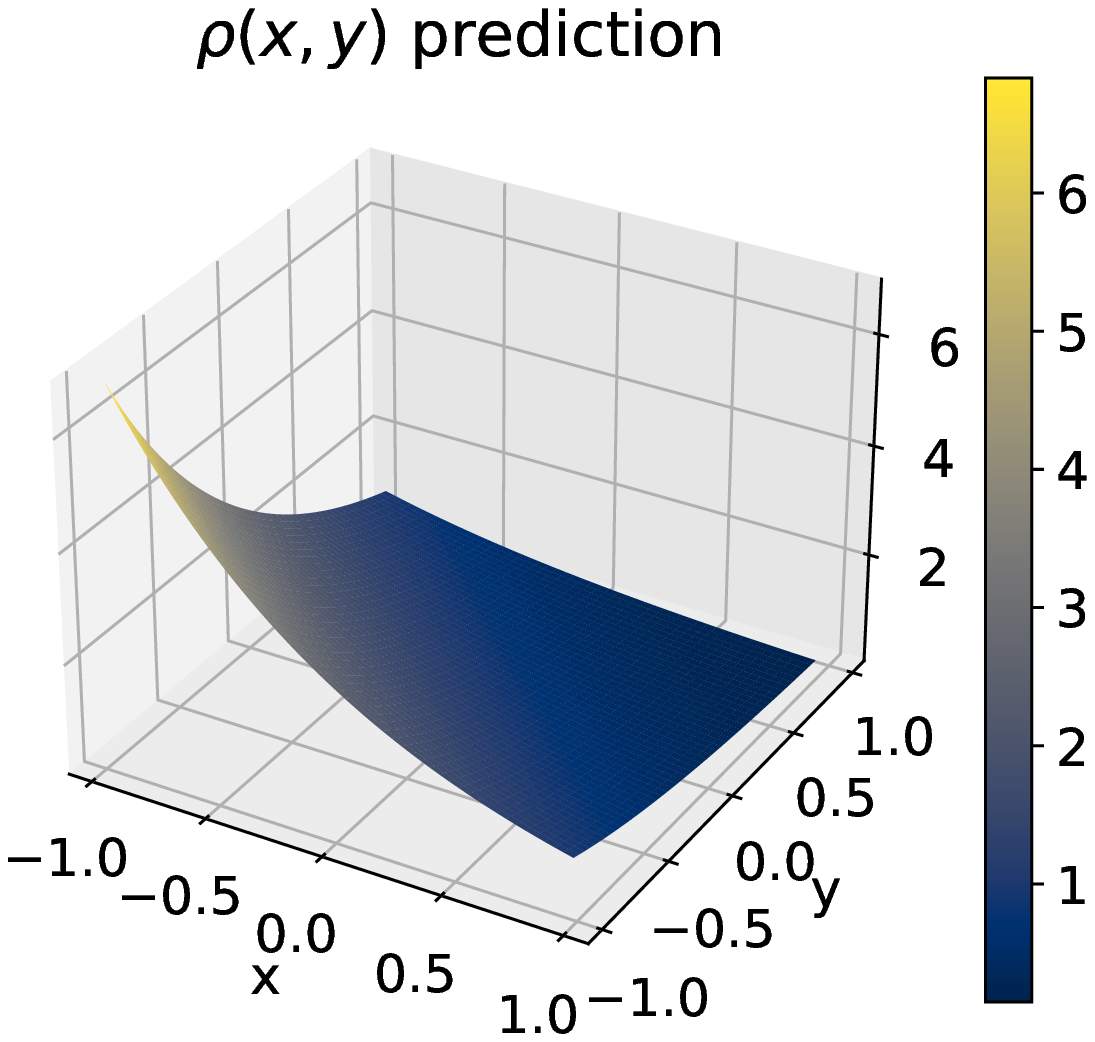}}
	{\includegraphics[width=0.3\textwidth]{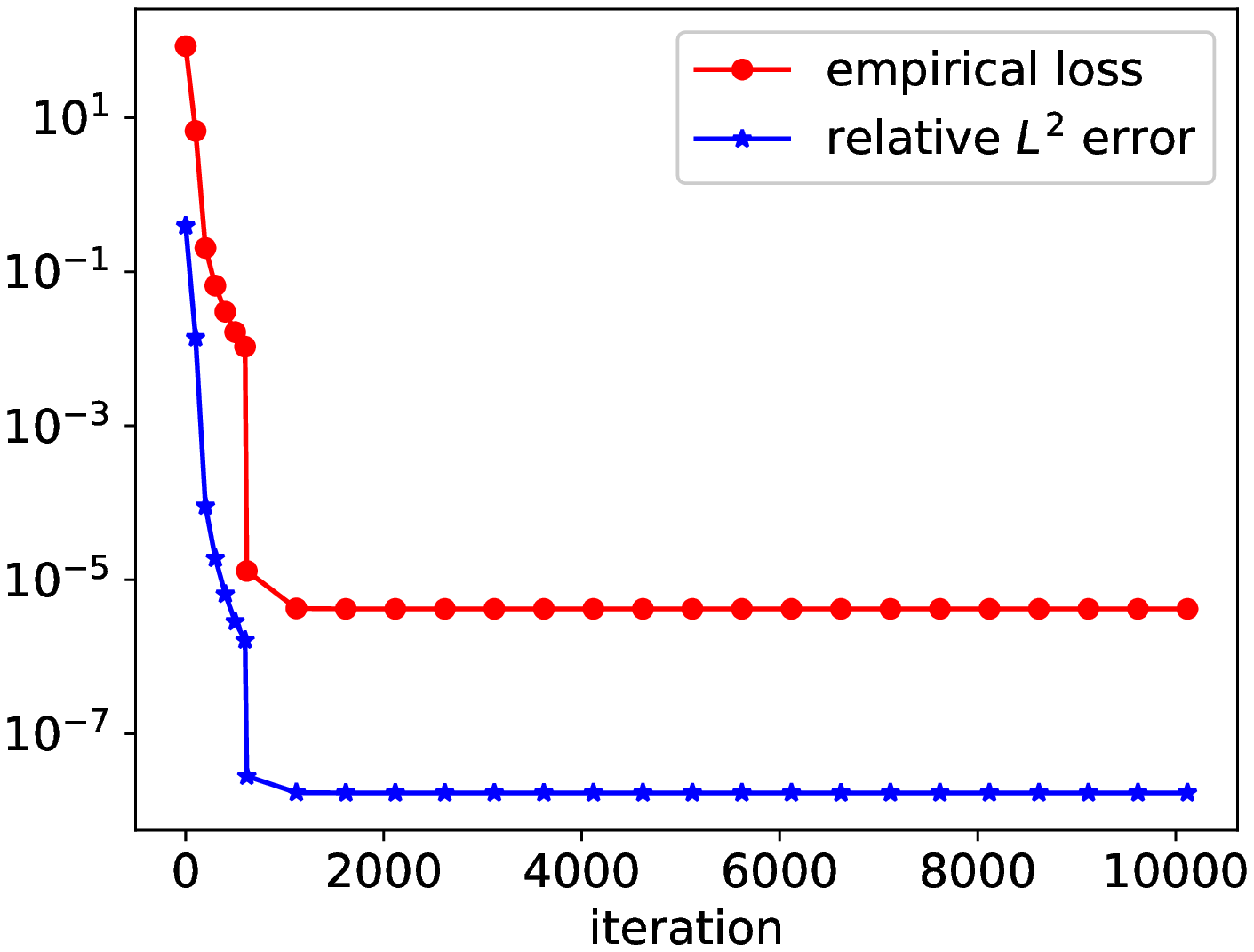}}
	\caption{Example~\ref{ex:analytic} with $\eps = 1$ in top row and $\eps = 0.001$ in bottom row. The left column is prediction of $\rho(x,y)$. The middle column is analytic $\rho(x,y)$. Right column is empirical loss and relative $L^2$ error.}
	\label{fig:2d_analytic}
\end{figure}


\begin{example}\label{ex:inhomo}
	1D problem with spatially heterogeneous scattering: 
	\begin{equation*} 
		\begin{cases}{}
			v \partial_x f = \sigma(x)  (\langle f \rangle -  f) \,, \\
			f(0, v>0)= 5, \qquad 
			f(1, v<0)= 0\,,
		\end{cases}
	\end{equation*}
	where $\sigma$ is a smooth varying function 
	$\sigma(x) = 1+{b}/{e^{-a(x-0.5)}}$.
\end{example}

In practice, we rewrite the equation as 
\begin{equation*} 
	\eps(x) v \partial_x f =   \langle f \rangle -  f \,, \quad \textrm{~with~} \quad \eps(x) = \frac{1+e^{-a(x-0.5)}}{b+1+e^{-a(x-0.5)}}\,.
\end{equation*}
Accordingly, our macro-micro decomposition reads $f(x,v) = \rho(x) + \eps(x) g(x,v)$, where $\rho$ and $g$ solve 
\begin{equation*} 
	\begin{cases}{}
		\average{v \partial_x(\eps(x)  g)} = 0 \, ,\\
		v \partial_x (\rho + \eps(x) g) + g= 0 \, ,\\
		\rho(0) + \eps(0) g(0, v>0)= 5, \quad 
		\rho(1) + \eps(1) g(1, v<0)= 0.
	\end{cases}
\end{equation*}

Choosing $a=10$ and $b=20$, we plot the shape of $\sigma(x)$ and gather the corresponding numerical solutions in Figure~\ref{fig:1d_epsix}. Here the neural network is constructed using $n_l = 4$, $n_r = 50$, $N^r_x = 80$, $N^r_v=60$,  $N^b_v=60$; and trained with initial learning rate $0.001$ for Adam and decrease by a factor of $0.95$ after every 2000 steps. The reference solution is obtained by a finite difference method with $N_x = 200$ and $N_v = 80$. 
As expected, a good match to the reference solution is observed, and a good control of relative $L^2$ error is obtained.

\begin{figure}[htbp]
	\centering
	{\includegraphics[width=0.45\textwidth]{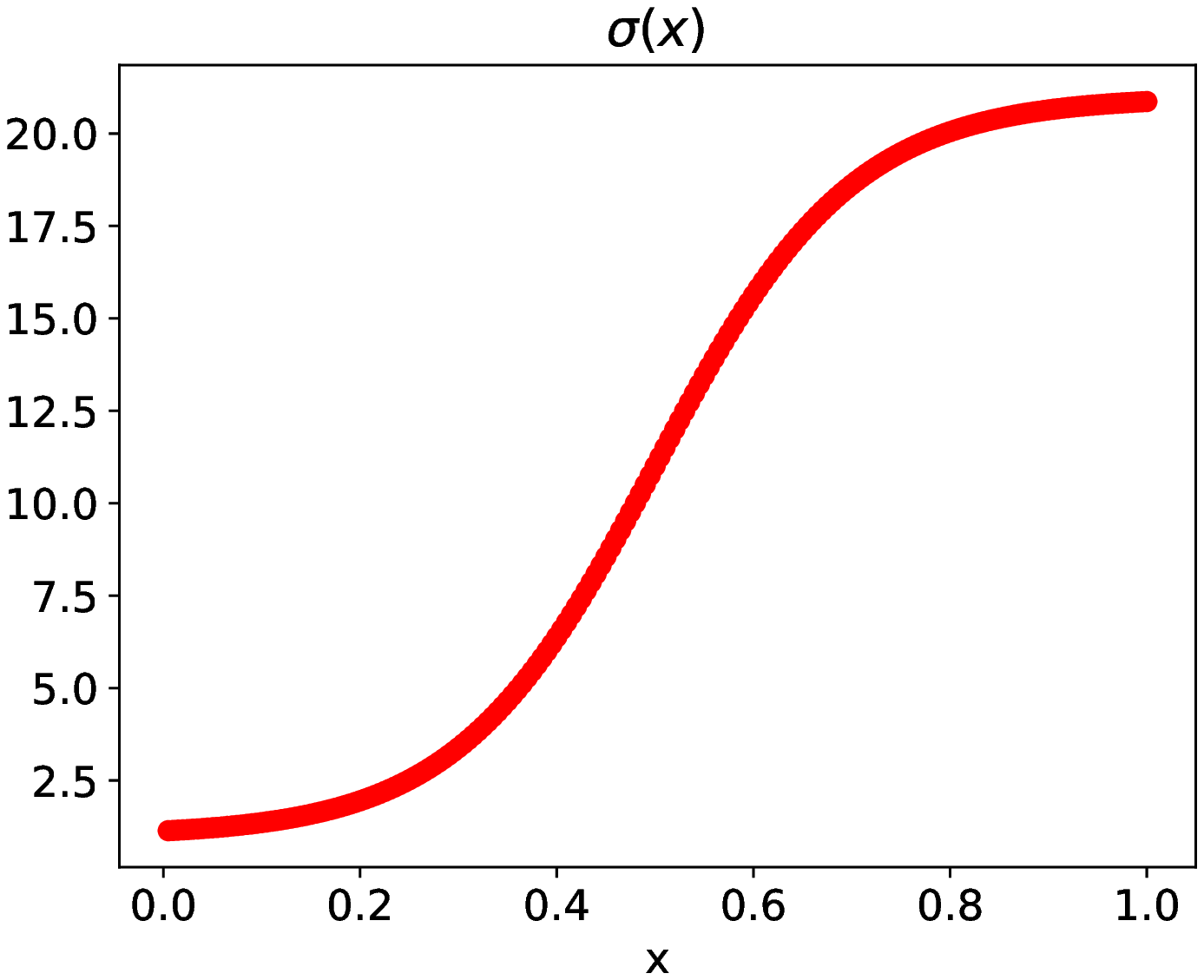}}   
	{\includegraphics[width=0.45\textwidth]{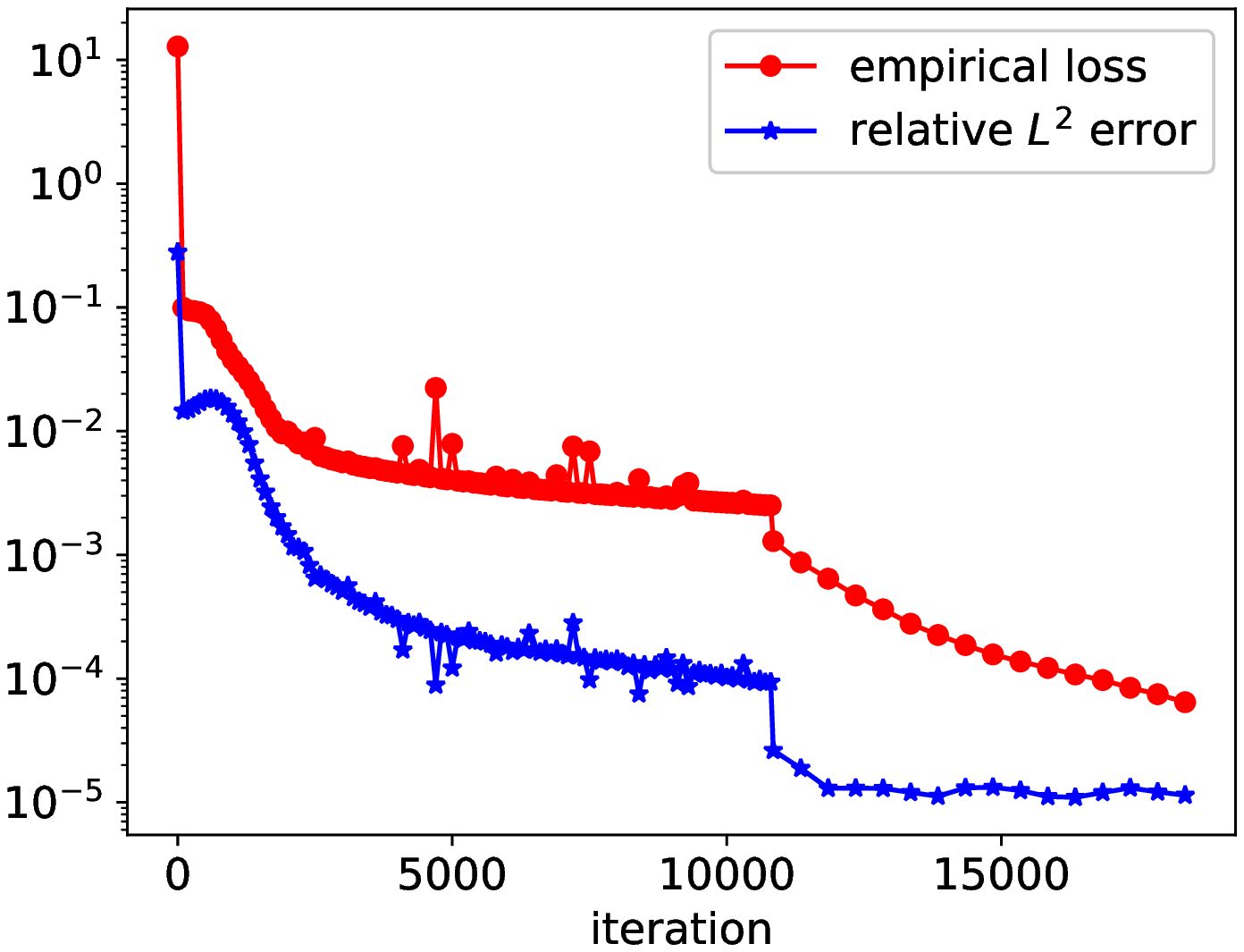}}
	{\includegraphics[width=0.45\textwidth]{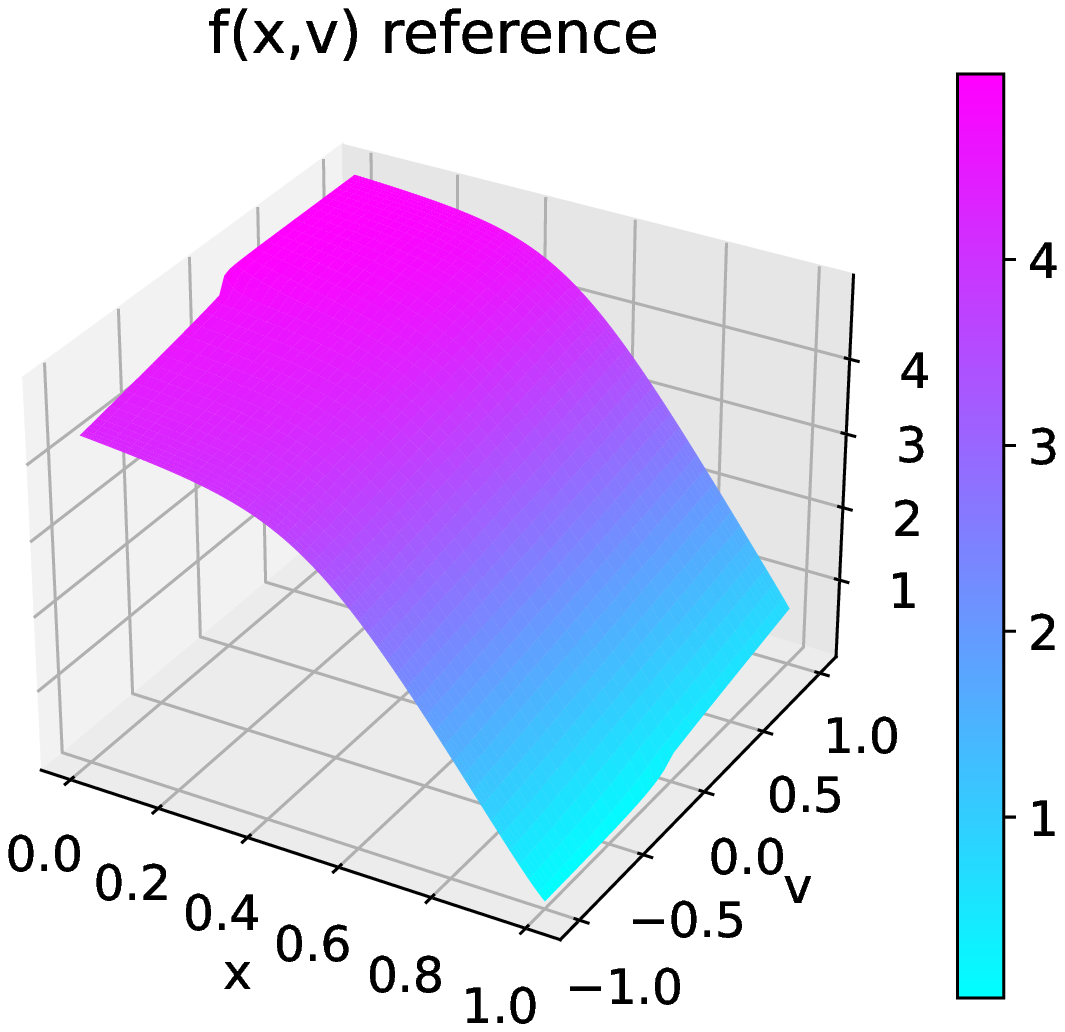}}
	{\includegraphics[width=0.45\textwidth]{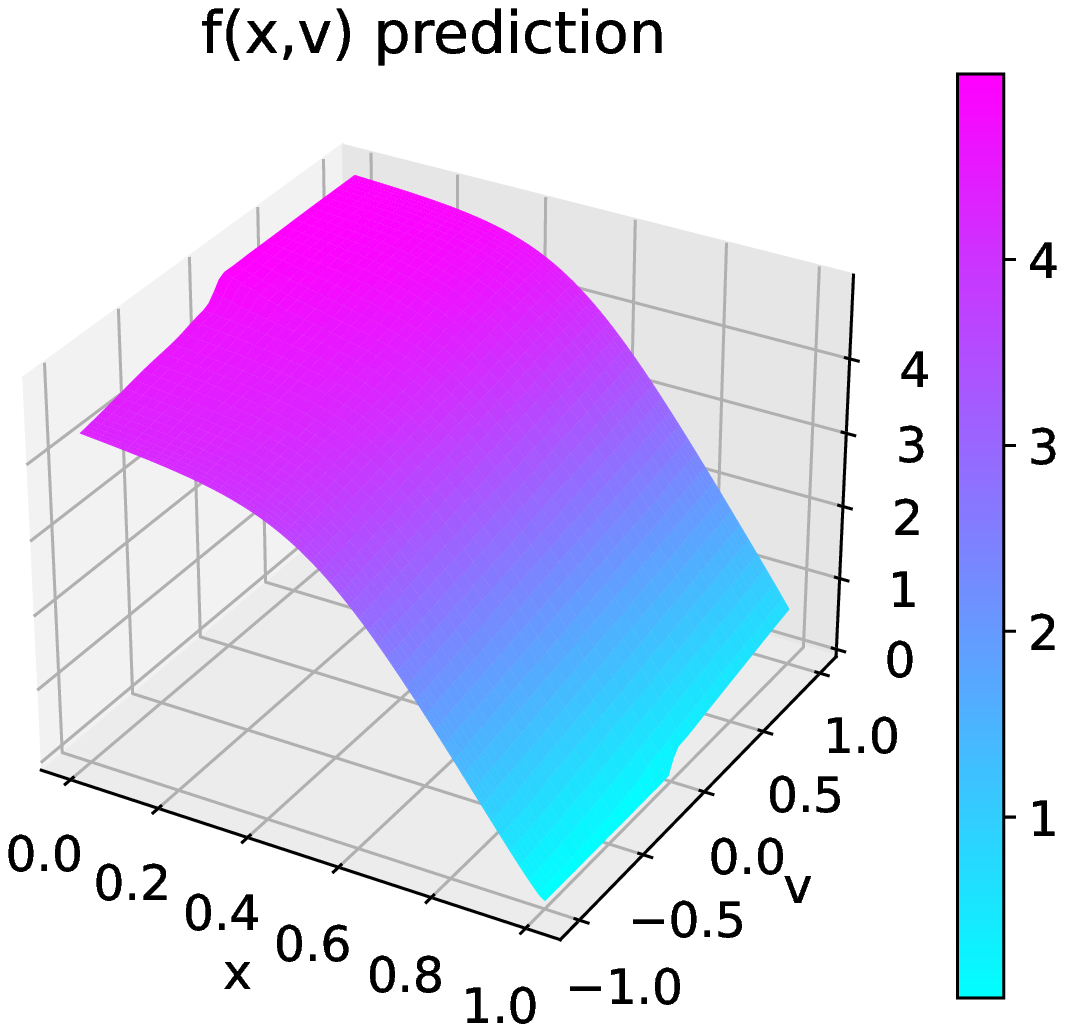}}
	
	\caption{Example~\ref{ex:inhomo} with $a=10$ and $b=20$.  The top left is the plot of $\sigma$, top right is the empirical loss and relative $L^2$ error, bottom left is referenced $f(x,v)$, and bottom right is the predicted $f(x,v)$ by neural network.  }
	\label{fig:1d_epsix}
\end{figure}

\begin{example}\label{ex:aniso}
	2D RTE with an-isotropic scattering: 
	\begin{equation*} 
		\begin{cases}
			\eps \boldsymbol{v} \cdot \nabla_{\boldsymbol{x}} f =  \int_{|\boldsymbol{v}|=1} K(\bv, \bv') f(\bx, \bv) d \bv' - f, \quad \bx \in [-1,1]^2, \bv = (\cos \alpha, \sin \alpha)  \, \\
			f(-1, y, \alpha) =  (1-y^2) , ~ \alpha \in [0, \pi/2] \cup [3 \pi/2, 2 \pi]  \, \\
			f(1, y ,\alpha) = 0,  ~ \alpha \in [\pi/2, 3 \pi /2]  \, \\
			f(x, -1 ,\alpha) = 0,  ~ \alpha \in [0, \pi ] \, \\
			f(x, 1 ,\alpha) = 0,  ~ \alpha \in [\pi, 2 \pi ] \,,
		\end{cases}
	\end{equation*}
	with Henyey-Greenstein scattering
	kernel: 
	\begin{equation*}
		K(\bv, \bv') = \frac{1-h^2}{2 \pi (1+h^2 - 2h \bv \cdot \bv')}\,, \quad h \in (0,1)\,.
	\end{equation*}
\end{example}


The numerical solutions with $\eps=1$ and $\eps=0.001$ are presented in Figure~\ref{fig: hsp_2d_aniso_1}. Here the neural network is constructed with  $n_l = 4$, $n_r = 30$, $N^r_x = 40$, $N^r_y=40$, $N^r_v=40$,  $N^b_v=40$, $N^b_x = 40$, $N^b_y = 40$ and $N^b_v=40$; and the reference solution is obtained by a finite difference method with $N_x = 60$, $N_y = 60$, $N_v=40$.
\begin{figure}[htbp]
	\centering
	{\includegraphics[width=0.3\textwidth]{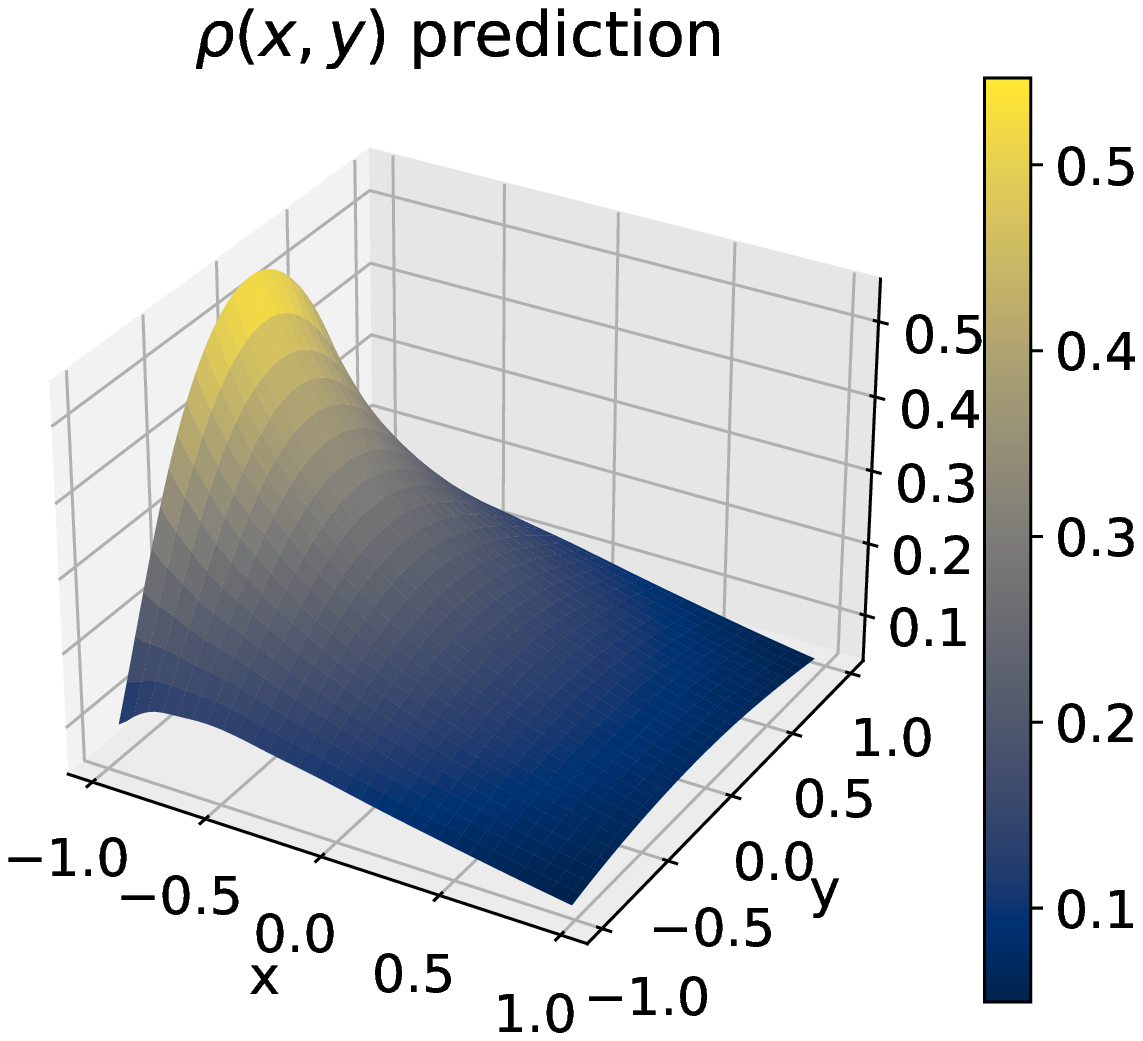}}
	{\includegraphics[width=0.3\textwidth]{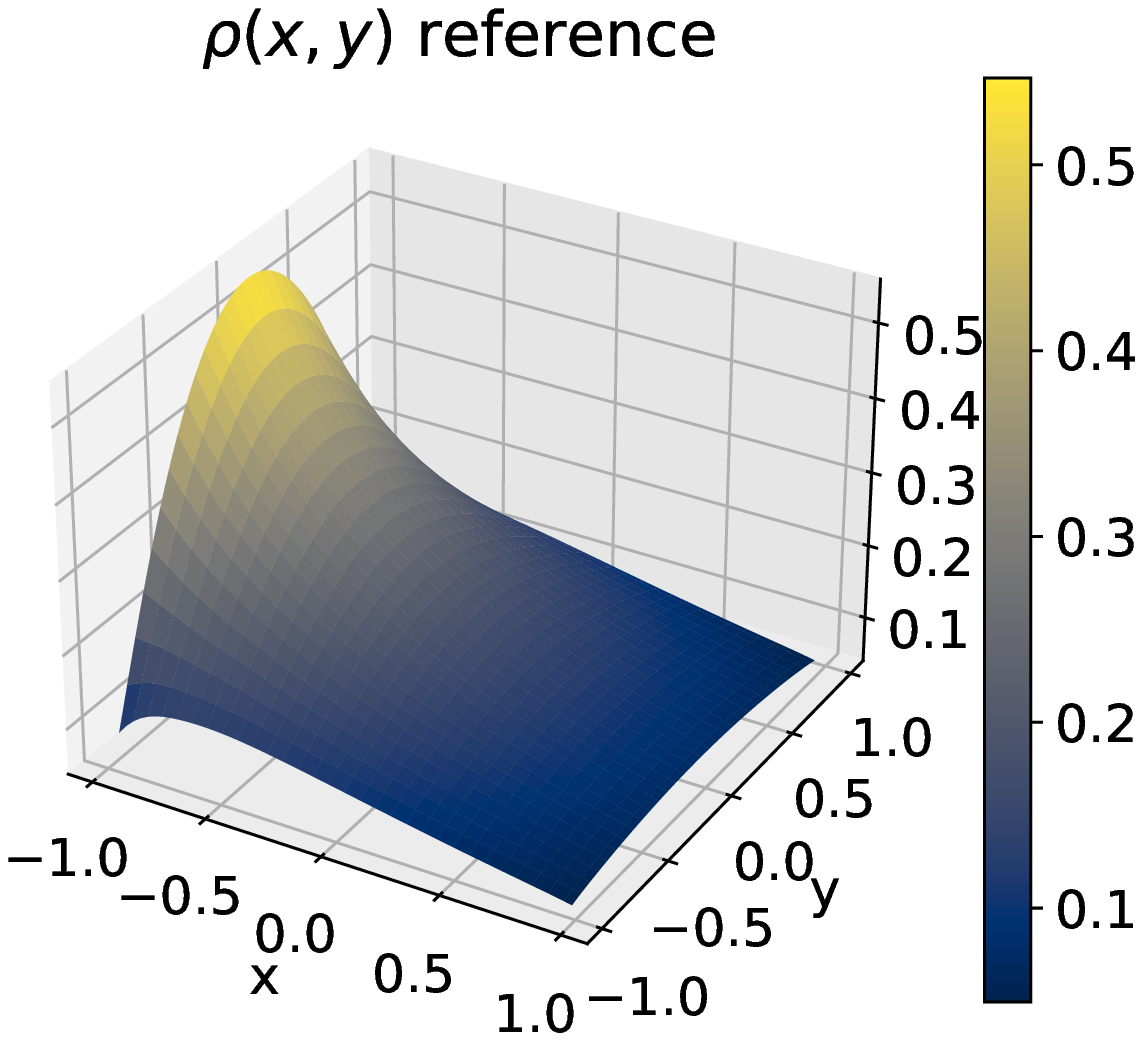}}
	{\includegraphics[width=0.3\textwidth]{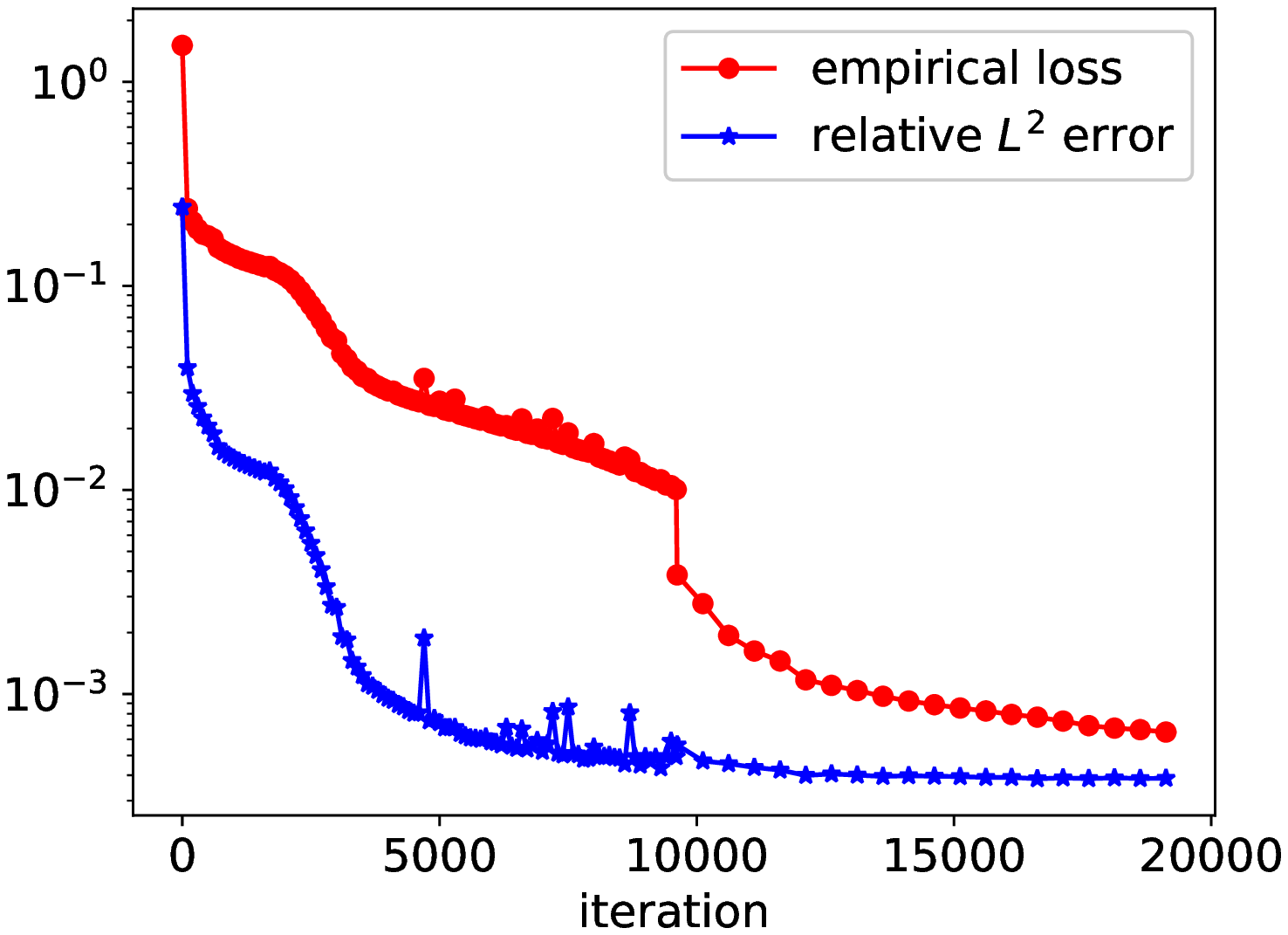}}
	{\includegraphics[width=0.3\textwidth]{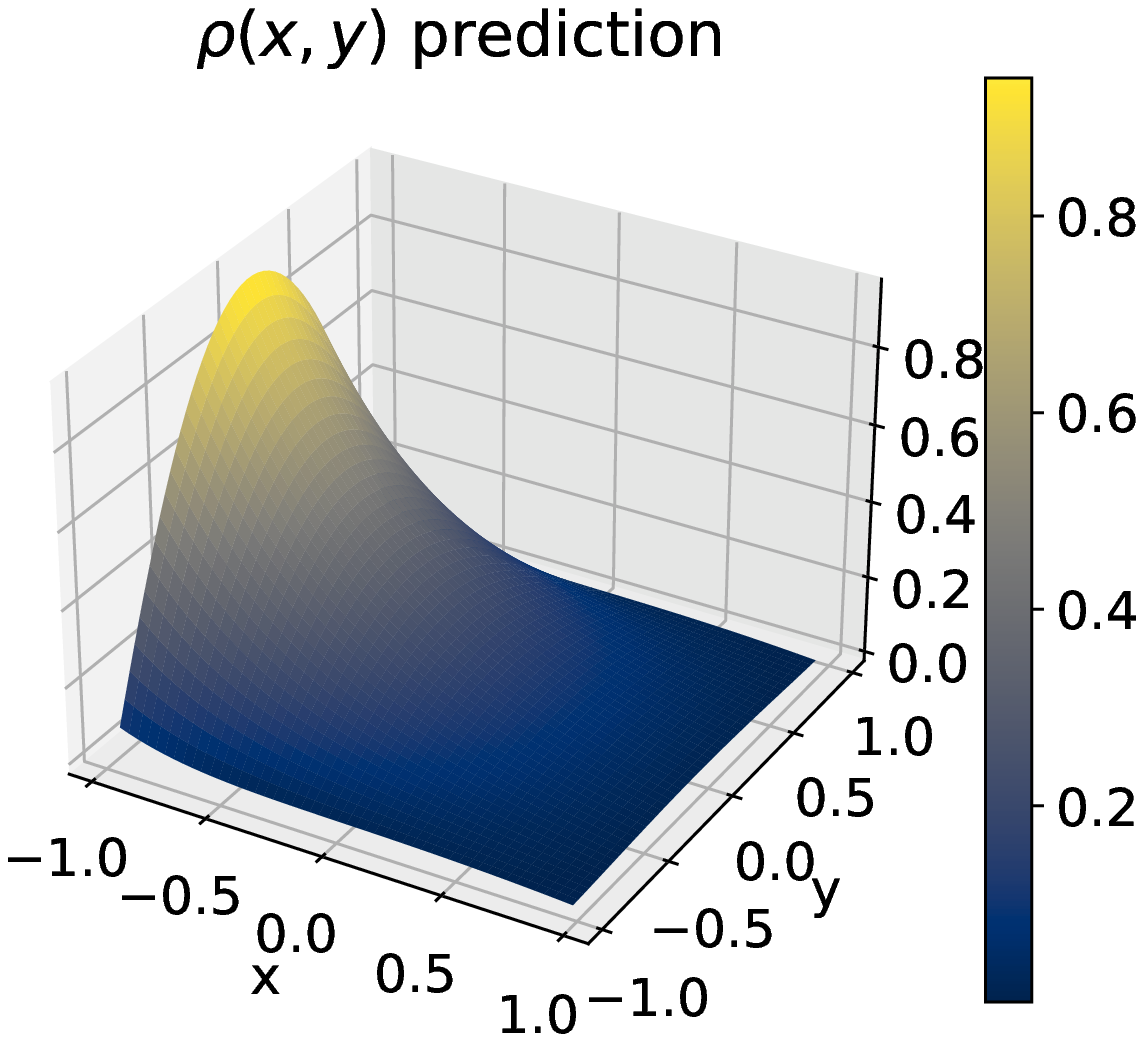}}
	{\includegraphics[width=0.3\textwidth]{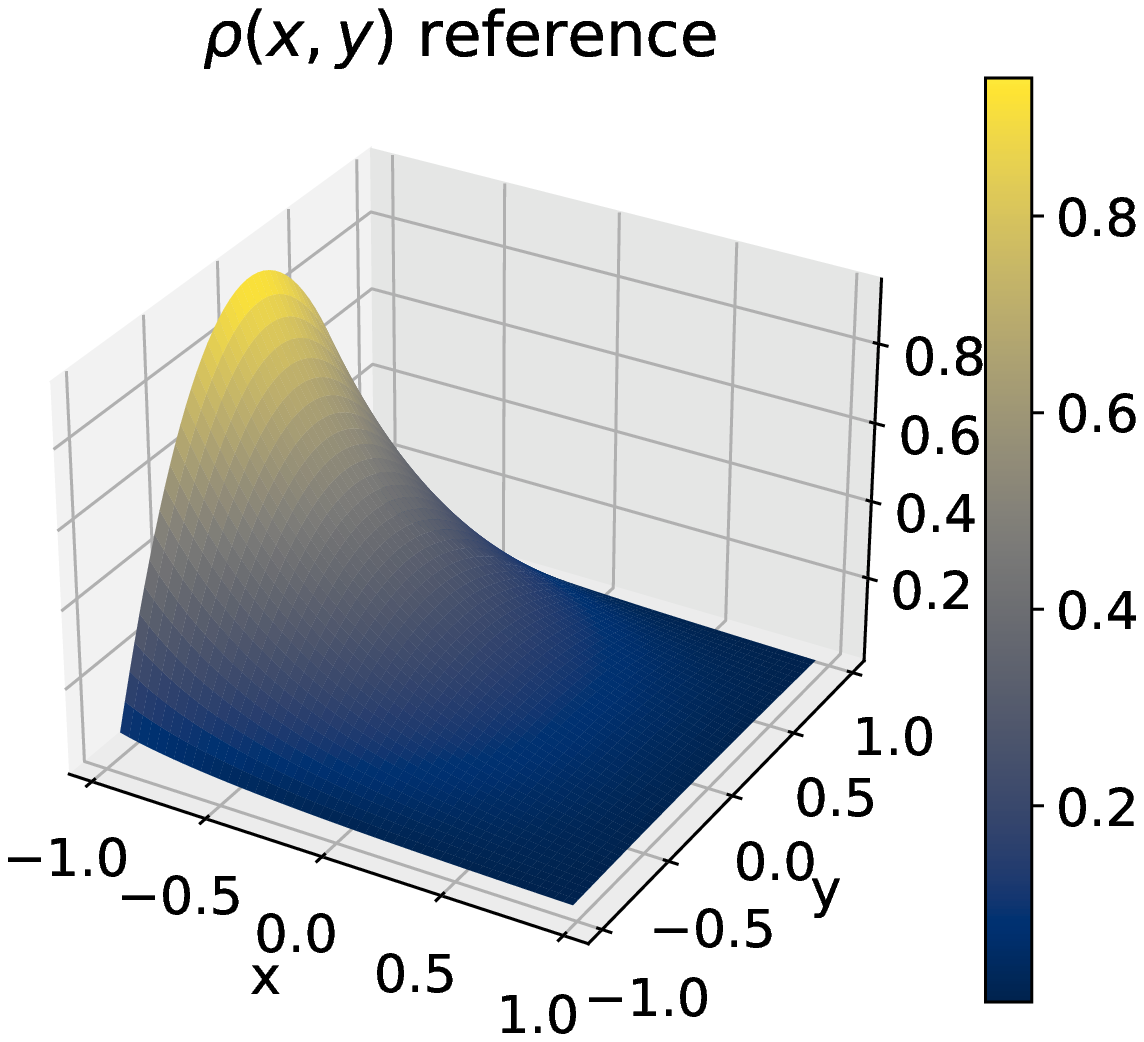}}
	{\includegraphics[width=0.3\textwidth]{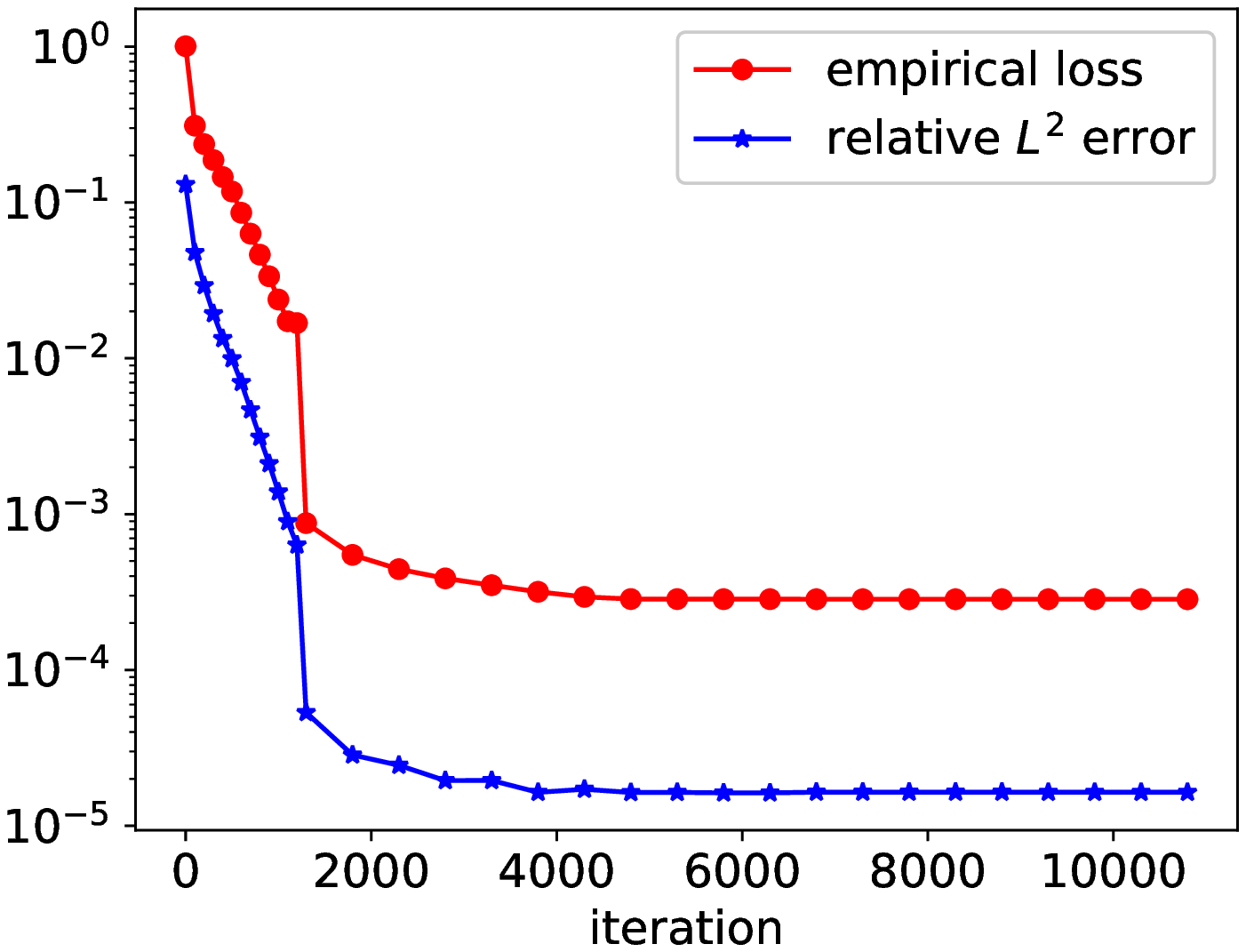}}
	\caption{Top and bottom are Example~\ref{ex:aniso} with $\eps = 1$ and $\eps = 0.001$ respectively. }
	\label{fig: hsp_2d_aniso_1}
\end{figure}

\subsection{1D RTE with boundary layer}\label{sec:1D2}
Here we consider a one dimensional example with velocity dependent boundary condition. With a velocity-dependent boundary term, a boundary layer is present in the solution of RTE when  $\eps$ is small. To find such solution, we first solve the half space problem. 
\begin{example}\label{ex:1dhsp}
	1D half space problem: 
	\begin{equation} \label{eqn:hsp1d}
		\begin{cases}{}
			v \partial_z f_\bl(z, v) = \average{f_\bl} - f_\bl \,, \\
			f_\bl(0,v) = 5 \sin(v) .
		\end{cases}
	\end{equation}
\end{example}

For this problem, we know that its solution $f_\bl(z,v)$ admits an analytical limit 
\begin{equation} \label{BLinf}
	f_\bl^{\infty} = \frac{\sqrt{3}}{2} \int_0^1 5 \sin(v) H(v) v \rd v 
\end{equation}
from the Chandrasekhar H-function, which satisfies
\[
\frac{1}{H(v)} = \int_0^1 \frac{H(w)}{2(v+w)}w \rd w\,.
\]
Additionally,  the reflection boundary condition has the form
\begin{equation} \label{eqn1010}
	f_\bl(0,v<0) = \half H(v) \int_0^1 5 \sin(w) \frac{H(w)}{w+v} w \rd w\,.
\end{equation}
In Figure~\ref{fig: hsp_1d}, we plot the numerical solution to $\eqref{eqn:hsp1d}$, and compare its reflection boundary with \eqref{eqn1010} with good agreement. 
\begin{figure}[htbp]
	\centering
	{\includegraphics[width=0.4\textwidth]{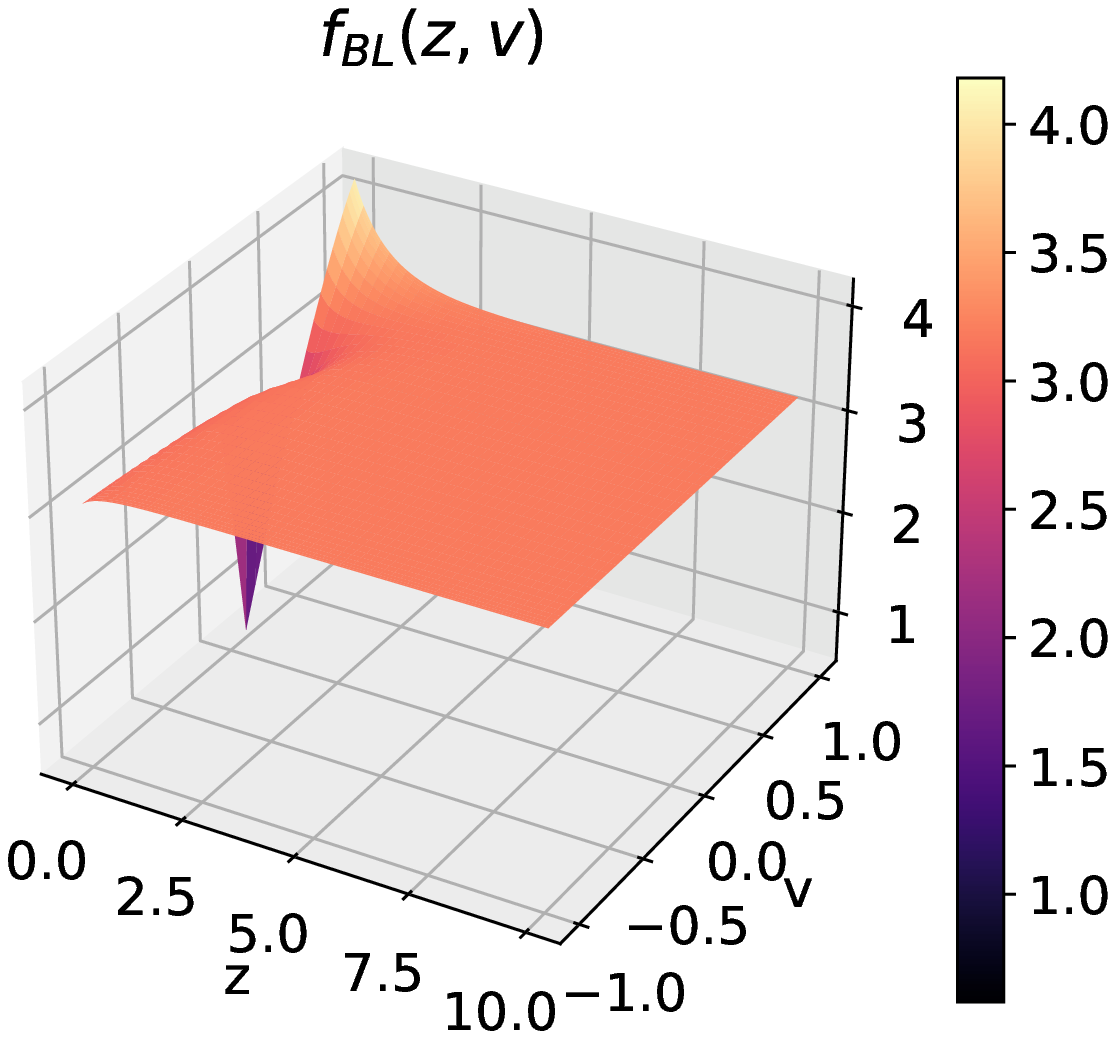}}
	{\includegraphics[width=0.4\textwidth]{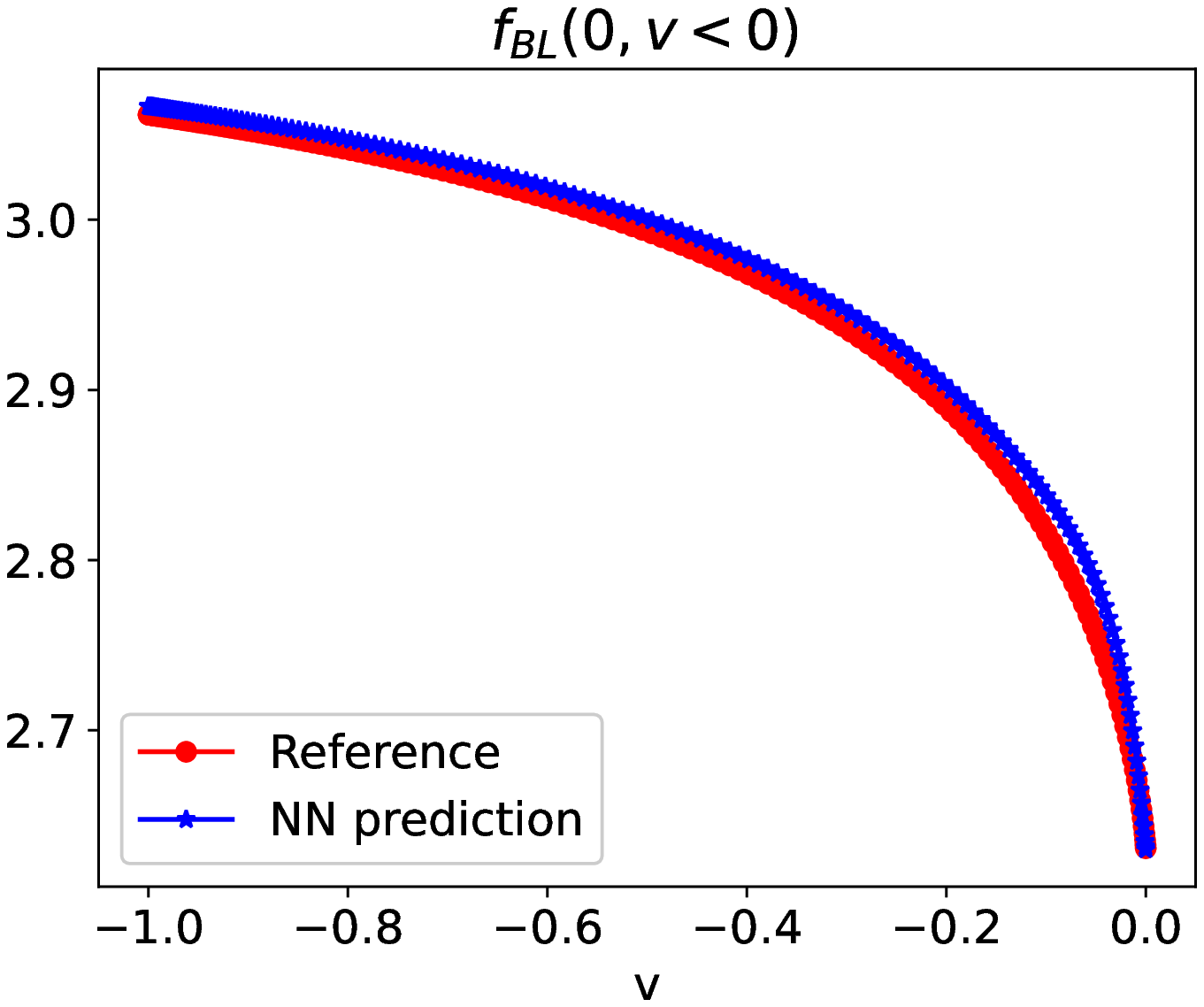}}
	\caption{Example~\ref{ex:1dhsp}, here we use $nl = 4$, $nr = 50$, $N^r_z = 400$, $N^r_v=40$ and $N^b_v=60$. The left is $f_\bl(x,v)$ prediction, and the right is the comparison between \eqref{eqn1010} and the prediction by Neural network. The numerical approximation of the limit constant $f^{nn}_\bl(10,0) \approx 3.1919$ and the exact $f_\bl^{\infty} \approx 3.1889$.}
	\label{fig: hsp_1d}
\end{figure}

\begin{example}\label{ex:transp}
	We then proceed to solve the transport equation:
	\begin{equation} \label{eqn:1d_iso_bl}
		\begin{cases}{}
			\eps v \partial_x f = \langle f \rangle - f \,, \\
			f(0, v>0)= 5 \sin v, \quad 
			f(1, v<0)= 0.
		\end{cases}
	\end{equation}
\end{example}

\begin{figure}[htbp]
	\centering
	{\includegraphics[width=0.35\textwidth]{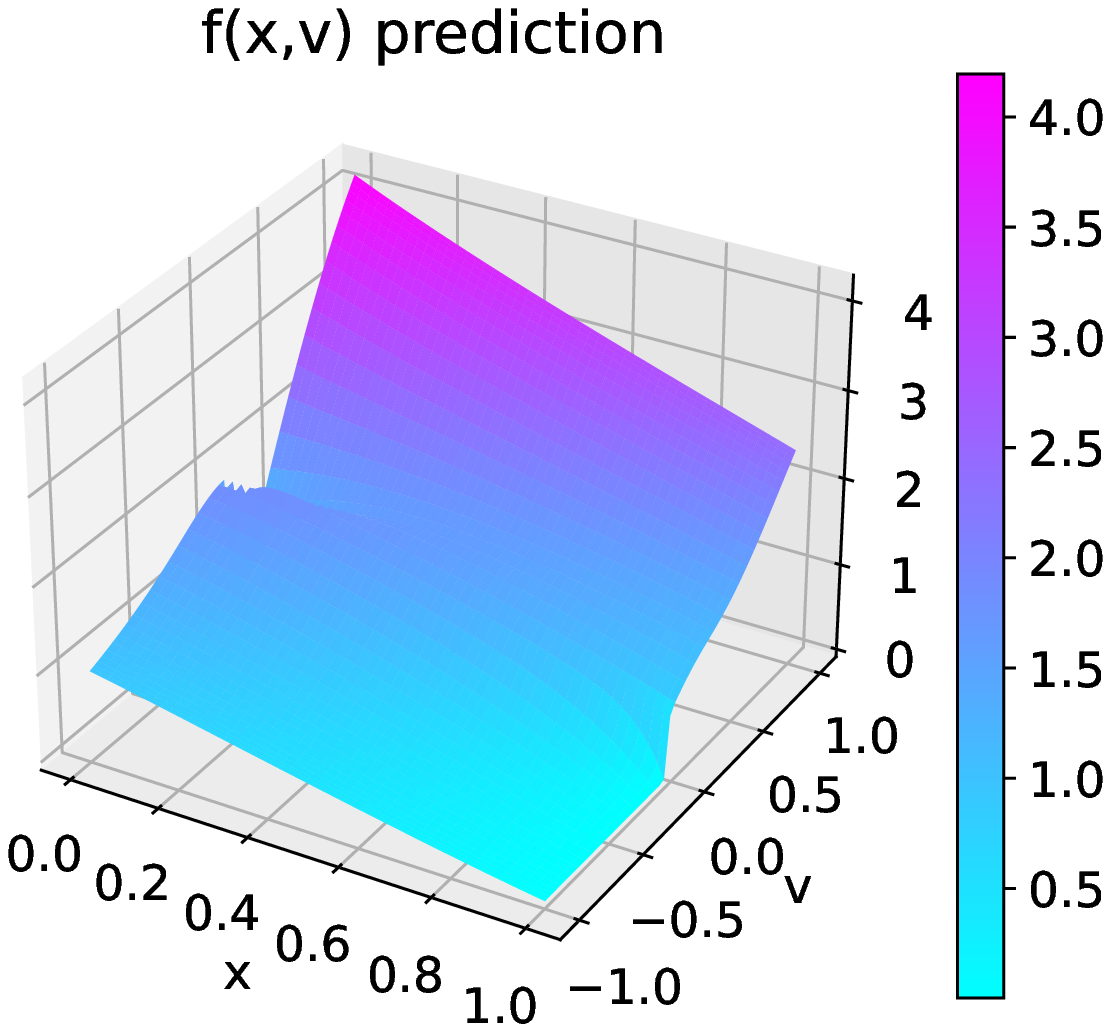}}
	{\includegraphics[width=0.35\textwidth]{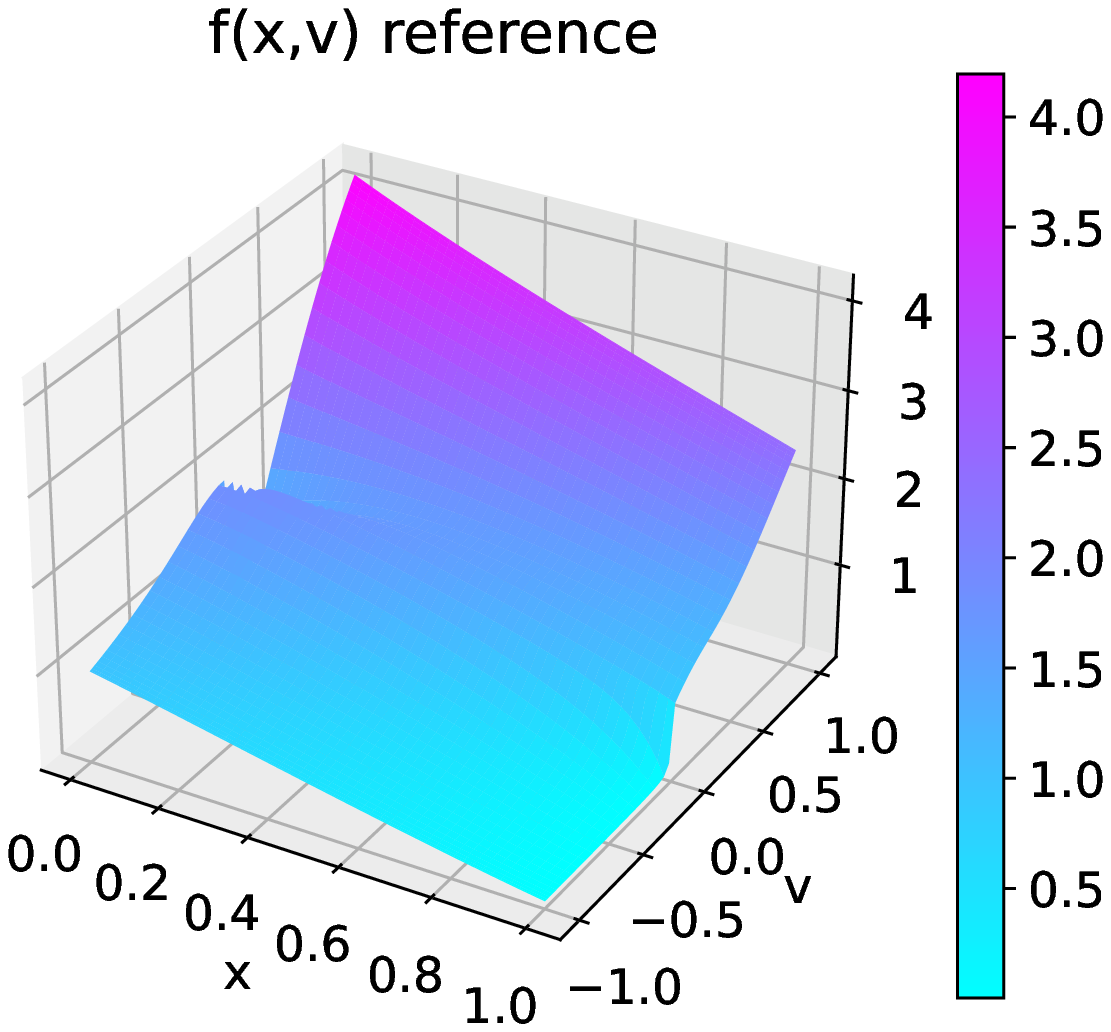}}
	{\includegraphics[width=0.35\textwidth]{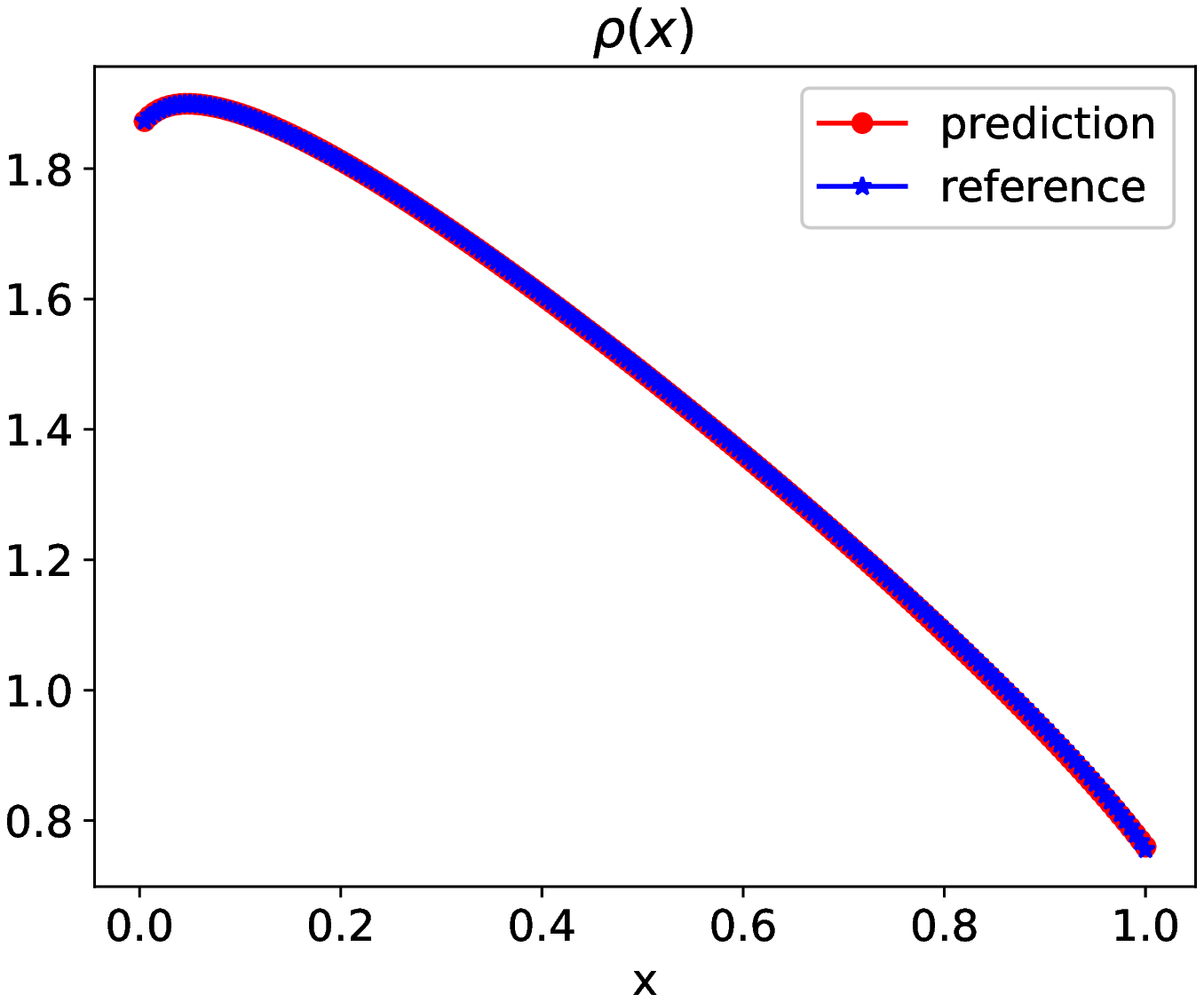}}
	{\includegraphics[width=0.35\textwidth]{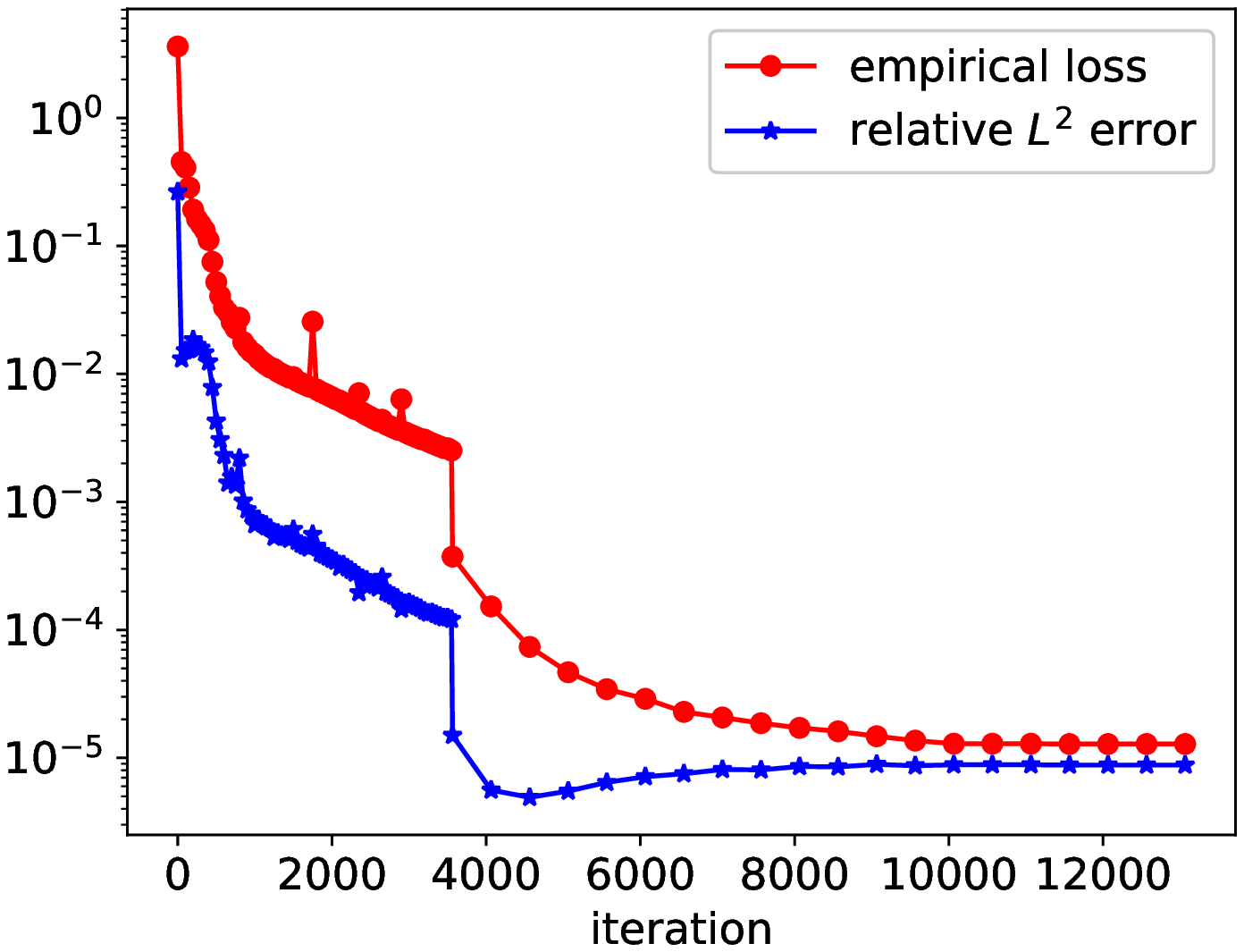}}
	\caption{Example~\ref{ex:transp} with $\eps = 1$. we use $nl = 4$, $nr = 50$, $N^r_x = 80$, $N^r_v=60$ and $N^b_v = 60$ for training. Compute the reference solution by finite difference method on test set with $N_x=200$, $N_v=80$. The top left is $f(x,v)$ prediction, the top right is the reference $f(x,v)$, bottom left is comparison of $\rho(x)$ and bottom right is the empirical loss and relative $L^2$ error vs  number of iterations.  }
	\label{fig: hsp_1d_bl_epsi1}
\end{figure}

When $\eps = 1$, we obtain the neural network prediction by using the macro-micro decomposition. The results are collected in Figure~\ref{fig: hsp_1d_bl_epsi1}, where the reference solution is obtained by solving \eqref{eqn:1d_iso_bl} with a finite difference method on a uniform mesh. On the other hand,  when $\eps  = 10^{-3}$, we obtain the neural network prediction by the macro-micro-boundary layer decomposition. For comparison, we construct two reference solution. One is obtained by the same finite difference method but on a non-uniform mesh in $x$, with 150 points in $[0, \eps)$ and 50 points in $[\eps,1]$. The other is obtained by solving the diffusion limit, whose boundary condition is computed via the H-function \eqref{BLinf}. In this specific example, we have $f_\bl^\infty = 3.188$, and therefore the limit density is $\rho_0 (x) = 3.188(1-x)$. The comparisons with good agreement are displayed in Figure~\ref{fig: hsp_1d_bl_epsizpzz1}.

\begin{figure}[htbp]
	\centering
	{\includegraphics[width=0.35\textwidth]{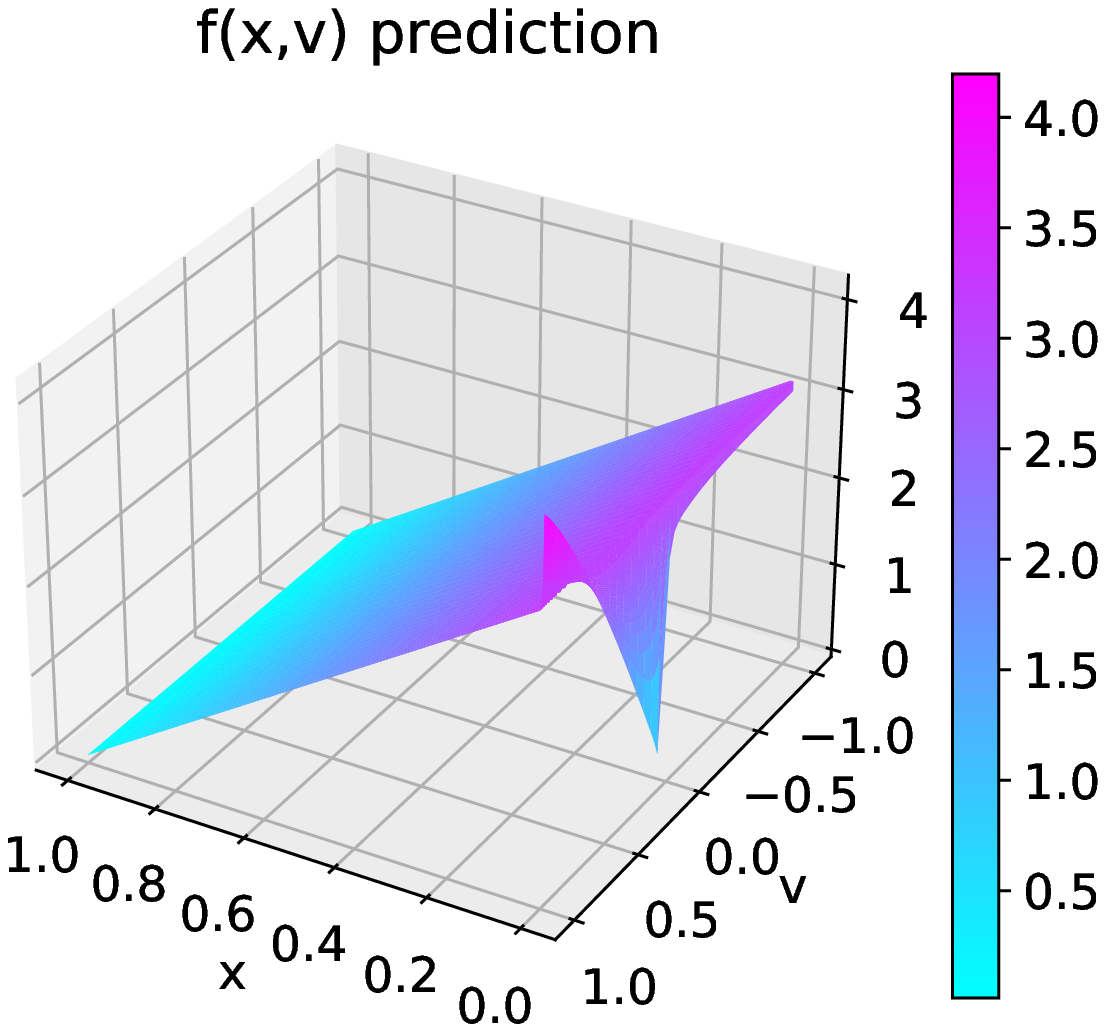}}
	{\includegraphics[width=0.35\textwidth]{rg_1d_bl_epsi_zpzz1_f_ref2.eps}}
	{\includegraphics[width=0.35\textwidth]{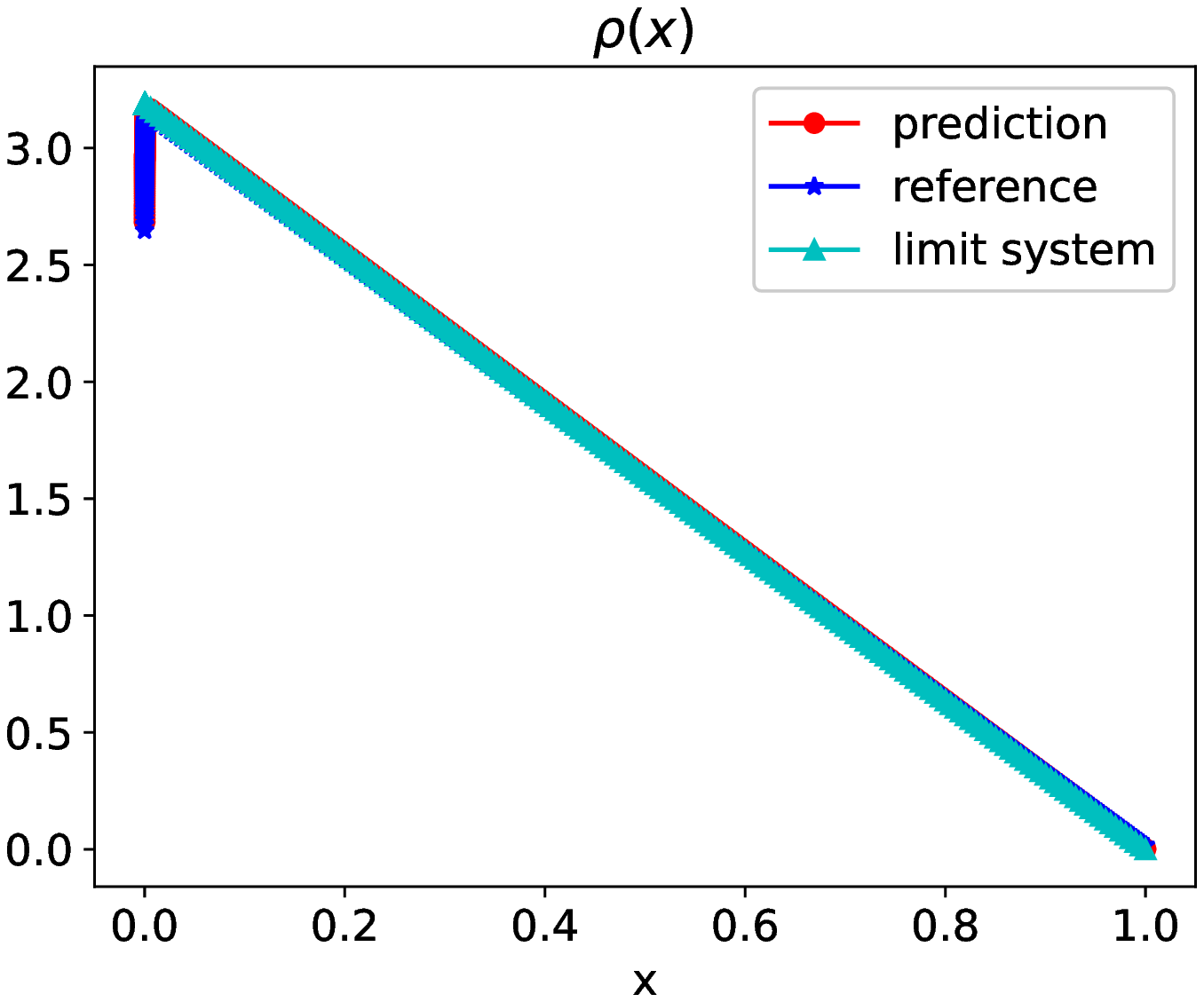}}
	{\includegraphics[width=0.35\textwidth]{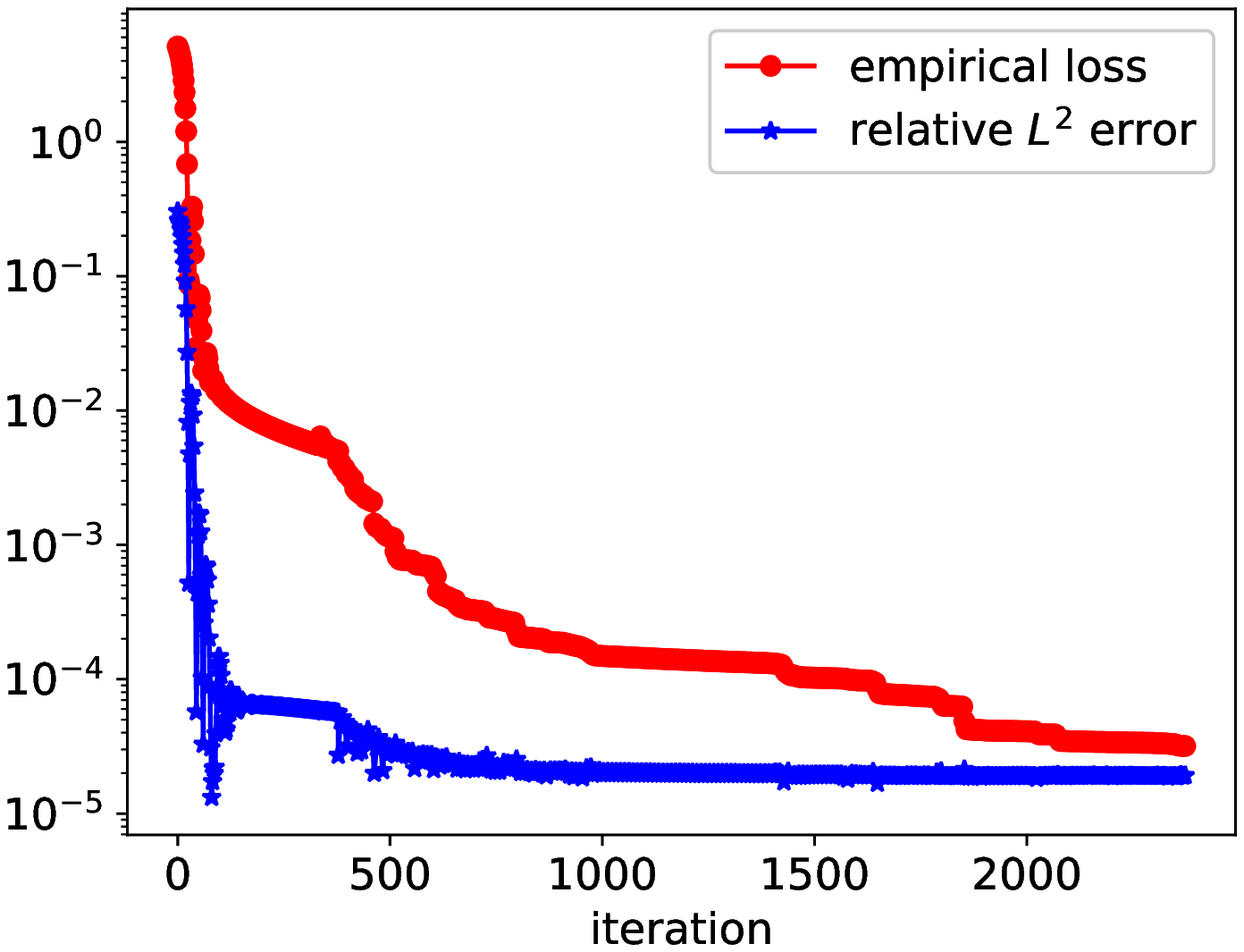}}
	\caption{Example~\ref{ex:transp} with $\eps = 10^{-3}$, we use $n_l = 4$, $n_r = 50$, $N^r_x = 80$, $N^r_v=60$ and $N^b_v = 60$ for training. The top left is $f(x,v)$ prediction, top right is the reference $f(x,v)$, bottom left is comparison of $\rho(x)$ and bottom right is the empirical loss and relative $L^2$ error vs iteration number. }
	\label{fig: hsp_1d_bl_epsizpzz1}
\end{figure}

\subsection{RTE in two dimensions with boundary layer} \label{2DBL}
As with \cref{sec:1D2}, we first solve the half space problem and then the corresponding RTE. 
\begin{example}\label{ex:2dhsp}
	2D half space problem: for $z\in [0,\infty)$ and $y\in[-1,1]$
	\begin{equation*} 
		\begin{cases}{}
			\cos \alpha \partial_z f_\bl(z, y, \alpha) = \average{f_\bl} - f_\bl \,, \quad \average{f_\bl} = \frac{1}{2\pi} \int_0^{2\pi} f_\bl \rd \alpha\,, \\
			f_\bl(0, y, \alpha) = (1-y^2) \alpha , \quad \cos \alpha<0\,.
		\end{cases}
	\end{equation*}
\end{example}
Then it admits a limit (see formula (7.3) in \cite{golse2003domain}):
\begin{equation} \label{10111}
	f_\bl^{\infty}(y):= \lim_{z\rightarrow \infty }f_\bl(z, y, \alpha) = \frac{1}{\sqrt{\pi}} \int_{\Gamma_-}  (1- y^2) \alpha  \cos \alpha H(\alpha) \rd \alpha\,,
\end{equation}
where $H$ is the Chandrasekhar H-function that satisfies 
\[
\frac{1}{H(\alpha)} = \int_{\Gamma_-}  \frac{H(\xi)}{\cos \alpha + \cos \xi} \cos \xi \rd \xi\,,
\]
and $\Gamma_- = [0, \pi/2] \cup [3 \pi /2, 2 \pi]$. 

Additionally, we can get the reflection boundary condition at $x=0$. Since the reflected velocity $\tilde{\bv}$ at boundary is
\[
\tilde{\bv} = \bv - 2(\bv \cdot \bn_{\bx}) \bn_{\bx}\,,\quad \bv = (\cos \alpha, \sin \alpha)\,,
\]
and in our case $\bn_{\bx} = (-1, 0)$, we have $\tilde \bv = (-\cos \alpha, \sin \alpha)$. Then 
\begin{equation} \label{10112}
	\begin{cases}{}
		f_\bl(0, y, \pi-\alpha) = \int_{\Gamma_-}  \xi (1-y^2) \cos \xi {H(\xi)H(\alpha)}/{(\cos \alpha + \cos \xi)} \rd \xi \,, \quad \alpha \in [0, \pi/2]\,,
		\\ f_\bl(0, y, 3\pi-\alpha) = \int_{\Gamma_-}  \xi (1-y^2) \cos \xi {H(\xi)H(\alpha)}/{(\cos \alpha + \cos \xi)} \rd \xi \,,\quad  \alpha \in [3\pi/2, 2\pi]\,.
	\end{cases}
\end{equation}
In Figure~\ref{fig: hsp_2d_H}, we plot the numerical prediction from the neural network approximation with parameters $n_l = 3$, $n_r = 50$, $N_x = 200$, $N_y = 50$ and $N_v=40$, and compared it with the reference solution \eqref{10111} and \eqref{10112}.

\begin{figure}[htbp]
	\centering
	{\includegraphics[width=0.3\textwidth]{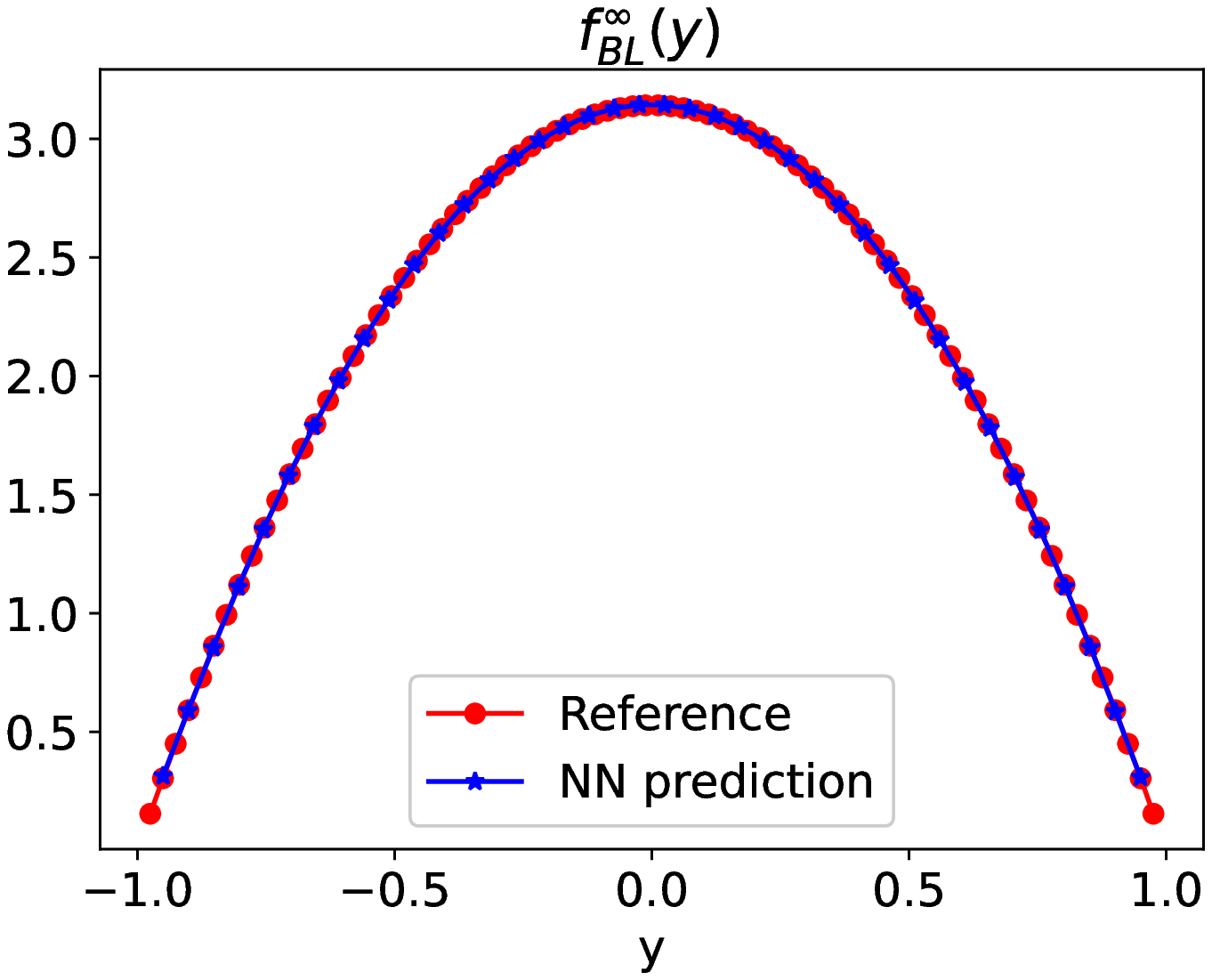}}
	{\includegraphics[width=0.3\textwidth]{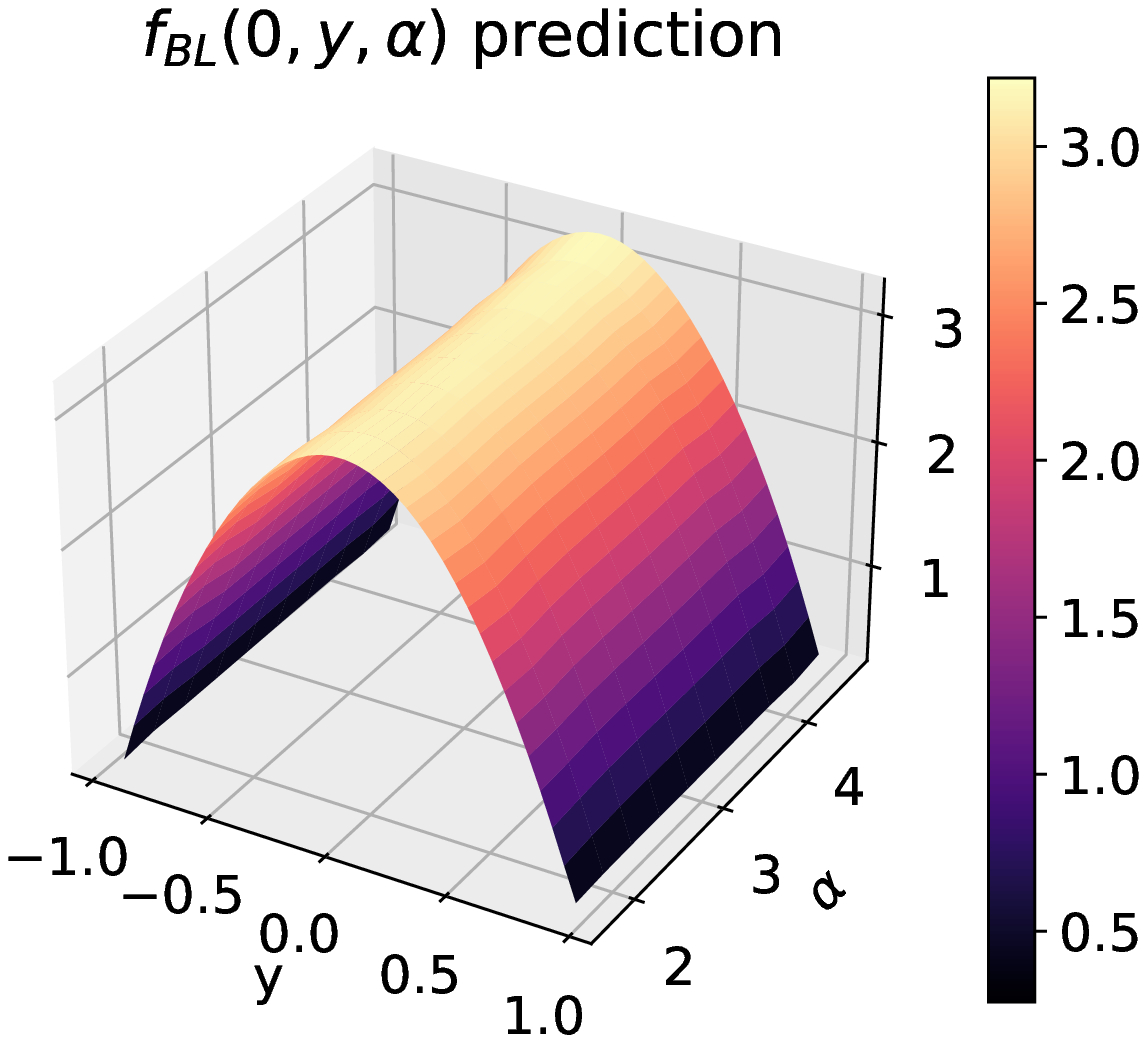}}
	{\includegraphics[width=0.3\textwidth]{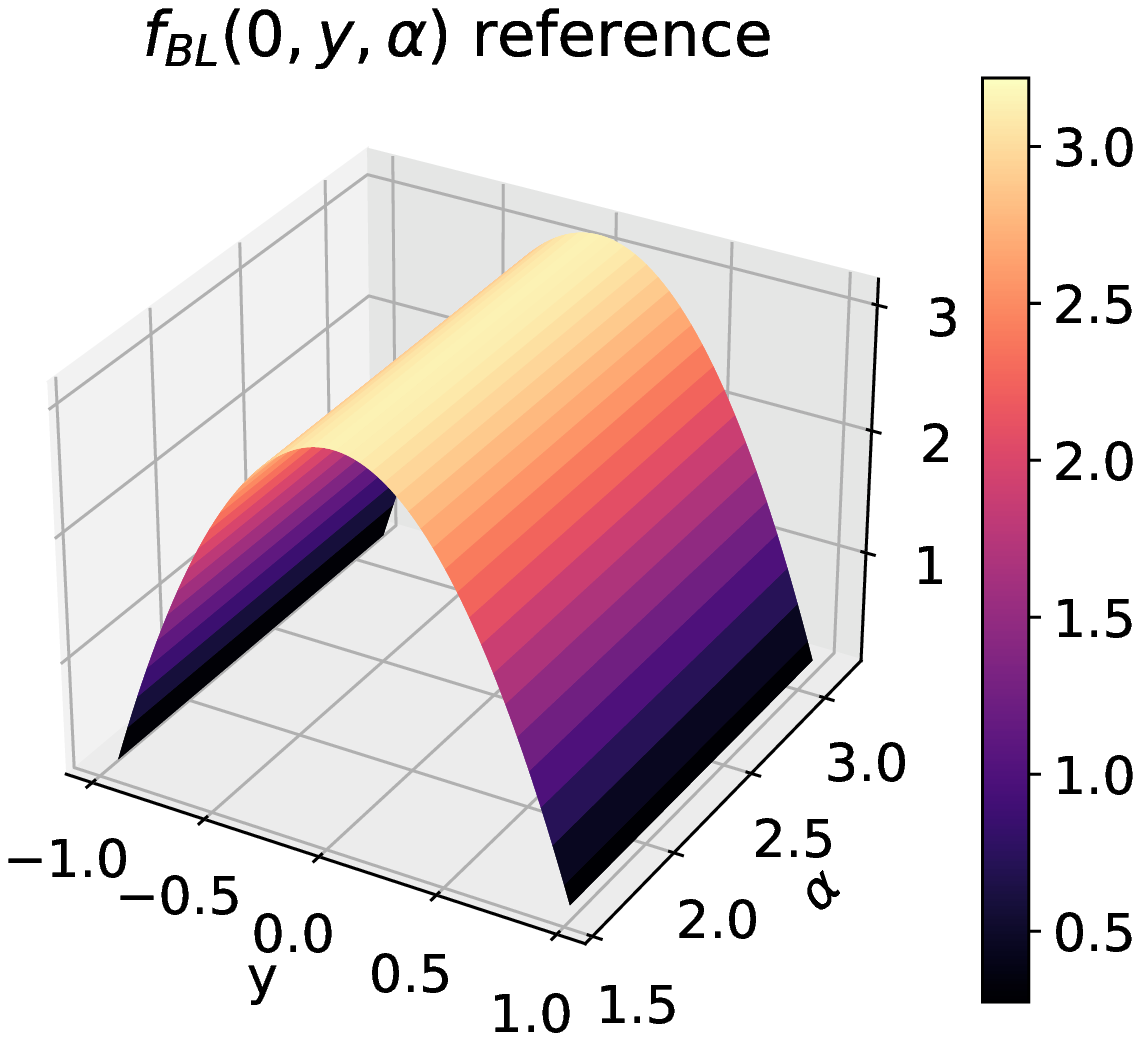}}
	\caption{Example~\ref{ex:2dhsp}. Left: plot of $f_\bl^\infty(y)$. Middle and right are plots of $f_\bl(0,y,\alpha)$ for $\alpha \in [0, 2\pi]$.The reference solution are obtained from the formula \eqref{10111} and \eqref{10112}.}
	\label{fig: hsp_2d_H}
\end{figure}

\begin{example}\label{ex:2dtransp}
	We then move on to solve the 2D transport equation, $\bx \in [-1,1]^2$, $\bv = (\cos \alpha, \sin \alpha)$, $\alpha \in [0, 2 \pi]$: 
	\begin{equation*} 
		\begin{cases}
			\eps \boldsymbol{v} \cdot \nabla_{\boldsymbol{x}} f = \frac{1}{2 \pi} \int_{|\boldsymbol{v}|=1} f(\bx, \bv) d \bv' - f \,, \\
			f(-1, y, \alpha) =  (1-y^2) \alpha, ~ \alpha \in [0, \pi/2] \cup [3 \pi/2, 2 \pi]  \,, \\
			f(1, y ,\alpha) = 0,  ~ \alpha \in [\pi/2, 3 \pi /2]  \,, \\
			f(x, -1 ,\alpha) = 0,  ~ \alpha \in [0, \pi ] \,, \\
			f(x, 1 ,\alpha) = 0,  ~ \alpha \in [\pi, 2 \pi ] .
		\end{cases}
	\end{equation*}
\end{example}

When $\eps = 1$, we use the macro-micro decomposition based PINN, and when $\eps=10^{-3}$, we include a boundary layer corrector which is computed in Example~\ref{ex:2dhsp}. In both cases, the numerical parameters are $n_l = 4$, $n_r = 30$, $N^r_x = 40$, $N^r_y=40$, $N^r_v=40$,  $N^b_v=40$, $N^b_x = 40$, $N^b_y = 40$ and $N^b_v=40$ for training. As a comparison, we use a finite difference method with uniform grid $N^r_x = 60$, $N^r_y = 60$, $N^r_v=60$, for $\eps =1$. For $\eps = 10^{-3}$, we solve the diffusion limit 
\begin{equation*} 
	\begin{cases}
		\Delta \rho = 0  \,, \\
		\rho(-1, y) =  \pi (1-y^2),\\
		\rho(1,y) = \rho(x,-1) = \rho(x, 1) = 0\,.
	\end{cases}
\end{equation*}
Here the boundary condition $\rho(-1,y)$ is obtained from \eqref{10111}. 
The numerical results are presented in Figure~\ref{fig: hsp_2d_bl_eps1} and \ref{fig: hsp_2d_bl}.

\begin{figure}[htbp]
	\centering
	{\includegraphics[width=0.3\textwidth]{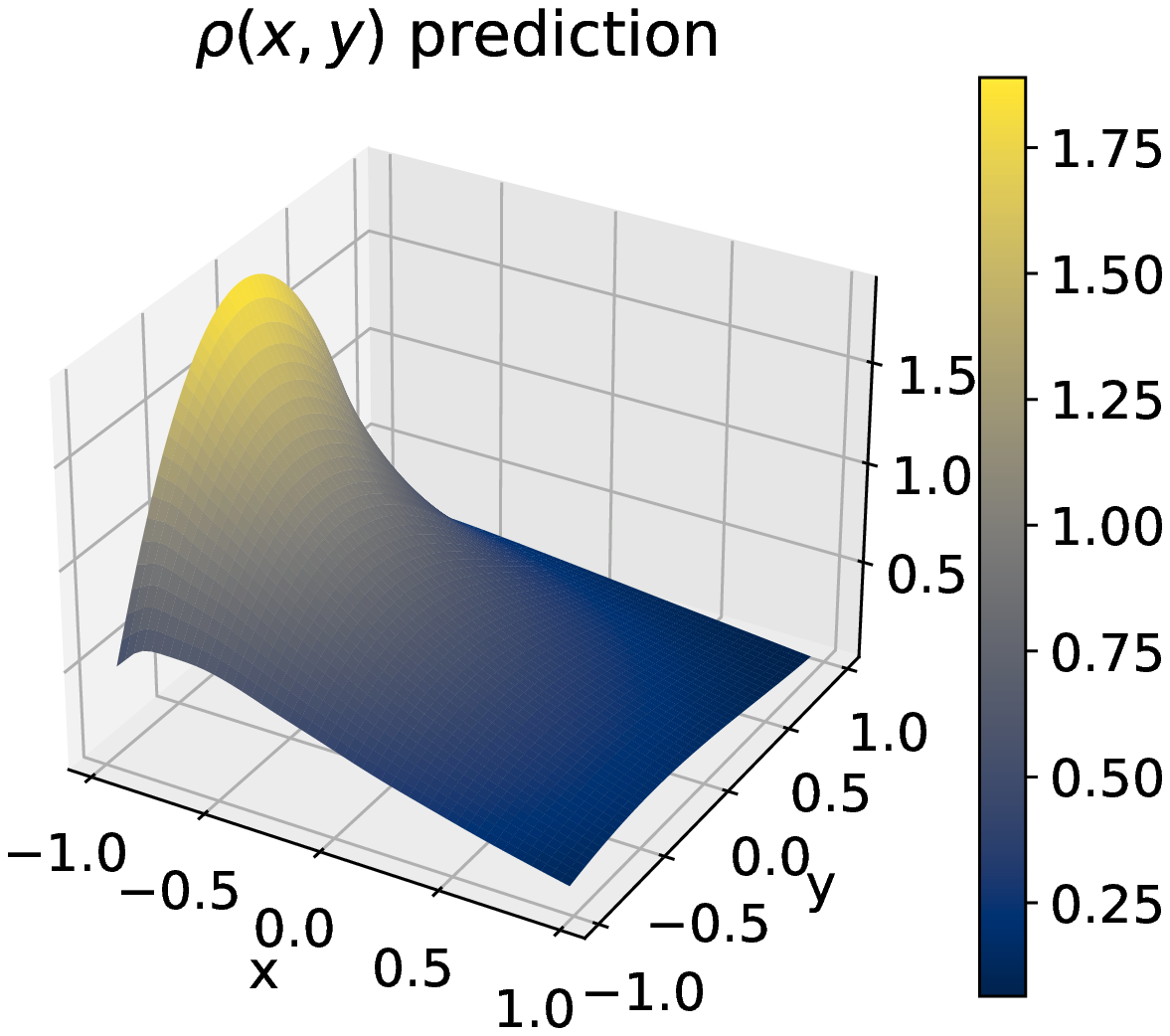}}
	{\includegraphics[width=0.3\textwidth]{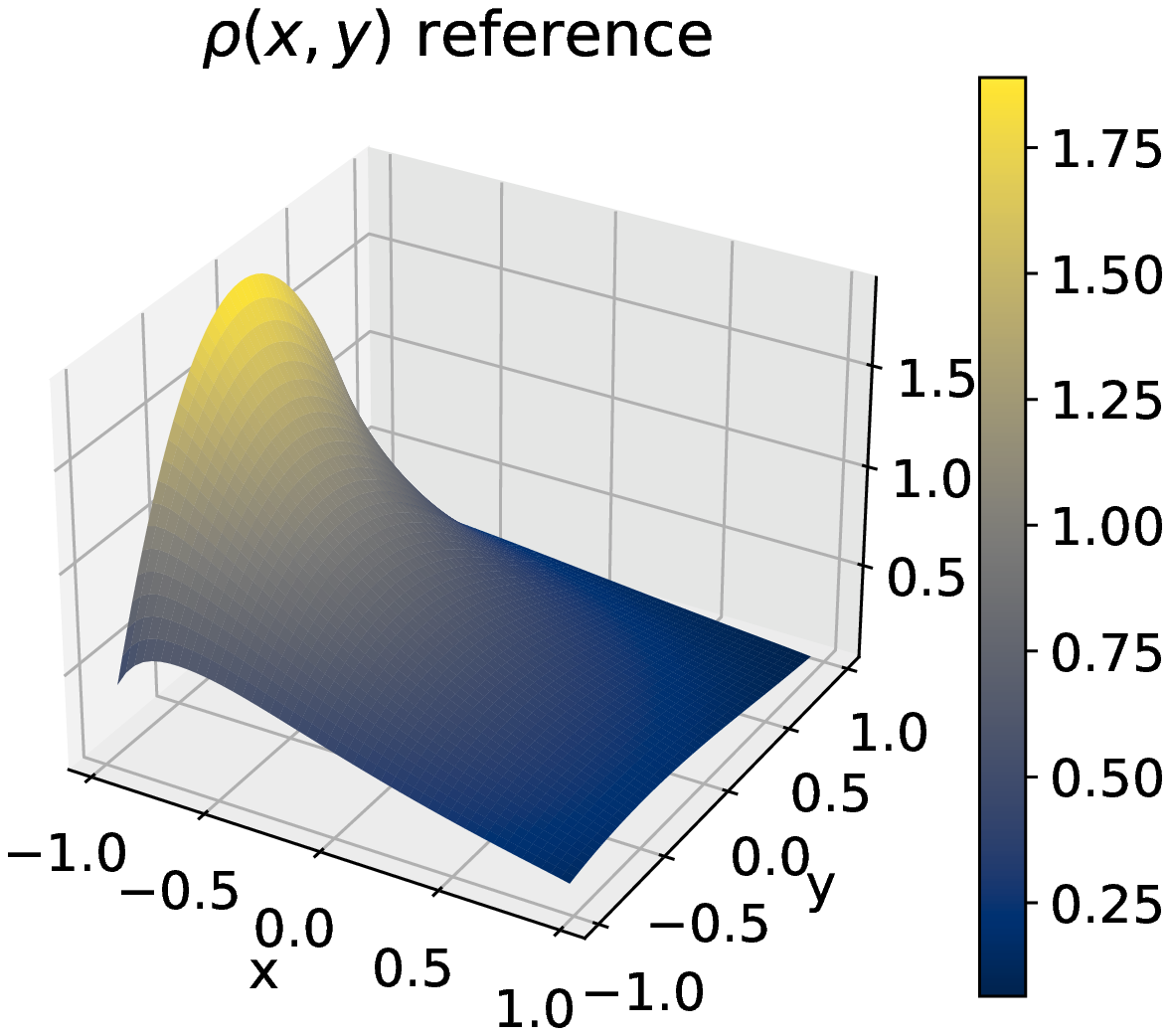}}
	{\includegraphics[width=0.3\textwidth]{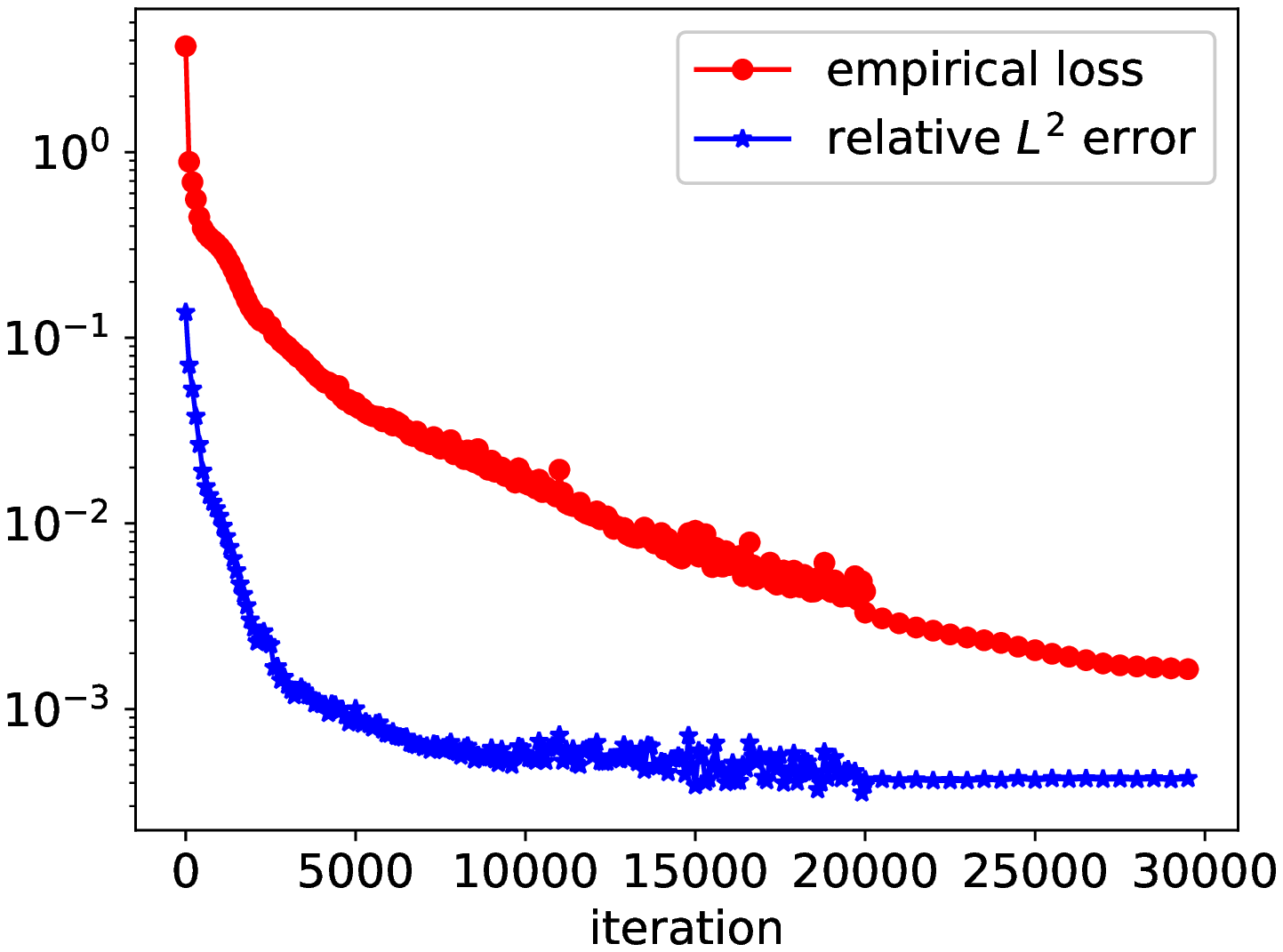}}
	\caption{Example~\ref{ex:2dtransp} with $\eps = 1$.}
	\label{fig: hsp_2d_bl_eps1}
\end{figure}

\begin{figure}[htbp]
	\centering
	{\includegraphics[width=0.45\textwidth]{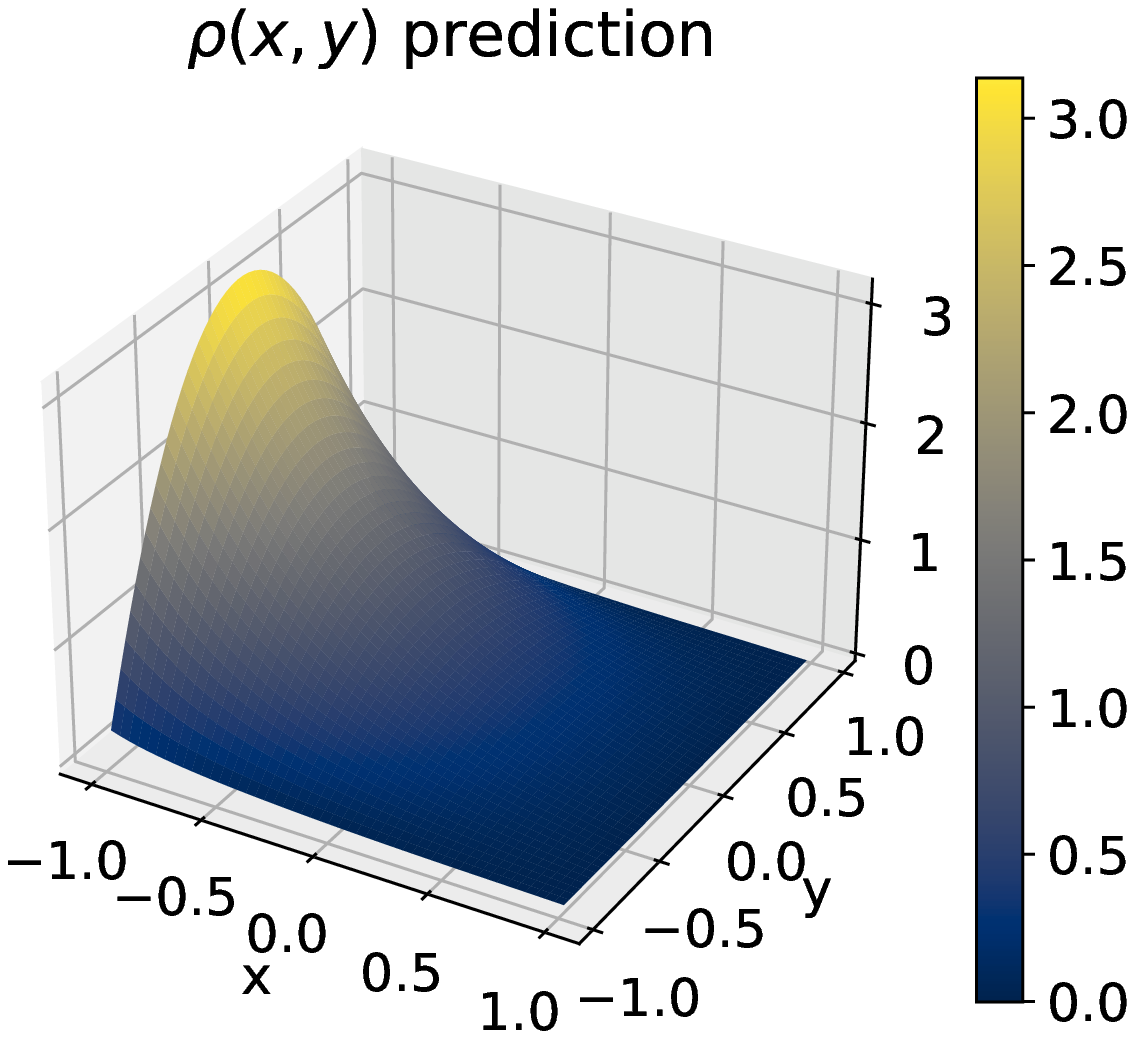}}
	{\includegraphics[width=0.45\textwidth]{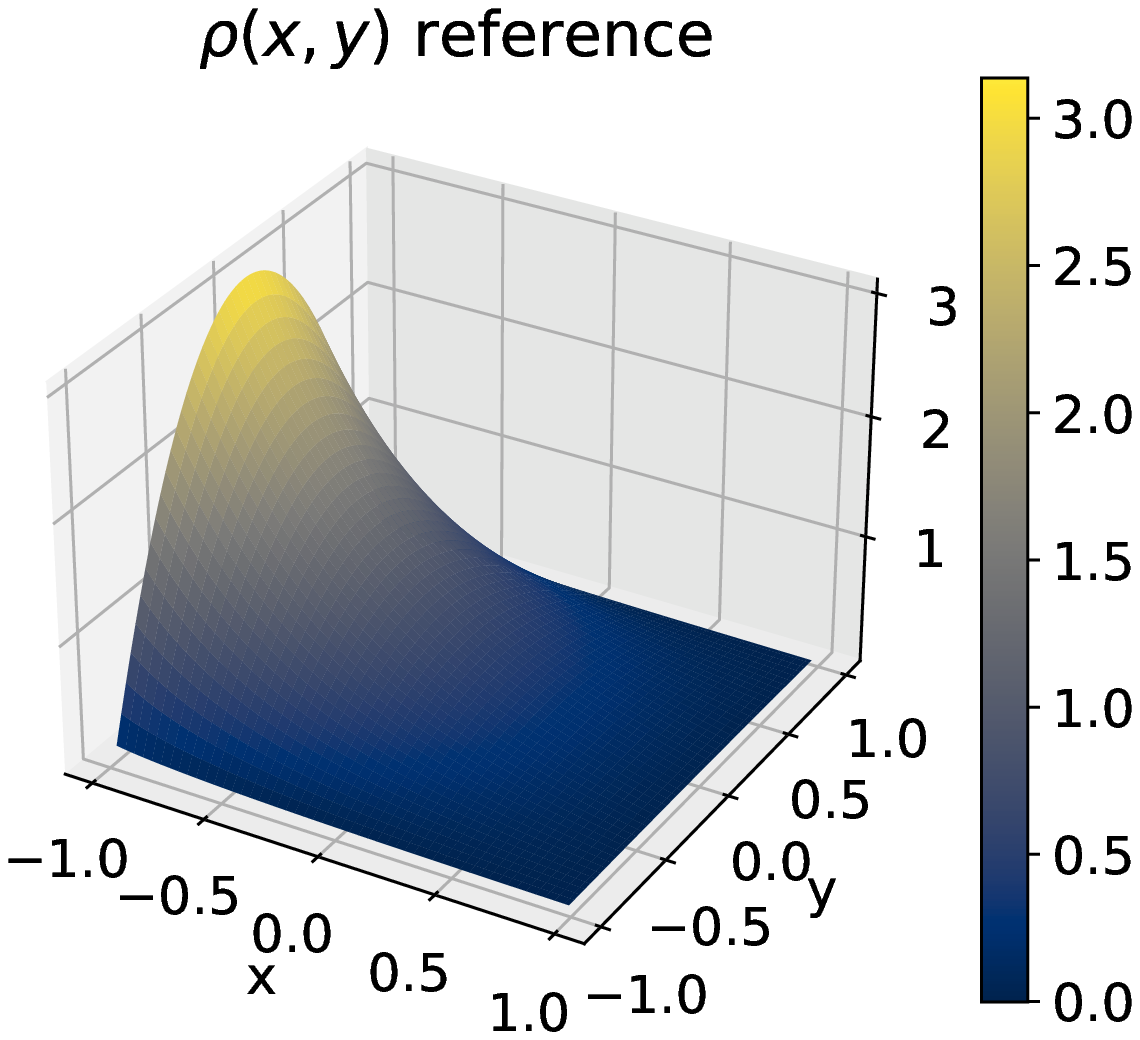}}
	\caption{Example~\ref{ex:2dtransp} with $\eps = 10^{-3}$. }
	\label{fig: hsp_2d_bl}
\end{figure}

\subsection{Nonlinear RTE}
At last, we consider an example of nonlinear RTE in one dimension.
\begin{example}\label{ex:nonl}
	For $x \in [0,1]$, $v \in [-1, 1]$ 
	\begin{equation} \label{eqn:NRTE}
		\begin{cases}{}
			\eps v \partial_x I(x,v) = \sigma (ac T^4(x) - I(x,v)), \\
			\eps^2 \partial_{xx} T(x) =  \sigma (ac T^4(x) -\langle I(x,v) \rangle ) \,,\\
			I(0,v>0) = 1, ~ I(1, v<0) = 0 \,, \\
			T(0) = 1, ~ T(1) = 0\,,
		\end{cases}
	\end{equation}
	where $a$, $c$ and $\sigma$ are three constants. 
\end{example}
For more details on nonlinear RTE, please refer to \cite{li2020asymptotic}. When $\eps \rightarrow 0$, $I$ and $T$ will converge to $I_0$ and $T_0$, which satisfy
\begin{equation}\label{eqn:T_limit}
	\begin{cases}
		\frac{ac}{3\sigma } \partial_{xx} T_0^4 +\partial_{xx} T_0 = 0 \,, \\
		T_0(0) = 1, ~ T_0(1) = 0 .
	\end{cases}
\end{equation}

To solve \eqref{eqn:NRTE}, we again conduct the macro-micro decomposition for $I$:
\[
I = \rho(x) + \eps g(x,v), \quad \text{where} ~ \rho(x) = \average{I}, ~ \average{g} = 0.
\]
Then the corresponding decomposed system reads:
\begin{equation*} 
	\begin{cases}{}
		\langle v \partial_x g \rangle = \partial_{xx} T, \, \\
		v \partial_x(\rho + \eps g) - \eps \partial_{xx} T = - \sigma g  , \, \\
		\eps^2 \partial_{xx} T = \sigma ac T^4  - \sigma \rho , \\
		\rho(0)  + \eps g(0, v>0)  = 1, \quad 
		\rho(1) + \eps g(1, v<0)  = 0, \\
		T(0) = 1, \quad T(1) = 0.
	\end{cases}
\end{equation*}
As a result, the loss function has the form: 
\begin{equation*} 
	\begin{aligned}
		& \mE(I, g, T)   = \|\average{v\partial_x g} \!-\! \partial_{xx} T \|_{L^2(\Omega_x)}^2 + \|v\partial_x (\rho + \eps g) \!-\! \eps \partial_{xx} T + \sigma g \|_{L^2(\Omega)}^2  + (T(0)-1)^2 + T(1)^2 \\
		&  + \| \eps^2 \partial_{xx} T \!-\! \sigma ac T^4 + \sigma \rho \|_{L^2(\Omega_x)}^2  
		+ \int_0^1(\rho (0)+\eps g(0,v) \!-\! 1)^2 \rd v 
		+ \int_{-1}^0(\rho (1)+\eps g(1,v))^2 \rd v .
	\end{aligned}
\end{equation*}
We then train the neural network with $n_l=4$ and $n_r = 50$, using $N^r_x = 80$, $N^r_v=60$ and $N^b_v = 60$ to generate training set. When $\eps=1$, we compute the reference solution by a finite difference method on a uniform grid with $N_x =200$ and $N_v=80$. When $\eps=10^{-3}$, we solve the limit system \eqref{eqn:T_limit} instead to get the reference solution. The results are collected in Figure~\ref{fig: rg_nonl_1} and \ref{fig: rg_nonl_zpzz1}, respectively.

\begin{figure}[!h]
	\centering
	{\includegraphics[width=0.45\textwidth]{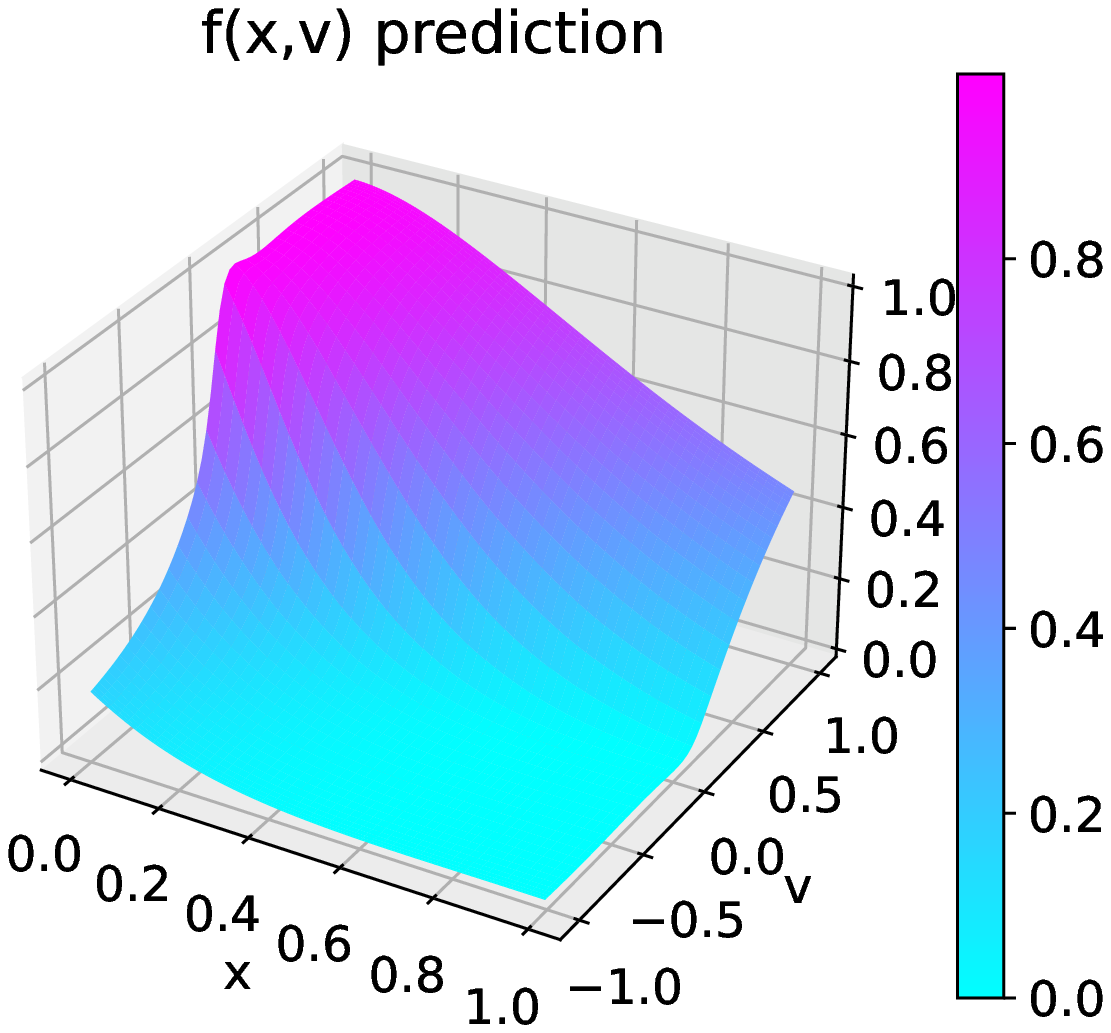}}
	{\includegraphics[width=0.45\textwidth]{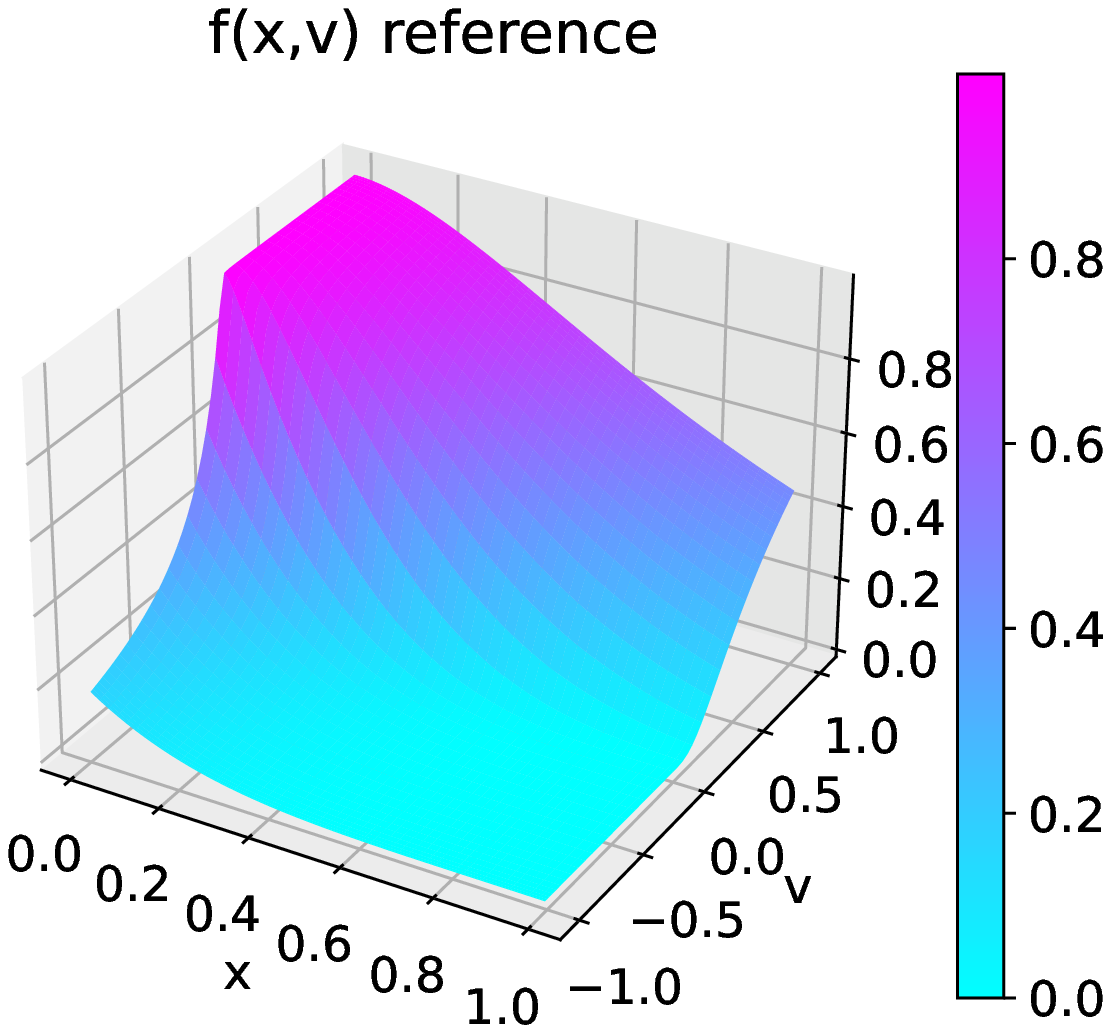}}
	{\includegraphics[width=0.45\textwidth]{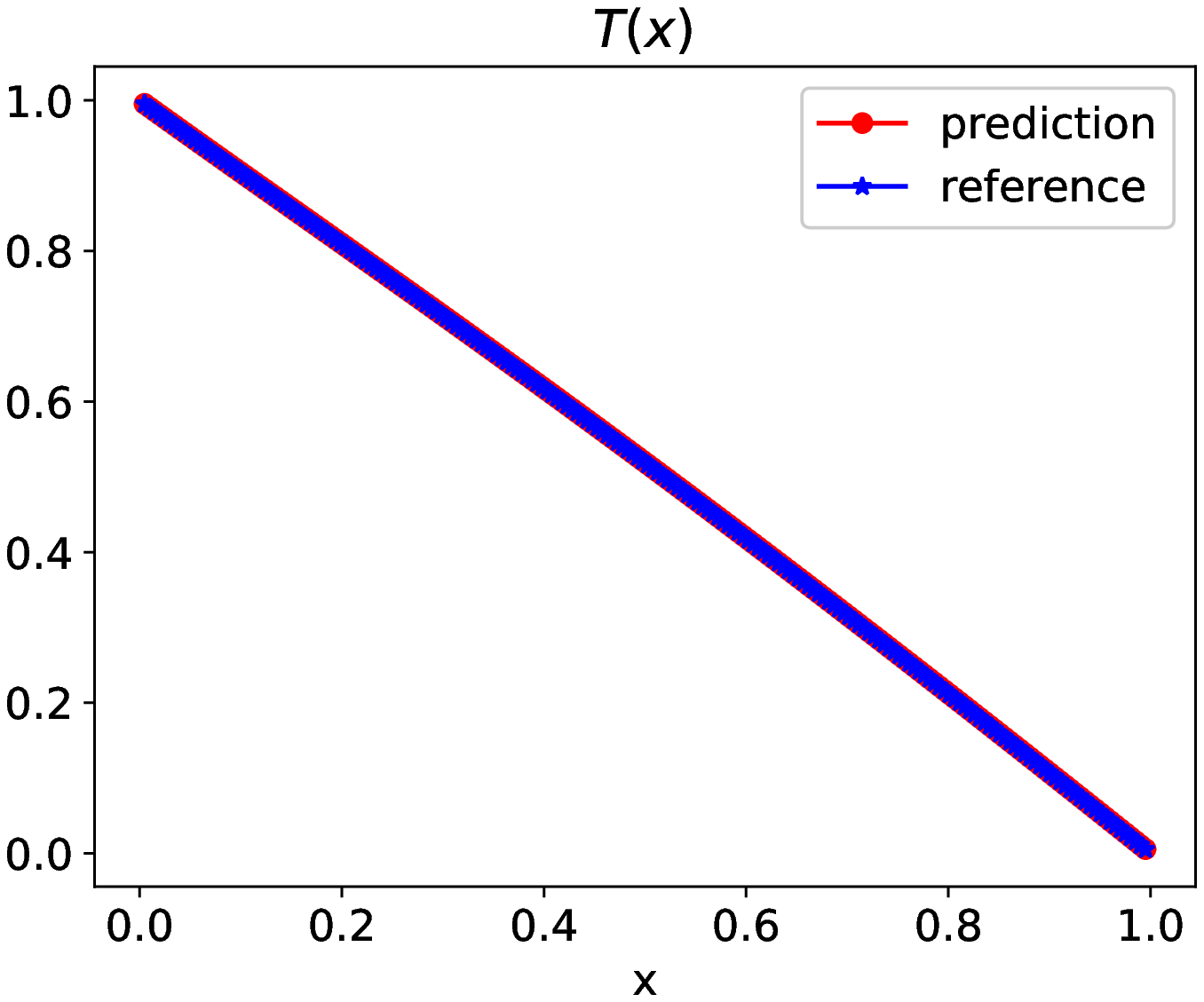}}
	{\includegraphics[width=0.45\textwidth]{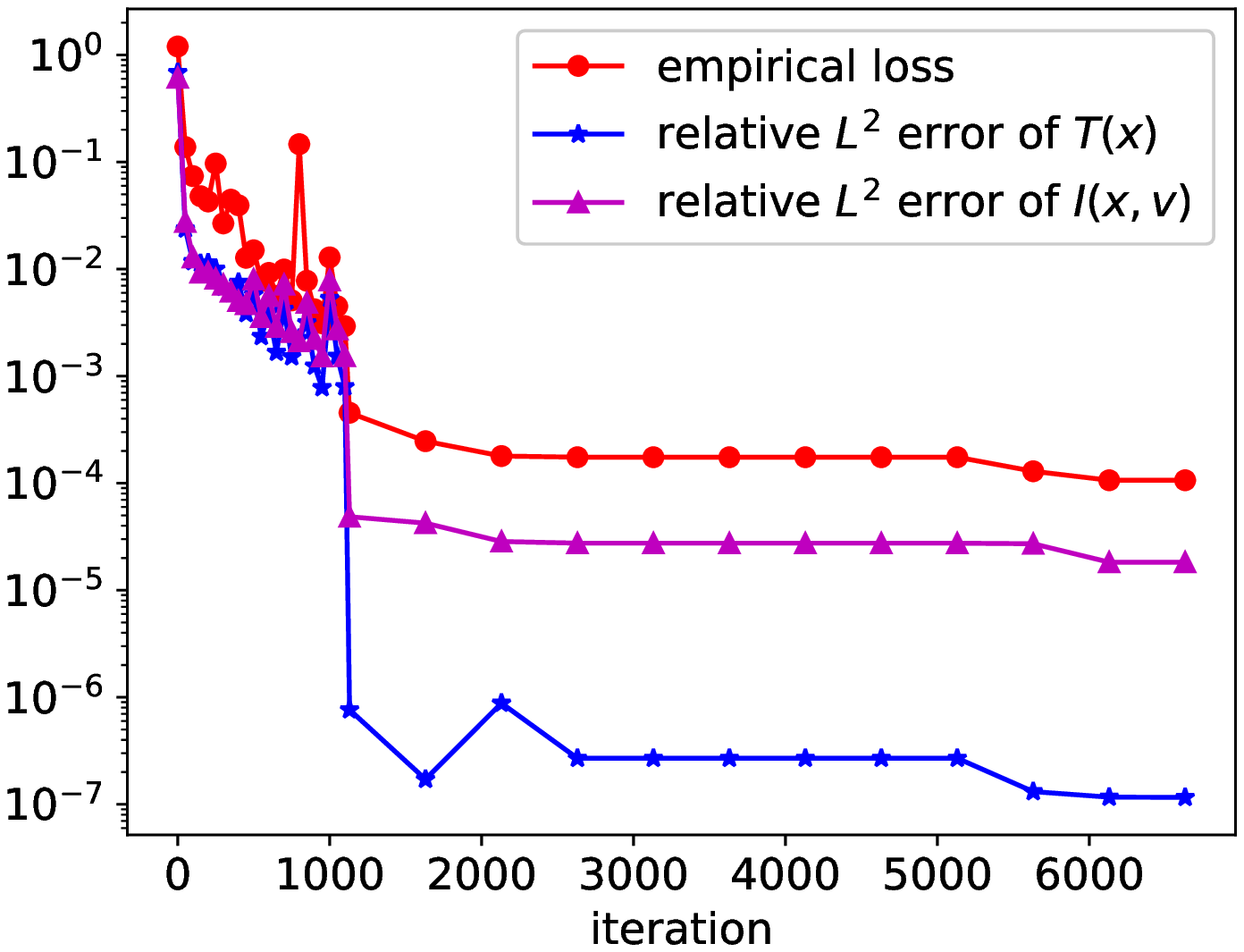}}
	\caption{Example~\ref{ex:nonl} with $\eps = 1$.
		Top left: $I(x,v)$ prediction. Top right: reference $I(x,v)$. Bottom left is comparison of $T(x)$. Bottom: the empirical loss and relative $L^2$ errors to reference solutions.}
	\label{fig: rg_nonl_1}
\end{figure}

\begin{figure}[!h]
	\centering
	{\includegraphics[width=0.45\textwidth]{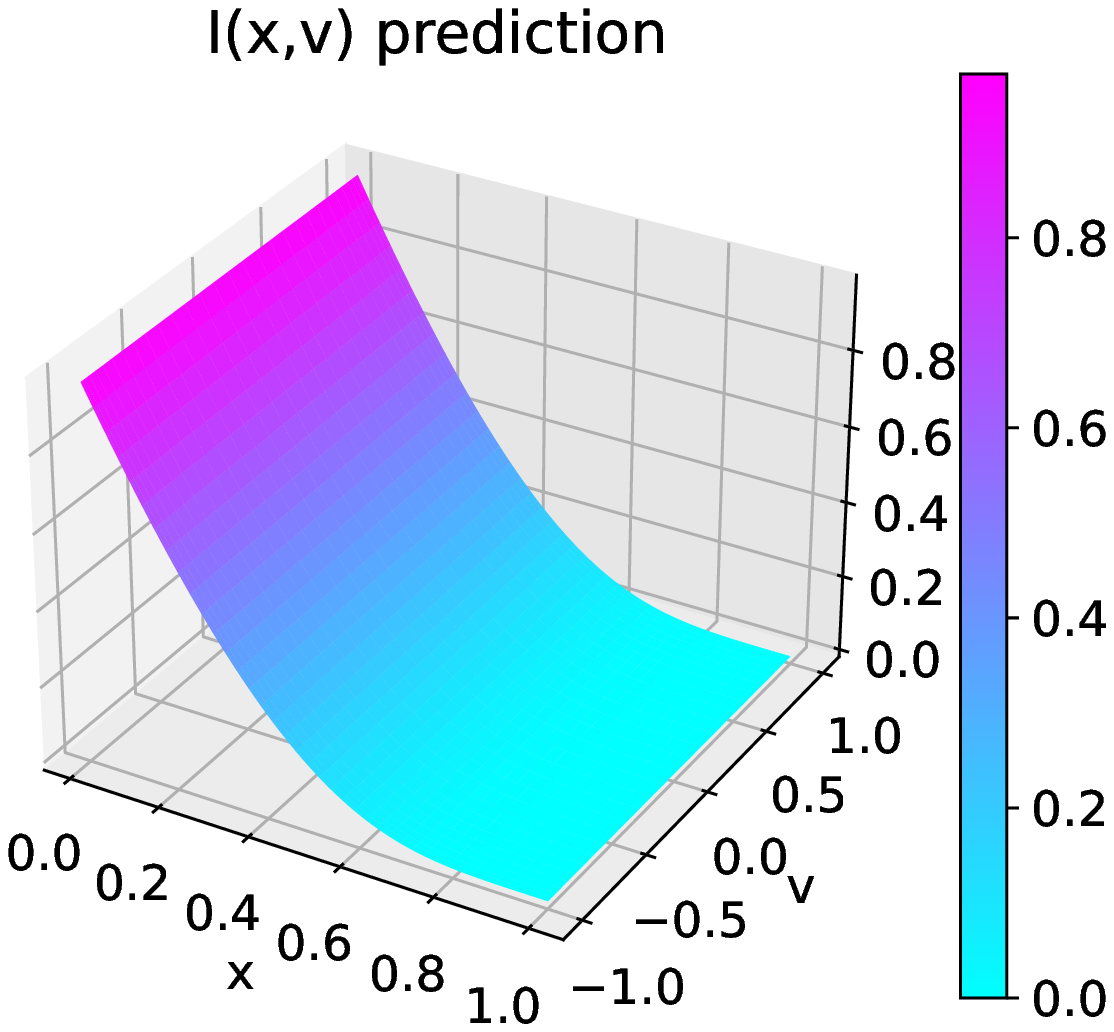}}
	{\includegraphics[width=0.45\textwidth]{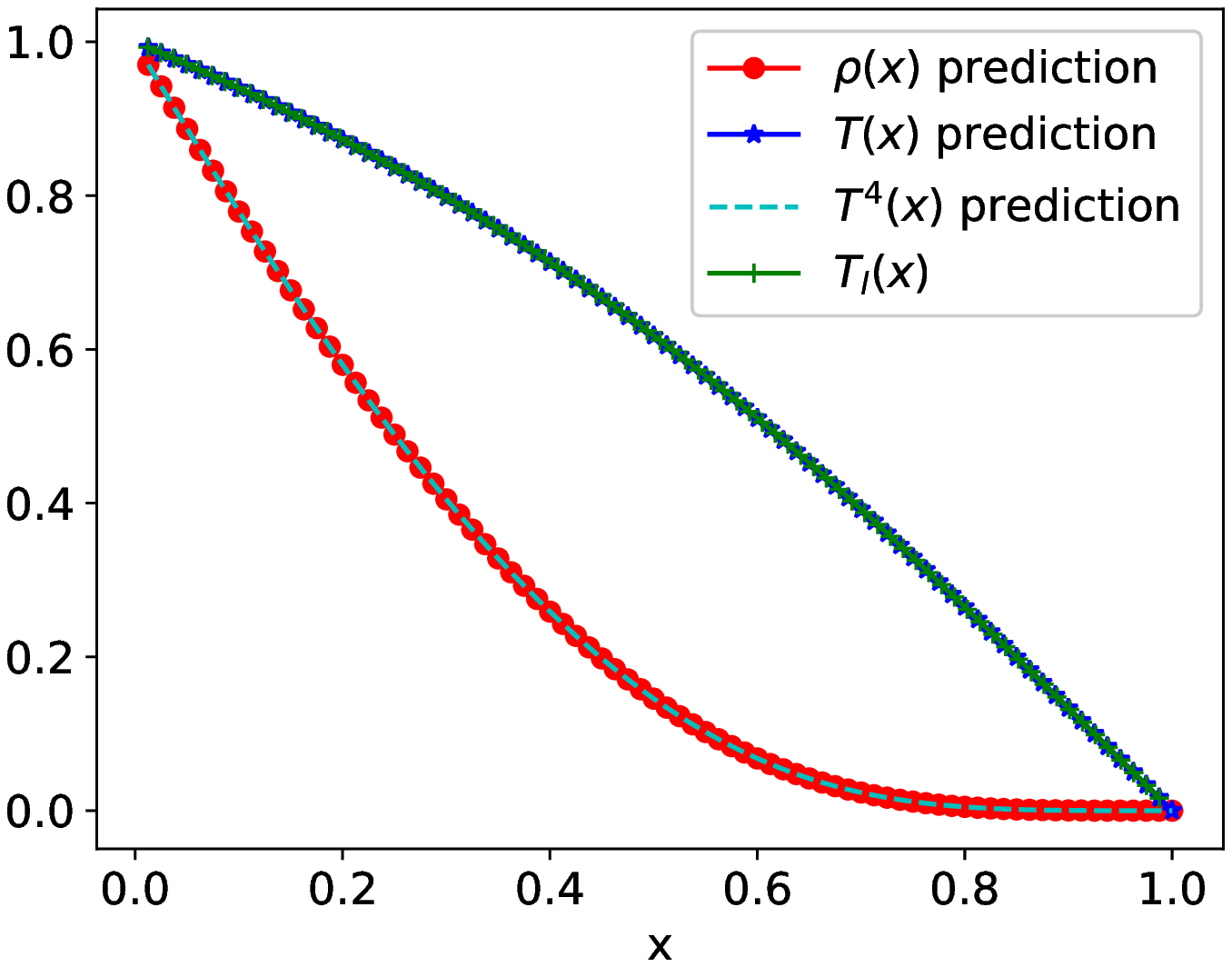}}
	\caption{Example~\ref{ex:nonl} with $\eps =10^{-3}$. 
	}
	\label{fig: rg_nonl_zpzz1}
\end{figure}

{
	\section{Conclusion}
	In this paper, we develop a numerical scheme based on PINNs for steady RTE with diffusive scaling. As illustrated in Section 3.1, vanilla PINNs suffer from the instability issue when $\eps$ is small, and our major contribution is to resolve this issue by proposing an novel empirical loss function based on the micro macro decomposition. More importantly, a rigorous uniform stability result is established. We prove that the $L^2$-error of the PINNs prediction can be bounded by the aforementioned new empirical loss function uniformly in $\eps$. When $\eps$ is small and an an-isotropic boundary condition is considered, a boundary layer is expected and the neural network is hence hard to converge. We construct a boundary layer corrector based on the solution of the associated half space problem, which encodes the sharp transition information within the boundary layer and leaves the rest part of solution smooth and thus can be easily approximated.  Extensive numerical results  demonstrate the effectiveness of our novel numerical methods.

}
\section*{Acknowledgements}
Y.L. thanks the US National Science Foundation for its support through the award  DMS-2107934. L.W. and W.X. thank the National Science foundation for its support through the award DMS-1846854. The authors also acknowledge the Minnesota Supercomputing Institute (MSI) at the University of Minnesota for providing resources that contributed to the research results reported within this paper. 

\newpage

\bibliographystyle{siamplain}
\bibliography{references}

\end{document}